\documentclass[a4paper, 11pt]{amsart}

\usepackage{amsmath,amssymb, amscd, color, hyperref}

\newcommand{\bdot}{\boldsymbol{\cdot}}

\theoremstyle{plain}
\newtheorem{theorem}{\bf Theorem}[section]
\newtheorem{proposition}[theorem]{\bf Proposition}
\newtheorem{lemma}[theorem]{\bf Lemma}
\newtheorem{corollary}[theorem]{\bf Corollary}

\theoremstyle{definition}
\newtheorem{example}[theorem]{\bf Example}

\newtheorem{remark}[theorem]{\bf Remark}

\newcommand{\N}{\mathbb N}
\newcommand{\Z}{\mathbb Z}
\newcommand{\R}{\mathbb R}

\newcommand{\BF}{\text{\rm BF}}
\newcommand{\FF}{\text{\rm FF}}

\newcommand{\AAP}{\text{\rm AAP}}
\newcommand{\AAMP}{\text{\rm AAMP}}

 \DeclareMathOperator{\ord}{ord}
\DeclareMathOperator{\lcm}{lcm} 
 
\DeclareMathOperator{\spec}{spec} \DeclareMathOperator{\supp}{supp}

\renewcommand{\t}{\, | \,}

\numberwithin{equation}{section}

\usepackage{ stmaryrd }

\newcommand{\be}{\begin{equation}}
\newcommand{\ee}{\end{equation}}

\numberwithin{equation}{section}

\thanks{}


\begin{document}
\title{On product-one sequences over subsets of groups}

\author{Victor Fadinger and Qinghai Zhong}

\address{School of Mathematics and statistics, Shandong University of Technology, Zibo, Shandong 255000, China.}
\email{qinghai.zhong@uni-graz.at}
\urladdr{https://imsc.uni-graz.at/zhong/}

\address{University of Graz, NAWI Graz \\
Institute for Mathematics and Scientific Computing \\
Heinrichstra{\ss}e 36\\
8010 Graz, Austria}
\email{victor.fadinger@uni-graz.at,  qinghai.zhong@uni-graz.at}

\thanks{This work was supported by the Austrian Science Fund FWF (Projects W1230 and P33499-N) and by National Natural Science Foundation of China (Grant No. 12001331).}


\begin{abstract}
Let $G$ be a group and $G_0 \subseteq G$ be a subset. A sequence over $G_0$ means a finite sequence of terms from $G_0$, where the order of elements is disregarded and the repetition of elements is allowed. A product-one sequence is a sequence whose elements can be ordered such that their product equals the identity element of the group. We study algebraic and arithmetic properties of monoids of product-one sequences over finite subsets of $G$ and over the whole group $G$, with a special emphasis on the infinite dihedral group.
\end{abstract}

\maketitle

\section{Introduction} \label{1}

Let $G$ be a group and $G_0 \subseteq G$ be a subset. Elements of the free abelian monoid over $G_0$ are called sequences over $G_0$. Thus, in combinatorial language, a sequence over $G_0$ means a finite sequence of terms from $G_0$, where the order of elements is disregarded and the repetition of elements is allowed, and the concatenation of sequences corresponds to the multiplication in the free abelian monoid.   A product-one sequence is a sequence whose elements can be ordered such that their product equals the identity element of the group. Thus, the set $\mathcal B (G_0)$ of product-one sequences over $G_0$ is a submonoid of the free abelian monoid over $G_0$.

The study of sequences, in particular the study of their structure under extremal properties, is a classical  topic in additive combinatorics. The monoid of product-one sequences arises naturally in various subfields of algebra, from abstract semigroup theory, to factorization theory, and to invariant theory. For a long time the focus of interest was in the abelian setting, but the last decade has seen a growing interest for the non-abelian setting. There is work on combinatorial invariants (including small and large Davenport constants and Erd{\H{o}}s-Ginzburg-Ziv constants \cite{G-G-O-Z21a, Oh-Zh20a, Oh-Zh20b, Br-Ri18a,Zha21a, Gr13b, Ha15a, Ha-Zh19a}), on problems stemming from invariant theory (including work on Noether numbers \cite{Cz-Do-Sz18,Cz-Do14a,Cz-Do15a, Cz-Do-Ge16a}), and on algebraic properties of the monoid $\mathcal B (G_0)$ (\cite{Oh20a, Oh19b, Do20a}).

All work in the non-commutative setting was restricted so far to finite groups. In the present paper we study sequences over  finite subsets of arbitrary groups (note that there are infinite torsion groups having finite sets of generators). In Section \ref{3}, we provide an in-depth study of algebraic properties of the monoid $\mathcal B (G_0)$ for subsets of arbitrary groups. Under certain conditions, we characterize when $\mathcal B (G_0)$ is Krull, root closed, finitely generated, or a C-monoid (Theorem \ref{3.11}). For the whole group $G$, we establish characterizations for being weakly Krull, or seminormal, and more (Theorem \ref{3.14}). In Section \ref{4}, we investigate arithmetical invariants of monoids of product-one sequences, where our focus is on sets of lengths. Based on these results for general groups, Section \ref{5} offers explicit algebraic characterizations for the finiteness of key arithmetical invariants, including the catenary degree, the $k$th elasticities, local tame degrees, and more  in the special case of finite subsets of infinite dihedral groups (Theorem \ref{5.1}) and for the whole group (Theorem \ref{5.2}).

\section{Preliminaries} \label{2}

We denote by $\N$ the set of positive integers and we put $\N_0 = \N\cup\{0\}$. For real numbers $a, b \in \R$, we denote by $[a, b] = \{ x \in \Z \colon a \le x \le b\}$ the discrete interval between $a$ and $b$.

\smallskip
\noindent
\textbf{Monoids.} A \textit{monoid} means a commutative, unitary and cancellative semigroup. Our notation and terminology for ideal theory and the arithmetic of monoids are  consistent with  \cite{HK98, Ge-HK06a}. We briefly gather some key notions.
Let $H$ be a monoid and $\mathsf q(H)$ its quotient group.
 There are three common closure operations for a monoid $H$, namely
\begin{itemize}
\item $H'=\{x\in\mathsf q(H)\colon \text{there is $n\in\N$ such that $x^m\in H$ for all $m\geq n$} \}$ is called the \textit{seminormalization} of $H$,
\item $\widetilde{H}=\{x\in\mathsf q(H)\colon \text{there is $n\in\N$ such that $x^n\in H$}\}$ is called the \textit{root closure} of $H$ and
\item $\widehat{H}=\{x\in\mathsf q(H)\colon \text{there is $c\in H$ such that $cx^n\in H$ for all $n\in\N$}\}$ is called the \textit{complete integral closure} of $H$.
\end{itemize}
A monoid $H$ is called \textit{seminormal} (resp. \textit{root closed}, resp. \textit{completely integrally closed}) if $H=H'$ (resp. $\widetilde{H}$, resp. $\widehat{H}$).\\
Note the following two facts:
\begin{enumerate}
\item $H\subseteq H'\subseteq \widetilde H\subseteq \widehat H$ (the first two inclusions being clear, the third follows from \cite[Proposition 2.7.11]{Ge-HK06a}).
\item $H$ is seminormal if and only if for every $x\in \mathsf{q}(H)$ we have that $x^2,x^3\in H$ implies $x\in H$ (e.g. see \cite[Lemma 2.4]{Re13a}).
\end{enumerate}

%

A monoid homomorphism $\varphi: H\to D$ is 
\begin{itemize}
	\item a \textit{divisor homomorphism} if for all $a,b\in H$ we have that $\varphi(a)\mid_D \varphi(b)$ implies $a\mid_H b$;
	\item  \textit{cofinal} if for all $a\in D$, there exists $b\in H$ such that $a\mid_D \varphi(b)$.
\end{itemize}		
	 A submonoid $H\subseteq D$ is called \textit{saturated} if the inclusion $H\hookrightarrow D$ is a divisor homomorphism (equivalently, $\mathsf q(H)\cap D=H$) and it is called \textit{divisor closed} if for all $d\in D$ and $h\in H$ we have that $d \mid_D h$ implies $d\in H$. A monoid $H$ is said to have a \textit{divisor theory} if there exists a divisor homomorphism $\varphi:H\to \mathcal F(P)$ into a free abelian monoid over some set $P$ such that, for every $p\in P$, there is a finite nonempty subset $X\subseteq H$ with $p=\gcd(\varphi(X))$.

Let $X\subseteq \mathsf q(H)$. We set $(H\colon X)=\{y\in \mathsf q(H)\colon yX\subseteq H\}$ and $X_v=(H\colon (H\colon X))$.
 We say 
$X$ is
\begin{itemize}		
	\item an $s$-ideal if $X\subseteq H$ and $XH=X$;

	\item a $v$-ideal if $X\subseteq H$ and $X=X_v$.

	\item a $t$-ideal if $X\subseteq H$ and $$X=\bigcup_{E\subset X \text{  nonempty and finite}} E_v\,.$$
\end{itemize} 
We denote by $s$-$\operatorname{spec}(H)$ the set of all nonempty prime $s$-ideals of $H$ and by $\mathfrak X(H)$ the set of all nonempty minimal prime $s$-ideals of $H$. 

A monoid is \textit{$v$-noetherian} if it satisfies the ascending chain condition for $v$-ideals
and 
a monoid is a \textit{Krull monoid} if it is completely integrally closed and $v$-noetherian (equivalently, if it has a divisor theory).  A list of further equivalent conditions can be found in \cite[Theorems 2.3.11 and  2.4.8]{Ge-HK06a}. If $H$ is a Krull monoid and $F=\mathcal F(P)$ is a free abelian monoid such that $H_{\text{red}}\subseteq F$ and the inclusion $H_{\text{red}}\hookrightarrow F$ is a divisor theory (such an $F$ always exists), then $F$ is called a \textit{monoid of divisors} and $P$ a \textit{set of prime divisors} for $H$. In the given case, we call
$\mathcal C(H)=F/H_{\text{red}}$
the \textit{divisor class group} (or just \textit{class group}) of $H$.
Let $\mathfrak p$ be a prime $s$-ideal. We denote by $H_{\mathfrak p}=\{s_1/s_2\in \mathsf q(H)\colon s_1\in H \text{ and }s_2\in H\setminus \mathfrak p\}$ the localization of $H$ at $\mathfrak p$.
A monoid $H$ is called a {\it weakly Krull monoid} (\cite[Corollary 22.5]{HK98}) if
\[
H = \bigcap_{{\mathfrak p} \in \mathfrak X (H)} H_{\mathfrak p}  \quad \text{and} \quad \{{\mathfrak p} \in \mathfrak X (H) \colon a \in {\mathfrak p}\} \quad \text{is finite for all} \ a \in H \,.
\]

Let $D$ be a monoid and $H\subseteq D$ a submonoid. Two elements $y,y'\in D$ are called \textit{$H$-equivalent} if $y^{-1}H\cap D=y'^{-1}H\cap D$ (equivalently: If $x\in D$, then $yx\in H$ if and only if $y'x\in H$). This is a congruence relation on $D$ and for $y\in D$ we denote its congruence class by $[y]_H^D$. We set
\begin{equation*}
\mathcal C(H,D)=\{[y]_H^D\colon y\in D\} \ \ \ \text{ and } \ \ \ \mathcal C^*(H,D)=\{[y]_H^D\colon y\in D\setminus D^{\times}\cup \{1_D\}\},
\end{equation*}
the first being called the \textit{class semigroup} of $H\subseteq D$ and the latter the \textit{reduced class semigroup} of $H\subseteq D$.
A monoid $H$ is called a \textit{C-monoid} if it is a submonoid of a factorial monoid $F$ such that $H^{\times}=H\cap F^{\times}$ and the reduced class semigroup $\mathcal C^*(H,F)$ is finite. To emphasize the fact that $H$ is a C-monoid as a submonoid of the factorial monoid $F$, we also say that $H$ is a C-monoid defined in $F$. Every Krull monoid $H$ with finite divisor class group is a C-monoid in $H^{\times}\times F$, where $F$ is a monoid of divisors for $H$. We refer to  (\cite[Chapter 2.8 and 2.9]{Ge-HK06a}, \cite{Re13a, Ge-Ra-Re15c})  for background on C-monoids and for further examples.

\smallskip
\noindent
\textbf{Monoids of product-one sequences.} Let $G$ be a multiplicatively written group and $G_0\subseteq G$ be a subset. The submonoid generated by $G_0$ will be denoted by $[G_0]$, while the notation for the subgroup generated by $G_0$ is $\langle G_0\rangle$. If $g,h\in G$, then the commutator of $g$ and $h$ is the element $[g,h]=g^{-1}h^{-1}gh\in G$. The commutator subgroup of $G$ is $G'=\langle [g,h]\colon g,h\in G\rangle$.

Elements of the free abelian monoid over $G_0$ will be called {\it sequences} (over $G_0$). Our notation on sequences  is consistent with \cite{Gr13b, Cz-Do-Ge16a}. We recall and fix notation for the key objects needed in the sequel.
To begin with, let $\mathcal F(G_0)$ denote  the free abelian monoid over  $G_0$. Then every element $S\in\mathcal F(G_0)$ has a unique  representation of the form
\begin{equation*}
S=\prod_{i=1}^n g_i^{[m_i]},
\end{equation*}
where $n \in \N_0$, $m_1,\hdots, m_n\in \N$, and $g_1,\hdots,g_n\in G_0$.  In order to distinguish the group multiplication from the operation on $\mathcal F(G_0)$, we denote the first by $"\cdot "$ (and avoid it when possible) and the latter by $"\bdot "$. For the same reason, if $g\in G_0$, the sequence $g\bdot\hdots\bdot g$ of length $n$ is denoted $g^{[n]}$, whence $g^{[n]}\in \mathcal F(G_0)$ and $g^n\in G$. Let $S=g_1\bdot\hdots\bdot g_{\ell} \in\mathcal F(G_0)$ be a sequence. Then $|S|= \ell \in \N_0$ is the {\it length} of $S$, $\supp (S) = \{g_1, \ldots, g_{\ell} \}$ is the {\it support} of $S$, and
\begin{equation*}
\pi(S)=\{g_{\sigma(1)}\cdot \ldots\cdot g_{\sigma(n)}\in G \colon \sigma\in \mathfrak S_n \text{ is a permutation of }  [1,n]\}
\end{equation*}
is the \textit{product set} of $S$. If $T,S\in\mathcal F(G_0)$, then $T$ \textit{divides} $S$, denoted by $T\mid S$, provided there exists $U\in\mathcal F(G_0)$ such that $TU=S$. In this case $T$ is called a \textit{subsequence} of $S$.\\
We say that  $S$ is
\begin{itemize}
\item a \textit{product-one sequence} if $1_G\in \pi(S)$,

\item{\it product-one free} if $1_G \notin \pi (T)$ for any $1 \ne T \in \mathcal F (G_0)$ with $T \mid S$.
\end{itemize}
The set
\begin{equation*}
\mathcal B(G_0)=\{S\in \mathcal F(G_0) \colon 1_G\in \pi(S)\} \subseteq \mathcal F (G_0)
\end{equation*}
is a reduced submonoid of $\mathcal F (G_0)$, called  the \textit{monoid of product-one sequences} over $G_0$.  We denote  by $\mathcal A(G_0)$ its set of irreducible elements and by
\[
\mathsf D (G_0) = \sup \{ |S| \colon S \in \mathcal A (G_0) \} \in \N \cup \{\infty\}
\]
the {\it (large) Davenport constant} of $G_0$.

We call $G_0$ \textit{condensed} if every element of $G_0$ is in the support of some sequence in $\mathcal B(G_0)$ (equivalently, if $\mathcal B(G_0)\subseteq\mathcal F(G_0)$ is cofinal).  If $G_0\subseteq G$, then  $G_1=\bigcup_{S\in\mathcal B(G_0)}\supp(S) \subseteq G_0$ is the maximal condensed subset of $G_0$ and $\mathcal B(G_1)=\mathcal B(G_0)$. If $G_0$ is condensed, then $[G_0] = \langle G_0\rangle$. It is obvious, that $[G_0]\subseteq \langle G_0\rangle$. For the other inclusion let $g\in G_0$. Then by condensedness there exists $S\in\mathcal B(G_0)$ such that $g\in\supp(S)$, say $S=g_1\bdot\ldots\bdot g_{|S|}$ with $1=g_1 \ldots g_{|S|}$ and $g=g_r$ for some $r\in [1,|S|]$. Then, since elements that are inverse to each other commute, we have that $1=g_{r+1}\ldots g_{|S|}g_1\ldots g_r$. Therefore $g^{-1}\in\pi(S\bdot g^{[-1]})$, so $g^{-1}\in[G_0]$ and $\langle G_0\rangle \subseteq [G_0]$ follows.  We will use this fact and the argument of its proof throughout the manuscript without further reference. For simplicity of notation, many statements will be formulated for condensed subsets, which entails no restriction in generality.


\section{Algebraic properties of $\mathcal B(G)$ and $\mathcal B(G_0)$} \label{3}

This section contains our main algebraic results (Propositions \ref{3.3}, \ref{3.7}, and Theorems \ref{3.11} and \ref{3.14}). Parts of these results were known before for abelian groups and for finite groups.
We start with some elementary remarks on monoids of product-one sequences. For the following lemma, we denote the set $\{g\in G\colon \text{there is $S\in\mathcal{B}(G_0)$ such that $g\in\pi (S)$}\}$ by $\pi(\mathcal B(G_0))$.

\smallskip
\begin{lemma} \label{3.1}
Let $G$ be a group and let $G_0\subseteq G$ be a condensed subset such that $\langle G_0\rangle=G$. Then $\pi(\mathcal B(G_0))=G'=\pi(\mathcal B(G))$.  Moreover,
$$\{S\in \mathcal F(G)\colon |S|=1, S\in \mathsf q(\mathcal B(G_0))\}=\{S\in \mathcal F(G_0)\colon |S|=1, \supp(S)\subseteq G'\}\,.$$
\end{lemma}

\begin{proof}
Note that for every sequence $S\in\mathcal B(G)$, we have that $\pi(S)\subseteq G'$. Then
	$\pi(\mathcal B(G_0))\subseteq \pi(\mathcal B(G)) \subseteq G'$. To show $G'\subseteq \pi(\mathcal B(G_0))$, it suffices to show $ghg^{-1}h^{-1}\in \pi(\mathcal B(G_0))$ for every $g,h\in G$. Let $g,h\in G=\langle G_0\rangle=[G_0]$. Then there exist $S_g, S_h, S_{g^{-1}},S_{h^{-1}}\in \mathcal F(G_0)$ such that $g\in\pi(S_g)$, $h\in\pi(S_h), g^{-1}\in\pi(S_{g^{-1}})$, and $h^{-1}\in\pi(S_{h^{-1}})$. Therefore $\{1,ghg^{-1}h^{-1}\}\subseteq\pi(S_g\bdot S_h\bdot S_{g^{-1}}\bdot S_{h^{-1}})$ and we are done.

For the moreover statement, let $S=g\in \mathcal F(G)$ with $S=\frac{S_1}{S_2}\in \mathsf q(\mathcal B(G_0))$, where $S_1,S_2\in \mathcal B(G_0)$. Then $\supp(S)\subseteq G_0$ and $\pi(S_1), \pi(S_2)\subseteq G'$, whence $g\in \pi(S_1)\subseteq G'$, i.e., $\supp(S)\subseteq G'$.
For the other inclusion, let $S=g\in \mathcal F(G_0)$ with $g\in G'$. Then $g^{-1}\in G'$ and hence $g^{-1}\in \pi(\mathcal B(G_0))$. There exists $T\in \mathcal B(G_0)$ such that $g^{-1}\in \pi(T)$, whence $S=\frac{g\bdot T}{T}\in \mathsf q(\mathcal B(G_0))$ and we are done.
\end{proof}

In view of the previous lemma we obtain the following: If $G_0\subseteq \langle G_0\rangle'$ (in case $G_0=G$, we speak of \textit{perfect groups}), then $\mathsf q(\mathcal B(G_0))=\mathsf q(\mathcal F(G_0))$, since by $\{S\in\mathcal F(G_0)\colon |S|=1\}\subseteq \mathsf q(\mathcal B(G_0))$ it follows that $\mathcal F(G_0)\subseteq \mathsf q(\mathcal B(G_0))$. Therefore $\mathsf q(\mathcal B(G_0))\cap \mathcal F(G_0)=\mathcal F(G_0)$. If $\langle G_0\rangle'$ is trivial, it is well known that $\mathcal B(G_0)$ is saturated in $\mathcal F(G_0)$. In this sense we can consider the commutator subgroup $\langle G_0\rangle'$ to be a measure of how non-saturated $\mathcal B(G_0)\subseteq \mathcal F(G_0)$ is.\\

The first statement of the following lemma is known for finite groups (\cite[Lemma 3.3]{Oh20a}). Nonetheless, we provide its simple proof.

\smallskip
\begin{lemma} \label{3.2}
Let $G$ be a group and let $G_0\subseteq G$ be a subset.
\begin{enumerate}
\item A  submonoid $H\subseteq \mathcal{B}(G_0)$ is divisor closed if and only if there exists a subset $G_1\subseteq G_0$ such that $H=\mathcal{B}(G_1)$.

\item Let $G_1 \subseteq G_0$ be a subset consisting of torsion elements. Then $G_1$ is a condensed subset and $\mathcal B (G_1)$  is a divisor-closed submonoid of $\mathcal B(G_0)$.
\end{enumerate}

\end{lemma}

\begin{proof}
1. Let $G_1\subseteq G_0$. Then clearly $\mathcal B(G_1)\subseteq \mathcal B(G_0)$ is divisor closed. Conversely, let $H\subseteq \mathcal B(G_0)$ be a divisor closed submonoid. We define $G_1=\bigcup_{B\in H}\supp(B)$ and hence  $H\subseteq \mathcal B(G_1)$. Let $S=g_1\bdot\ldots\bdot g_n\in \mathcal B(G_1)$, where $g_1,\ldots, g_n\in G_1$. By the definition of $G_1$,  there exist $T_1,\ldots,T_n\in H$ such that $g_i\in\supp(T_i)$ for all $i\in[1,n]$. It follows that $g_i^{-1}\in\pi(T_i\bdot g_i^{[-1]})$ for all $i\in [1,n]$ and hence $1\in \pi(g_1^{-1}\bdot\ldots\bdot g_n^{-1})\subseteq \pi(T_1\bdot\ldots\bdot T_n\bdot g_1^{[-1]}\bdot \ldots\bdot g_n^{[-1]})$. Therefore
 $S$ divides $T_1\bdot\hdots \bdot T_n$ in $\mathcal B(G_0)$. By $H\subseteq\mathcal B(G_0)$ being divisor closed, we obtain $S\in H$ and we are done.

2. It is sufficient to show that $G_1$ is condensed. But this is clear, since for all $g\in G_1$ we have that $g^{[\ord(g)]}\in\mathcal B(G_1)$.
\end{proof}

 Lemma \ref{3.2}.2 need not be true  in the non-torsion case. Indeed,
 if  $G=\langle a,b \colon a^n=b^n=1\rangle$ and $G_0=\{a,b,ab\}$, then $a$ and $b$ have finite order, but $ab$ has infinite order. Moreover, $G_0$ is condensed, since $ab\bdot a^{[n-1]}\bdot b^{[n-1]}\in\mathcal B(G_0)$, but $G_1=\{ab\}$ is not. 

Our next result generalizes and refines \cite[Theorem 3.2.1]{Cz-Do-Ge16a}.   Let $G$ be a group and  $G_0\subseteq G$ be a subset. The set
\begin{equation*}
\mathcal B(G_0)^{\ast}=\{S\in \mathcal F(G_0)\colon \pi(S)\subseteq \langle G_0\rangle'\}
\end{equation*}
 is a monoid with
\[
\mathcal B (G_0) \subseteq \mathcal B(G_0)^{\ast} \subseteq \mathcal F (G_0) \,.
\]
If $\langle G_0 \rangle$ is abelian, then $\mathcal B (G_0) = \mathcal B(G_0)^{\ast}$, and if $\langle G_0 \rangle$ is perfect, then $\mathcal B (G_0)^* = \mathcal F (G_0)$. Thus, in both extremal cases, $\mathcal B (G_0)^*$ is a Krull monoid, and the next proposition shows that this holds true in all cases.

\smallskip
\begin{proposition} \label{3.3}
Let $G$ be a group and let $G_0\subseteq G$ be a condensed subset such that $\langle G_0\rangle =G$.
\begin{enumerate}
\item The map
\begin{align*}
\Phi: \mathcal{F}(G_0)/\mathcal{B}(G_0) &\to G /G'\\
[S]_{\mathcal{F}(G_0)/\mathcal{B}(G_0)} &\mapsto gG' \text{ for any } g\in \pi (S),
\end{align*}
is a group isomorphism.
\item $\mathcal B(G_0)^{\ast}=\mathsf q(\mathcal B(G_0))\cap \mathcal F(G_0)$ is a saturated submonoid of $\mathcal F(G_0)$. In particular, $\mathsf q(\mathcal B(G_0))=\mathsf q(\mathcal B(G_0)^{\ast})$ and $\mathcal B(G_0)^{\ast}$ is a Krull monoid.

\item $\mathcal B(G_0)^{\ast}\hookrightarrow \mathcal F(G_0)$ is a divisor theory if and only if for all $g\in G_0$ we have that $[hG'\colon h\in G_0]=[hG'\colon h\in G_0\setminus \{g\}]$. In this case, $\mathcal C(\mathcal B(G_0)^{\ast})\cong G /G'$.

\item Suppose $|G|\neq 2$. Then the inclusion $\mathcal B(G)^{\ast}\hookrightarrow \mathcal F(G)$ is a divisor theory, $\mathcal C(\mathcal B(G)^{\ast})\cong G/G'$ and each class contains precisely $|G'|$ prime divisors.
\item The following are equivalent.
\begin{enumerate}
	\item[(a)] $\widehat{\mathcal B(G_0)}\subseteq \mathcal F(G_0)$ is saturated.
	\item[(b)] The map $\varphi\colon\mathcal C(\widehat{\mathcal B(G_0)}, \mathcal F(G_0))\rightarrow G/G '$ defined by  $\varphi\big([S]_{\widehat{\mathcal B(G_0)}}^{\mathcal F(G_0)}\big)=gG'$, where $g\in\pi(S)$, is an isomorphism.
	\item[(c)] $\widehat{\mathcal B(G_0)}=\mathcal B(G_0)^{\ast}$.
\end{enumerate}
\end{enumerate}
\end{proposition}

\begin{proof}
1. First, note that $G_0$ is condensed if and only if $\mathcal B(G_0)\subseteq \mathcal F(G_0)$ is cofinal, whence $\mathcal F(G_0)/\mathcal B(G_0)$ is a group.
For $S\in\mathcal F(G_0)$ we denote $[S]_{\mathcal F(G_0)/\mathcal B(G_0)}$ by $[S]$ and recall that $[S_1]=[S_2]$ if and only if there exist $C_1,C_2\in \mathcal B(G_0)$ such that $S_1\bdot C_1=S_2\bdot C_2$.

To show that $\Phi$ is well-defined, let $S_1,S_2\in\mathcal F(G_0)$ with $[S_1]=[S_2]$ and let $g_1\in\pi(S_1)$ and $g_2\in \pi(S_2)$. Then there exist $C_1,C_2\in\mathcal B(G_0)$ such that $S_1\bdot C_1=S_2\bdot C_2$. As
$g_1\in \pi(S_1)\subseteq \pi(S_1\bdot C_1)=\pi(S_2\bdot C_2)\subseteq g_2 G'$, we obtain
 $g_1 G'=g_2 G'$.

 Obviously, $\Phi$ is a group homomorphism.
For the surjectivity, let $g\in G=\langle G_0\rangle=[G_0]$. Then there is $S\in \mathcal F(G_0)$ such that $g\in\pi(S)$.
To prove injectivity, let $S_1,S_2\in\mathcal F(G_0)$ with $g_1\in\pi(S_1),g_2\in\pi(S_2)$ such that $g_1 G'=g_2 G'$. Then there exists $h\in G'$ such that $g_1=g_2h$. By Lemma \ref{3.1}, there exists $T\in\mathcal B(G_0)$ with $h\in \pi(T)$. Since $G=[G_0]$, there is a sequence $W\in \mathcal F(G_0)$ such that $g_1^{-1}\in \pi(W)$, whence $1=g_1g_1^{-1}1\in \pi(S_1\bdot W\bdot T)$ and $1=g_2hg_1^{-1}\in\pi(S_2\bdot T\bdot W)$.
Note that $S_1\bdot (S_2\bdot T\bdot W)=S_2\bdot (S_1\bdot W\bdot T)$ and hence $[S_1]=[S_2]$.

2. Let $S\in \mathcal B(G_0)^*$. Then $\pi(S)\subseteq G'$. By Lemma \ref{3.1}, there exists $T\in \mathcal B(G_0)$ such that $1\in \pi(S\bdot T)$. Therefore $S=\frac{S\bdot T}{T}\in \mathsf q(\mathcal B(G_0))\cap \mathcal F(G_0)$.
Let $S\in \mathsf q(\mathcal B(G_0))\cap \mathcal F(G_0)$. Then there exist $S_1,S_2\in \mathcal B(G_0)$ such that $S=\frac{S_1}{S_2}$, whence $\pi(S)\subseteq \pi(S_1)\pi(S_2)^{-1}\subseteq G'G'=G'$. Therefore $S\in \mathcal B(G_0)^*$.

We proved $\mathcal B(G_0)^*=\mathsf q(\mathcal B(G_0))\cap \mathcal F(G_0)$. Since $\mathcal B(G_0)\subseteq \mathcal B(G_0)^*\subseteq \mathsf q(\mathcal B(G_0))$, we obtain $\mathsf q(\mathcal B(G_0)^*)=\mathsf q(\mathcal B(G_0))$ and hence $\mathcal B(G_0)^*=\mathsf q(\mathcal B(G_0)^*)\cap \mathcal F(G_0)$.
It follows that $\mathcal B(G_0)^*$ is a saturated submonoid of $\mathcal F(G_0)$. In particular, as a saturated submonoid of a Krull monoid, $\mathcal B(G_0)^{\ast}$ is Krull \cite[Proposition 2.4.4.3]{Ge-HK06a}.

3.
Suppose $\mathcal B(G_0)^{\ast}\hookrightarrow \mathcal F(G_0)$ is a divisor theory and let $g\in G_0$. It suffices to show $gG'\in [fG'\colon f\in G_0\setminus \{g\}]$.
Note that for every $h\in G_0$, there exist  $n\in \N$ and $S_1,\hdots ,S_n\in \mathcal B(G_0)^{\ast}$ such that $h=\gcd(S_1,\hdots , S_n)$. If $h\neq g$, then there exists $i\in [1,n]$ such that $\mathsf v_h(S_i)>0$  and $\mathsf v_g(S_i)=0$. Assume that $S_i=h\bdot h_1\bdot\ldots\bdot h_{t}$ with $hh_1\ldots h_t\in G'$, where $t\in \N$ and $h_1,\ldots,h_{t}\in G_0\setminus\{g\}$. Then $h^{-1}G'=(h_1G')\ldots (h_tG')\in [fG'\colon f\in G_0\setminus \{g\}]$.
 This implies that $[fG'\colon f\in G_0\setminus \{g\}]$ is indeed an abelian group. If $h=g$, then there exists $i\in [1,n]$ such that $\mathsf v_g(S_i)=1$, whence $g^{-1}G'\in [fG'\colon f\in G_0\setminus \{g\}]$. It follows by the fact that $[fG'\colon f\in G_0\setminus \{g\}]$ is a group that $gG'\in [fG'\colon f\in G_0\setminus \{g\}]$.
%

Suppose for all $g\in G_0$, we have that $[hG'\colon h\in G_0]=[hG'\colon h\in G_0\setminus \{g\}]$. To show $\mathcal B(G_0)^{\ast}\hookrightarrow \mathcal F(G_0)$ is a divisor theory, by 2., it suffices to prove for every $g\in G_0$, there exist $T_1,\ldots, T_n\in \mathcal B(G_0)^{\ast}$ such that $g=\gcd(T_1,\ldots, T_n)$, where $n\in \N$.
Let $g\in G_0$. Note that $[G_0]=G$, whence $[hG':h\in G_0\setminus \{g\}]=G/G'$. So there exists $W=g_1\bdot\ldots\bdot g_{\ell}\in \mathcal F(G_0\setminus\{g\})$ such that $\pi(g\bdot W)\subseteq G'$. Furthermore, for every $j\in [1,\ell]$, there exists $W_j\in \mathcal F(G_0\setminus \{g_j\})$ such that $\pi(g\bdot W_j)\subseteq G'$. It follows that $\gcd(g\bdot W, g\bdot W_1,\ldots, g\bdot W_{\ell})=g$ and we are done.
If this is the case, we obtain that $\mathcal C(\mathcal B(G_0)^*)=\mathcal F(G_0)/\mathcal B(G_0)^*$ just by the very definition of the class group \cite[Definition 2.4.9]{Ge-HK06a}. Since by 2. we have that $\mathsf{q}(\mathcal B(G_0)^*)=\mathsf q(\mathcal B(G_0))$, we obtain by 1. that $\mathcal C(\mathcal B(G_0)^*)=\mathcal F(G_0)/\mathcal B(G_0)^*=\mathcal F(G_0)/\mathcal B(G_0)\cong G/G'$.

4. If $|G|=1$, the statement is trivial.  Suppose $|G|\geq 3$. By 3. we just need to show that $G=[G\setminus\{g\}]$ for all $g\in G$. Let $g\in G$. Then there is $h\in G\setminus \{1, g\}$ such that $g=h(h^{-1}g)\in[G\setminus \{g\}]$ and we are done.

It remains to prove the statement on the prime divisors. Let $S\in\mathcal F(G)$ with $g\in\pi(S)$ and let $h\in G$.
It suffices to show $[S]=[h]$ if and only if $h\in gG'$. In fact, if $[S]=[h]$, then there exist $C_1,C_2\in \mathcal B(G)^*$ such that $h=\frac{S\bdot C_1}{C_2}$ and hence $h\in \pi(S)\pi(C_1)\pi(C_2)^{-1}\subseteq gG'$. If $h\in gG'$, then $hg^{-1}\in G'$ and thus $h\bdot g^{-1}, S\bdot g^{-1}\in \mathcal B(G_0)^\ast$. It follows that $S\bdot (h\bdot g^{-1})=h\bdot (S\bdot g^{-1})$.

5. $(a)\Leftrightarrow (b)$ Since $G_0$ is condensed, we obtain $\mathcal B(G_0)\subseteq \mathcal F(G_0)$ and hence $\widehat{\mathcal B(G_0)}\subseteq \mathcal F(G_0)$ are cofinal. By 1., we have that $\mathcal F(G_0)/\widehat{\mathcal B(G_0)}=\mathcal F(G_0)/\mathcal B(G_0)\cong G /G'$. By \cite[Prop. 2.8.7.3]{Ge-HK06a}, the map $\varphi$ is  a well-defined epimorphism and
$\widehat{\mathcal B(G_0)}\subseteq \mathcal F(G_0)$ is saturated if and only if $\varphi$ is an isomorphism.

$(a)\Leftrightarrow (c)$ It follows by 2. that $\widehat{\mathcal B(G_0)}\subseteq \mathcal F(G_0)$ is saturated if and only if $\widehat{\mathcal B(G_0)}=\mathsf q(\widehat{\mathcal B(G_0)})\cap \mathcal F(G_0)=\mathsf q(\mathcal B(G_0)^*)\cap \mathcal F(G_0)=\mathcal B(G_0)^*$.
\end{proof}

\smallskip
\begin{lemma}\label{3.6}
Let $G$ be a group and let  $G_0\subseteq G$ be a condensed subset such that $\langle G_0\rangle=G$. If  $G '$ is torsion or $G_0$ consists of torsion elements,  then $\widetilde{ \mathcal B(G_0)} =\widehat{\mathcal B(G_0)}=\mathcal B(G_0)^{\ast}$ is Krull.
\end{lemma}

\begin{proof}
We just proved in Proposition \ref{3.3}, that $\mathcal B(G_0)^{\ast}$ is a Krull monoid, having the same quotient group as $\mathcal B(G_0)$. Therefore $\widetilde{\mathcal B(G_0)}\subseteq \widehat{\mathcal B (G_0)}\subseteq \mathcal B(G_0)^{\ast}$ and it suffices to show $\mathcal B(G_0)^*\subseteq \widetilde{\mathcal B(G_0)}$.

Let $S=g_1\bdot\hdots\bdot g_{\ell} \in \mathcal B(G_0)^{\ast}$ and $g\in \pi(S)$. We have to verify that  $S\in \widetilde{\mathcal B(G_0)}$.
If $G '$ is torsion, then $g$ has finite order, whence $S^{[\ord(g)]}\in \mathcal B(G_0)$ and  $S\in\widetilde{\mathcal{B}(G_0)}$.
If $G_0$ consists of torsion elements, then for $\alpha=\lcm\{\ord(g_i)\colon i\in[1,n]\}$, we have that $S^{[\alpha]}\in\mathcal B(G_0)$, whence  $S\in\widetilde{\mathcal B(G_0)}$.
\end{proof}

The above statement provides sufficient conditions implying that $\widehat{\mathcal B(G_0)}=\mathcal B(G_0)^{\ast}$. However, in general, iterated complete integral closures of $\mathcal B (G_0)$ may still be proper submonoids of $\mathcal B(G_0)^{\ast}$, even for finite subsets $G_0$.

\smallskip
\begin{example}
Let $G=\langle a,b\rangle$ be the free group with two generators and consider the subset $G_0=\{a,a^{-1},b,b^{-1},aba^{-1}b^{-1}\}$. Then $\mathcal B(G_0)=\widehat{\mathcal B(G_0)}$, but $\mathcal B(G_0)\subsetneq \mathcal B(G_0)^{\ast}$.
\end{example}

\begin{proof}
It is easy to see, that $\mathcal A(G_0)=\{a\bdot a^{-1}, b\bdot b^{-1}, aba^{-1}b^{-1}\bdot a\bdot a^{-1}\bdot b\bdot b^{-1}\}$ and that therefore $\mathcal B(G_0)$ is factorial, since the occurence of the commutator element $aba^{-1}b^{-1}$ in a sequence determines its factorization uniquely. Thus $\mathcal B(G_0)$ is completely integrally closed, but $aba^{-1}b^{-1}\in\mathcal B(G_0)^{\ast}\setminus \mathcal B(G_0)$.
\end{proof}

\smallskip
\begin{proposition}\label{3.7}
Let $G$ be a group and $G_0\subseteq G$ be a subset consisting of torsion elements. Then the following are equivalent.
\begin{enumerate}
\item[(a)] $\mathcal B(G_0)$ is Krull.
\item[(b)] $\mathcal B(G_0)$ is root closed.
\item[(c)] $\mathcal B(G_0)\subseteq \mathcal F(G_0)$ is saturated.
\item [(d)]$\mathcal B(G_0)=\mathcal B(G_0)^{\ast}$.
\end{enumerate}
\end{proposition}

\begin{proof}
$(a)\Rightarrow (b)$ is by definition and $(c)\Rightarrow (a)$ is by \cite[Proposition 2.4.4.3]{Ge-HK06a}.

$(b)\Rightarrow (c)$ By Lemma \ref{3.2}, we know $G_0$ is condensed. It follows by Lemma \ref{3.6} and Proposition \ref{3.3}.2, that $\mathcal B(G_0)=\mathcal B(G_0)^*$ is a saturated submonoid of $\mathcal F(G_0)$.

$(b)\Rightarrow (d)$ is clear by Lemma \ref{3.6} and $(d)\Rightarrow (b)$ follows from Proposition \ref{3.3}.
\end{proof}

\smallskip
\begin{proposition}\label{3.8}
Let $G$ be a group and let $G_0\subseteq G$ be a condensed subset. Then $s$-$\spec(\mathcal B(G_0))=\{\mathfrak p_X \colon  X\subseteq G_0\}$, where $\mathfrak p_X=\{S\in\mathcal B(G_0)\colon \text{there is } g\in X \text{ such that } \mathsf v_g(S)\geq 1\}$. In particular, $\mathfrak X(\mathcal B(G_0))\subseteq \{\mathfrak p_g\colon g\in G_0\}$.
\begin{enumerate}
\item If $G_0=G$ or $G_0$ consists of torsion elements, then equality holds and $|\mathfrak X(\mathcal B(G_0))|=|G_0|$.
\item If $\mathfrak X(\mathcal B(G_0))= \{\mathfrak p_g\colon g\in G_0\}$, then $\bigcap_{\mathfrak p\in\mathfrak X(\mathcal B(G_0))}\mathcal B(G_0)_{\mathfrak p}\subseteq \mathcal F(G_0)$.
\end{enumerate}

\end{proposition}

\begin{proof}
Let $ X\subseteq G_0$. Clearly, $\mathfrak p_X$ is a prime $s$-ideal of $\mathcal B(G_0)$. So it remains to show that every prime $s$-ideal can be written in such a form. Let $P\in s$-$\spec(\mathcal B(G_0))$. Then $\mathcal B(G_0)\setminus P$ is a divisor closed submonoid. By Lemma \ref{3.2}, we obtain $\mathcal B(G_0)\setminus P=\mathcal B(G_1)$ for some $G_1\subseteq G_0$. Therefore
 $$P=\{S\in\mathcal B(G_0)\colon \text{there is } g\in G_0\setminus G_1 \text{ such that } \mathsf v_g(S)\geq 1\}=\mathfrak p_{G_0\setminus G_1}\,.$$
 By definition, we know $ X\subseteq X'$ implies $\mathfrak p_X\subseteq \mathfrak p_{X'}$, whence $\mathfrak X(\mathcal B(G_0))\subseteq \{\mathfrak p_g\colon g\in G_0\}$.

	Suppose $G_0=G$. It suffices to show $\{\mathfrak p_g\colon g\in G\}\subseteq \mathfrak X(\mathcal B(G))$.
	Note that $g\bdot g^{-1} \in {\mathfrak p}_g \setminus {\mathfrak p}_h$ for all $h \in G \setminus \{g, g^{-1}\}$ and if $g\neq g^{-1}$, then $g^{[2]}\bdot g^{-2}\in \mathfrak p_g\setminus \mathfrak p_{g^{-1}}$, where $g,h\in G$. Therefore  ${\mathfrak p}_g \not\subseteq {\mathfrak p}_h$ for any distinct elements $g, h \in G$ and we are done.

	Suppose $G_0$ consists of torsion elements. It suffices to show $\{\mathfrak p_g\colon g\in G_0\}\subseteq \mathfrak X(\mathcal B(G_0))$.
	 Note that $g^{[\ord(g)]}\in {\mathfrak p}_g \setminus {\mathfrak p}_h$ for all $g\in G_0$ and all $h \in G_0 \setminus \{g\}$. Therefore  ${\mathfrak p}_g \not\subseteq {\mathfrak p}_h$ for any distinct elements $g, h \in G_0$ and we are done.

	 Suppose that $\mathfrak X(\mathcal B(G_0))= \{\mathfrak p_g\colon g\in G_0\}$ and that $S\in\bigcap_{\mathfrak p\in\mathfrak X(\mathcal B(G_0))}\mathcal B(G_0)_{\mathfrak p}$. Say $S=\frac{U}{T}$, where $U, T\in\mathcal F(G_0)$ with $\supp(U)\cap \supp(T)=\emptyset$. We have to prove that $T=1_{\mathcal F(G_0)}$. Since for every $g\in G_0$ we have that $S=\frac{U_g}{T_g}$ with $U_g,T_g\in\mathcal B(G_0)$ and $g\notin\supp(T_g)$, it follows by $U\bdot T_g=U_g\bdot T$ that $g\notin \supp(T)$ for all $g\in G_0$, what proves the assertion. 
%
\end{proof}

If $G$ is a finite group, it is well-known (and for the convenience of the reader this will also follow from Theorem \ref{3.11}) that $\mathcal B(G)$ is a C-monoid. The next statement shows that this is never the case for infinite groups.

\smallskip
\begin{proposition} \label{3.9}
Let $G$ be a group and $G_0\subseteq G$ be a condensed subset.
\begin{enumerate}
\item If $G_0$ is infinite, then $\mathcal B(G_0)$ is not a C-monoid defined in $\mathcal F(G_0)$.
\item If $G$ is infinite, then $\mathcal B(G)$ is not a C-monoid.
\item If $\mathcal B(G_0)$ is a C-monoid defined in $\mathcal F(G_0)$, then $\langle G_0\rangle /\langle G_0\rangle '$ is finite.
\end{enumerate}
\end{proposition}

\begin{proof}
1. Suppose $G_0$ is infinite. Let $g\in G_0$ and $S\in \mathcal F(G_0)$ such that $g\bdot S\in \mathcal B(G_0)$. Then for every $h\in G_0$ with $[h]_{\mathcal B(G_0)}^{\mathcal F(G_0)}=[g]_{\mathcal B(G_0)}^{\mathcal F(G_0)}$, we have that $h\bdot S\in \mathcal B(G_0)$, whence $$\left\{h\in G_0\colon [h]_{\mathcal B(G_0)}^{\mathcal F(G_0)}=[g]_{\mathcal B(G_0)}^{\mathcal F(G_0)}\right\}\subseteq \{h\in G_0\colon h^{-1}\in \pi(S)\}\quad \text{ is finite}\,.$$
Therefore $\{[h]_{\mathcal B(G_0)}^{\mathcal F(G_0)}\colon h\in G_0\}$ and hence $\mathcal C(\mathcal B(G_0), \mathcal F(G_0))$ are infinite.

2. Assume to the contrary that $\mathcal B(G)$ is a C-monoid. By \cite[Theorem 2.9.11.2]{Ge-HK06a}, it follows that $\mathcal C(\widehat{\mathcal B(G)})$ is finite and $(\mathcal B(G):\widetilde{\mathcal B(G)})\neq\emptyset$. If $G'$ is finite, then by Proposition \ref{3.3} and Lemma \ref{3.6}, we know $\mathcal C(\widehat{\mathcal B(G)})\cong G/G'$ is infinite, a contradiction. Suppose $G'$ is infinite.

If $G'$ is torsion, then by Proposition \ref{3.3} $\{S\in \mathcal F(G')\colon |S|=1\}\subseteq \mathsf q(\mathcal B(G)^*)=\mathsf q(\mathcal B(G))$ and for all $g\in G'$ there exists $n\in\N$ such that $g^{[n]}\in\mathcal B(G)$, whence $\{S\in \mathcal F(G')\colon |S|=1\}\subseteq \widetilde{\mathcal B(G)}$. Let  $U\in (\mathcal B(G):\widetilde{\mathcal B(G)})$. Then  $U\bdot g\in\mathcal B(G)$ for all $g\in G'$, whence $G'\subseteq \pi(U)$ must be finite, a contradiction.

If $G'$ is not torsion, then there exists an element $g\in G'$ such that  $\ord(g)=\infty$. Then $\mathcal B(\langle g\rangle)\subseteq \mathcal B(G)$ is a divisor closed submonoid, but is a Krull monoid with infinite class group, whence not a C-monoid.  Therefore $\mathcal B(G)$ cannot be a C-monoid, because divisor closed submonoids of C-monoids are C-monoids by \cite[Theorem 2.9.15.1]{Ge-HK06a}.


3. Since $G_0$ is condensed, $\mathcal B(G_0)\subseteq \mathcal F(G_0)$ is cofinal and by Proposition \ref{3.3} and \cite[Theorem 2.8.7.1]{Ge-HK06a}, there exists an epimorphism $\mathcal C(\mathcal B(G_0),\mathcal F(G_0))\rightarrow \mathcal F(G_0)/\mathcal B(G_0)\cong \langle G_0\rangle /\langle G_0\rangle '$, which completes the proof, since $\mathcal C(\mathcal B(G_0),\mathcal F(G_0))$ is finite by definition, as $\mathcal B(G_0)$ is a reduced C-monoid.
\end{proof}

The next example shows that in the proof of Proposition \ref{3.9}.1 it is necessary to assume that $\mathcal B(G_0)$ is a C-monoid defined in $\mathcal F(G_0)$. Moreover, we can see that the statements "$\mathcal B(G_0)$ is a C-monoid" and "$\mathcal B(G_0)$ is a C-monoid defined in $\mathcal F(G_0)$" are not equivalent.

\smallskip
\begin{example}\label{3.10}
Let $X$ be a set, $G$ be the free group over $X$, and $G_0=\{x,x^{-1}\colon x\in X\}$. Then $\mathcal B(G_0)$ is a C-monoid, but not a C-monoid defined in $\mathcal F(G_0)$.
\end{example}
\begin{proof}
Since the elements of $X$ have no relations, a sequence $S\in\mathcal F(G_0)$ is a product-one sequence if and only if for all $g\in G_0$ we have that $\mathsf v_g(S)=n$ implies $\mathsf v_{g^{-1}}(S)=n$. Therefore the elements of $\mathcal B(G_0)$ are of the form $g_1^{[n_1]}\bdot (g_1^{-1})^{[n_1]}\bdot\hdots\bdot g_r^{[n_r]}\bdot (g_r^{-1})^{[n_r]}$ for $g_i\in X$  and $r, n_i\in\N$. Then $\mathcal A(G_0)=\{g\bdot g^{-1} \colon g\in X\}$ and we obtain the factoriality of $\mathcal B(G_0)$. By definition, every factorial monoid is a C-monoid in itself. To see that $\mathcal B(G_0)$ is not a C-monoid defined in $\mathcal F(G_0)$, just note that for every $x\in X$ we have that  $x, x^{[2]}, x^{[3]},\hdots$ are all in different classes.
\end{proof}

Let $H$ be a monoid and $a\in H$. We set $[\![ a]\!]$ to be the submonoid of $H$ consisting of all the divisors of powers of $a$. $H$ is said to be
\begin{itemize}
\item a \textit{G-monoid} if there exists $a\in H$ such that $H=[\![a]\!]$ is a divisor closed submonoid  generated by $a$ ( or equivalently if $\bigcap_{\mathfrak p\in \text{$s$-}\spec(H)\setminus \{\emptyset\}}\mathfrak p\neq \emptyset$; for a list of equivalent conditions see \cite[Lemma 2.7.7]{Ge-HK06a}),

\item \textit{finitary} if it is a \BF-monoid (see Section 4) and there exist $n, M \in\N$ and $u_1,\hdots ,u_n\in H\setminus H^{\times}$ such that $(H\setminus  H^{\times})^M\subseteq \{u_1,\hdots ,u_n\}H$, where $(H\setminus H^\times)^M=\{a_1\cdot\ldots\cdot a_M\colon a_i\in H\setminus H^\times)\}$. In that case $\{u_1,\hdots ,u_n\}$ is called a \textit{finite almost generating set} of $H$ (for background on finitary monoids we refer to \cite[Chapters 2.7 and 4.4]{Ge-HK06a}).
\end{itemize}
Finitely generated monoids are $v$-noetherian G-monoids, and $v$-noetherian G-monoids are finitary (see \cite[Theorems 2.7.9 and 2.7.13]{Ge-HK06a}), but none of the converse implication holds. The next proposition  characterizes when monoids of product-one sequences are finitary. Theorem \ref{3.11}, in combination with Example \ref{infinitefinitary}, shows that a finitary monoid of product-one sequences need not be a G-monoid.

\smallskip
\begin{proposition}\label{finitary}
Let $G$ be a group and let $G_0\subseteq G$ be a condensed subset. Then $\mathcal B(G_0)$ is finitary if and only if there exist $n\in\N$ and nonempty sequences $A_1,\hdots ,A_n\in\mathcal B(G_0)$ such that for all nonempty sequences $S\in\mathcal B(G_0)$, there exists $i\in[1,n]$ with $\supp(A_i)\subseteq \supp(S)$. In particular, if $G_0$ is finite, then $\mathcal B(G_0)$ is finitary.
\end{proposition}

\begin{proof}
Suppose $\mathcal B(G_0)$ is finitary. Then there exist $n\in \N$, a finite almost generating set $A_1,\hdots ,A_n$, and $m\in\N$ such that for all nonempty sequences $S\in\mathcal B(G_0)$ there exists $i\in[1,n]$ with $A_i\mid_{\mathcal B(G_0)} S^{[m]}$, whence $\supp(A_i)\subseteq\supp(S^{[m]})=\supp(S)$.

On the other hand, suppose there are $n\in \N$ and $A_1,\ldots, A_n\in \mathcal B(G_0)$ such that for every nonempty sequence $S\in \mathcal B(G_0)$, there exists $i\in [1,n]$ such that $\supp(A_i)\subseteq \supp(S)$. Since $\mathcal B(G_0)$ is a \BF-monoid, it suffices to show there exists $M\in \N$ such that $(\mathcal B(G_0)\setminus \{1_{\mathcal B(G_0)}\})^M\subseteq \{A_1,\hdots , A_n\}\bdot\mathcal B(G_0)$.

Set $m=\max\{|A_i|\colon i\in [1,n]\}$  and $M=n(m-1)+1$. Let  $S_1,\hdots , S_{M}\in\mathcal B(G_0)\setminus \{1_{\mathcal B(G_0)}\}$. Then there exists $i\in [1,n]$, say $i=1$, such that there exists a subset $I\subseteq [1,M]$ with $|I|=m$ and $\supp(A_1)\subseteq \supp(S_j)$ for all $j\in I$. After renumbering if necessary, we may assume that $I=[1,m]$. Suppose $A_1=g_1\bdot\ldots\bdot g_{\ell}$, where $\ell\le m$ and $g_1,\ldots,g_{\ell}\in G_0$, and let $T_i=S_i\bdot g_i^{[-1]}$ for all $i\in [1,\ell]$. Therefore $g_i^{-1}\in \pi(T_i)$ for all $i\in [1,\ell]$. It follows by $1\in \pi(A_1)$ that $$1\in \pi(g_1^{-1}\bdot\ldots\bdot g_{\ell}^{-1})\subseteq \pi(T_1\bdot\ldots\bdot T_{\ell})\subseteq \pi(S_1\bdot\ldots\bdot S_{\ell}\bdot A_1^{[-1]})\subseteq \pi(S_1\bdot \ldots\bdot S_M\bdot A_1^{[-1]})\,,$$
whence $S_1\bdot \ldots\bdot S_M\in A_1\bdot \mathcal B(G_0)$.

For the "in particular" statement, we suppose $G_0$ is finite.
Let $E\subseteq \mathcal B(G_0)$ be a maximal subset such that for any two distinct $S_1,S_2\in E$, we have that $\supp(S_1)\neq \supp(S_2)$. It follows by the fact that $G_0$ has only finitely many subsets that $E$ is finite, whence the assertion follows.
\end{proof}

\smallskip
\begin{theorem} \label{3.11}
Let $G$ be a group and let $G_0\subseteq G$ be a condensed subset.
\begin{enumerate}
\item The following statements are equivalent.
    \begin{enumerate}
    \item $\mathcal B (G_0)$ is a G-monoid.
    \item $G_0$ is finite.
    \item $s$-$\spec ( \mathcal B (G_0))$ is finite.
    \item $\mathcal B(G_0)^{\ast}$ is a finitely generated Krull monoid.
    \end{enumerate}

\item The following statements are equivalent.
	\begin{enumerate}
	\item $\mathcal{B}(G_0)$ is finitely generated.
	\item $\mathcal B(G_0)$ is a G-monoid and $\mathsf D(G_0)<\infty$.
	\item $G_0$ is finite and $\mathsf D(G_0)<\infty$.
	\end{enumerate}
	\smallskip
	\noindent
	If, in addition, $G_0$ consists of torsion elements, then the following conditions are also equivalent to the conditions $2(a)-2(c)$ listed above.
	\smallskip
	
	\begin{enumerate}
	\item[(d)] $\mathcal{B}(G_0)$ is a C-monoid  defined in $\mathcal F(G_0)$ and $G_0$ is finite.
	\item[(e)] $\mathcal B(G_0)$ is a C-monoid defined in $\mathcal F(G_0)$.
	\item[(f)] $\mathcal B(G_0)$ is a C-monoid and $G_0$ is finite.
	\item[(g)] $\mathcal{B}(G_0)$ is v-noetherian and $G_0$ is finite.
	\item[(h)] $G_0$ is finite.
	\end{enumerate}
\end{enumerate}
\end{theorem}

\begin{proof}
1. (a) $\Rightarrow$ (b) If $\mathcal B(G_0)$ is a G-monoid, then there exists $S\in\mathcal B(G_0)$ such that $\mathcal B(G_0)=\llbracket S\rrbracket=\mathcal B(\supp(S))$. Since $G_0$ is condensed, it follows that $G_0=\supp(S)$ is finite.

(b) $\Rightarrow$ (c) If $G_0$ is finite, then there are only finitely many subsets of $G_0$ and the assertions follows by Proposition \ref{3.8}.

(c) $\Rightarrow$ (a)  follows from \cite[Lemma 2.7.7]{Ge-HK06a}.

(b) $\Rightarrow$ (d) Let $G_0$ be finite. We say $S_1\le S_2$ if $S_1 \t_{\mathcal F(G_0)} S_2$, where $S_1,S_2\in \mathcal F(G_0)$. Then  Dickson's Lemma \cite[Theorem 1.5.3]{Ge-HK06a} implies $\mathcal B(G_0)^{\ast}$ has  finitely many minimal elements, say $A_1,\hdots , A_n$, where $n\in \N$. By Proposition \ref{3.3}, it folllows that $\mathcal B(G_0)^*$ is a Krull monoid and hence it suffices to show $\mathcal A(\mathcal B(G_0)^*)\subseteq \{A_i\colon i\in [1,n]\}$. Let $A\in \mathcal A(\mathcal B(G_0)^*)$. Then there exists $i\in [1,n]$ such that $A_i\t_{\mathcal F(G_0)} A$. Since $\mathcal B(G_0)^{\ast}$ is a saturated submonoid of $\mathcal F(G_0)$ by Proposition \ref{3.3}, we obtain $A\bdot A_i^{[-1]}\in \mathcal B(G_0)^*$ and hence $A=A_i$.

(d) $\Rightarrow$ (b) Suppose $\mathcal B(G_0)^{\ast}$ is finitely generated and suppose  $\mathcal A(\mathcal B(G_0)^*)=\{A_1,\hdots , A_n\}$, where $n\in \N$. Since $G_0$ is condensed, we obtain $G_0=\bigcup_{i\in[1,n]} \supp(A_i)$ is finite.

\medskip
2. $(b)\Rightarrow (c)$ follows from 1.

$(a)\Rightarrow (b)$ Every finitely generated monoid is a G-monoid by \cite[Theorem 2.7.13]{Ge-HK06a}. Moreover, if $\mathcal B(G_0)$ is finitely generated, then $\mathcal A(G_0)$ is finite, whence $\mathsf D(G_0)<\infty$.

$(c)\Rightarrow (a)$  Suppose $G_0$ and $\mathsf D(G_0)$ are both finite. Then the set $\{S\in \mathcal F(G_0)\colon |S|\le \mathsf D(G_0)\}$ is finite  and the assertion follows by the fact that $\mathcal A(G_0)\subseteq \{S\in \mathcal F(G_0)\colon |S|\le \mathsf D(G_0)\}$.

\smallskip
Now suppose $G_0$ consists of torsion elements. Then $(d)\Rightarrow (f)$ and $(g)\Rightarrow (h)$ follow by definition, $(f)\Rightarrow (g)$ is just \cite[Theorem 2.9.13]{Ge-HK06a} and $(e)\Rightarrow (d)$ follows by Proposition \ref{3.9}. It suffices to show $(h)\Rightarrow (a)\Rightarrow (e)$.

$(h)\Rightarrow (a)$  Let $G_0$ be finite and let $$E=\{S\in\mathcal F(G_0) \colon \mathsf v_g(S)<\ord(g) \text{ for all } g\in G_0 \}\,.$$ Then $E$ is finite and for every $A\in\mathcal A(G_0)$ there exists precisely one $T\in E$ such that $\mathsf v_g(A)\equiv \mathsf v_g(T) \ (\ord(g))$ for all $g\in G_0$, whence $A\bdot T^{[-1]}\in\mathcal F(P)$, where $P=\{g^{[\ord(g)]}\colon g\in G_0 \}$ is finite.

For every $T\in E$, we set
\begin{equation*}
W_T=\{A\bdot T^{[-1]}\colon A\in\mathcal A(G_0) \text{ and } \mathsf v_g(A)\equiv\mathsf v_g(T) \ (\ord(g)) \text{ for all } g\in G_0 \}\subseteq \mathcal  F(P).
\end{equation*}
To show $\mathcal A(G_0)$ is finite, it is sufficient to show $W_T$ is finite for all $T\in E$. Let $T\in E$. If there exist $U,V\in W_T$ such that $U\t_{\mathcal F(P)} V$, then $U\bdot T$ divides $V\bdot T$ in $\mathcal B(G_0)$ which implies $U=V$. We say $U\le V$ if $U\t_{\mathcal F(P)} V$ for $U,V\in \mathcal F(G_0)$.
 Thus every element of $W_T$ can be viewed as a minimal element of $W_T$.
 Note that $\mathcal F(P)\cong \N_0^{\t G_0\t}$. It follows by Dickson's Lemma \cite[Theorem 1.5.3]{Ge-HK06a} that $W_T$ is finite.

$(a)\Rightarrow (e)$ Suppose $\mathcal B(G_0)$ is finitely generated. Then $G_0$ is finite, as we already proved $2. (a)$ being equivalent to $2. (c)$. Note that $G_0$ consists of torsion elements.
Let $\alpha =\lcm\{\ord(g)\colon g\in G_0\}$. Then for all $S\in \mathcal F(G_0)$, we have that $S^{[\alpha]}\in \mathcal B(G_0)$. It follows by \cite[Proposition 2.6.3]{Cz-Do-Ge16a} that $\mathcal B(G_0)$ is a C-monoid defined in $\mathcal F(G_0)$.
\end{proof}

\smallskip
\begin{remark}\label{3.12}
Neither Proposition \ref{3.7} nor Theorem \ref{3.11}.2 hold true in the non-torsion case, since for both, Example \ref{taualpha} is a counterexample with just one non-torsion element. Also we cannot conclude that finite $G_0$ implies $\mathcal B(G_0)$ is finitely generated in the non-torsion case, as this example shows. To see that $G_0$ finite need not imply $\mathcal B(G_0)$ is a C-monoid, we take a look at $G=\Z$ and $G_0=\{1,-1,2,-2\}$. Then $\mathcal B(G_0)$ is a finitely generated Krull monoid with infinite class group, hence not a C-monoid.
\end{remark}

Before stating the next theorem, we need the following elementary lemma for avoiding too many calculations. We will use it without further mention. 

\smallskip
\begin{lemma}\label{3.13}
Let $G$ be a group such that the commutator subgroup $G'$ is an elementary 2-group and $G'\subseteq \mathsf Z(G)$ is a subgroup of the center $\mathsf Z(G)=\{g\colon gh=hg\text{ for all }h\in G\}$ and let $g,f,h\in G$.
\begin{enumerate}
\item $[g,h]=[h^{-1},g]=[h,g]$.
\item $[fh,g]=[f,g][h,g]$.
\item $[fh,gf]=[f,h][f,g][h,g]$.
\end{enumerate}
\end{lemma}

\begin{proof}
1. Since every non-trivial element of $G'$ has order 2, we obtain $[g,h]=[g,h]^{-1}=[h,g]$. On the other hand we have that $[g,h]=g^{-1}h^{-1}gh=hh^{-1}(g^{-1}h^{-1}gh)=h(g^{-1}h^{-1}gh)h^{-1}=hg^{-1}h^{-1}g=[h^{-1},g]$.

2. $[fh,g]=h^{-1}f^{-1}g^{-1}fhg=h^{-1}(f^{-1}g^{-1}fg)g^{-1}f^{-1}fhg=[f,g][h,g]$.

3. Just apply 1. and 2. repeatedly and use the fact that $G'$ is abelian.
\end{proof}

On the one hand, the next theorem generalizes the fact that $\mathcal B(G)$ is Krull if and only if $G$ is abelian to arbitrary groups. On the other hand, the additional equivalent statements that are formulated show us reasons why $\mathcal B(G)$ fails to be Krull for non-abelian $G$. Since one of these is the root closedness of $\mathcal B(G)$, one could ask if it also lacks seminormality and except for a rare case where $|G'|=2$, non-seminormality holds true.\\
Since the notion of \textit{transfer Krull monoid} appears in the next theorem, but is more recent, we want to recall its definition. A monoid homomorphism $\theta: H\to D$ is called a \textit{transfer homomorphism} if it has the following two properties:
\begin{enumerate}
\item[(T1)] $D=\theta(H)D^\times$ and $\theta^{-1}(D^\times)=H^\times$.
\item[(T2)] If $u\in H$, $b,c\in D$ and $\theta(u)=bc$, then there exist $v,w\in H$ such that $u=vw$, $\theta(v)\simeq b$ and $\theta(w)\simeq c$.
\end{enumerate}
Now a monoid $H$ is said to be \textit{transfer Krull} if there exists a transfer homomorphism from $H$ into a Krull monoid.

\smallskip
\begin{theorem}\label{3.14}
Let $G$ be a  group.
\begin{enumerate}
\item The following are equivalent.
	\begin{itemize}
		\item[(a)] $G$ is abelian.
		
		\item[(b)] $\mathcal B(G)$ is Krull.
		
		\item[(c)] $\mathcal B(G)$ is completely integrally closed.
		
		\item[(d)] $\mathcal B(G)$ is root closed.
			
		\item[(e)] $\mathcal B(G)$ is weakly Krull.
			
		\item[(f)] $\mathcal B(G)$ is a transfer Krull monoid.
		
		\item[(g)] $\mathcal B(G) \subseteq \mathcal F(G)$ is saturated.
	\end{itemize}
\item $\mathcal B(G)$ is seminormal if and only if $|G'|\le 2$.

\item $\mathcal B (G)$ is a C-monoid if and only if $\mathcal B (G)$ is a G-monoid if and only if $G$ is finite.
\end{enumerate}
\end{theorem}

\begin{proof}
1. By definition, (a) $\Rightarrow$ (g) $\Rightarrow$ (b) $\Rightarrow$ (c) $\Rightarrow$ (d), (a) $\Rightarrow$ (b) $\Rightarrow$ (e),
	and (a) $\Rightarrow$ (b) $\Rightarrow$ (f) are clear.

\smallskip
(d) $\Rightarrow$ (a) Suppose $\mathcal B(G)$ is root closed and assume to the contrary, that $G$ is non-abelian. Thus there exist $g,h\in G$ such that $gh\neq hg$. If $gh^2=h^2g$ and $gh^3=h^3g$, then $gh^3=h^3g=hgh^2$, whence $gh=hg$, a contradiction. Therefore $gh^2\neq h^2g$ or $gh^3\neq h^3g$. We distinguish three cases.

If $gh^2= h^2g$, then $T= h\bdot (gh^{-1}g^{-1})=\frac{h\bdot g^{-1}\bdot (gh^{-1}g^{-1})\bdot g}{g\bdot g^{-1}}\in \mathsf q(\mathcal{B}(G))\setminus \mathcal{B}(G)$, but $T^{[2]}=h^{[2]}\bdot (gh^{-1}g^{-1})\bdot (gh^{-1}g^{-1})\in\mathcal{B}(G)$, a contradiction to our assumption that $\mathcal B(G)$ is root closed.

If  $gh^3= h^3g$, then $T= h\bdot (gh^{-1}g^{-1})\in \mathsf q(\mathcal{B}(G))\setminus \mathcal{B}(G)$, but $T^{[3]}=h^{[3]}\bdot (gh^{-1}g^{-1})\bdot (gh^{-1}g^{-1})\bdot (gh^{-1}g^{-1})\in\mathcal{B}(G)$, a contradiction to our assumption that $\mathcal B(G)$ is root closed.

If $gh^2\neq h^2g$ and $gh^3\neq h^3g$, then $$T= g\bdot (h^2g^{-1}h^{-2})\bdot h\bdot h^{-1}=\frac{g\bdot h^{-1}\bdot h^{-1}\bdot (h^2g^{-1}h^{-2})\bdot h\bdot h}{h\bdot h^{-1}}\in \mathsf q(\mathcal{B}(G))\setminus\mathcal{B}(G)\,,$$ but $T^{[2]}=g^{[2]}\bdot h^{-1}\bdot h^{-1}\bdot h^2g^{-1}h^{-2}\bdot h^2g^{-1}h^{-2} \bdot h\bdot h \in\mathcal{B}(G)$,
 a contradiction to our assumption that $\mathcal B(G)$ is root closed.

\medskip	
(e) $\Rightarrow$ (a)	Suppose $\mathcal B(G)$ is a weakly Krull monoid. We claim that $\mathcal B(G)^*=\mathcal B(G)$. If this holds, then for every element $g\in G'$, the sequence $g$ is a product-one sequence. Therefore $G'$ is trivial and hence $G$ is abelian.
 In fact we will show  \begin{equation*}
\mathcal B (G)\subseteq\mathcal B(G)^{\ast}\subseteq\bigcap_{{\mathfrak p} \in \mathfrak X (\mathcal B (G))} \mathcal B (G)_{\mathfrak p}=\mathcal B(G).
\end{equation*}	
In view of Proposition \ref{3.8}, it suffices to prove  $\mathcal B(G)^{\ast} \subseteq  \mathcal B(G)_{\mathfrak p_x}$ for all  $x\in G$.

 We first show  $ghg^{-1}h^{-1}\in \mathcal B(G)_{{\mathfrak p}_x}$ for all $x\in G$ and $g,h\in G$. 	Let $g, h \in G$ and $f = ghg^{-1}h^{-1}$. If $f=1$, then $f\in \bigcap_{x\in G} \mathcal B(G)_{\mathfrak p_x}$. Suppose $f\neq 1$. Then $\{hg, g^{-1}h^{-1}\}\cap \{h,h^{-1}, g,g^{-1}\}=\emptyset$.
 Since
 \begin{align*}
 f=\frac{f\bdot hg\bdot g^{-1}h^{-1}}{hg\bdot g^{-1}h^{-1}}=\frac{f\bdot g\bdot g^{-1}\bdot h\bdot h^{-1}}{g\bdot g^{-1}\bdot h\bdot h^{-1}} &\in\mathsf q(\mathcal B(G))\,,\\
 \text{ and }\quad f\bdot hg\bdot g^{-1}h^{-1},\quad f\bdot g\bdot g^{-1}\bdot h\bdot h^{-1}&\in  \mathcal B(G)\,,
 \end{align*} we know $f\in \mathcal B(G)_{{\mathfrak p}_x}$ for all $x\in G$.

 Next, we show that  $f\in \mathcal B(G)_{\mathfrak p_x}$ for all $x\in G$, where  $f\in G'$. Let $f\in G'=[ghg^{-1}h^{-1}:g,h\in G]$ and let $x\in G$. Then there exist $k\in \N$ and $f_1,\ldots, f_k\in \{ghg^{-1}h^{-1}: g,h\in G\}$ such that $f=f_1\ldots f_k$.  For each $i\in [1,k]$, there exists $S_i\in \mathcal B(G)\setminus \mathfrak p_x$ such that $f_i\bdot S_i\in \mathcal B(G)$, whence $S=S_1\bdot \ldots\bdot S_k\in \mathcal B(G)\setminus \mathfrak p_x$ and $f\bdot S\in \mathcal B(G)$. Therefore $f=(f\bdot S)/S\in \mathcal B(G)_{\mathfrak p_x}$.

 Finally we prove  $\mathcal B(G)^{\ast} \subseteq  \mathcal B(G)_{\mathfrak p_x}$ for all $x\in G$.
  Let $x\i G$ and $A\in \mathcal F(G)$ with $\pi(A)\subseteq G'$ be arbitrary. For any $f\in \pi(A)\subseteq G'$, we have that $f\in \mathcal B(G)_{\mathfrak p_x}$ as shown in the previous prargraph. Then there exists $S\in \mathcal B(G)\setminus \mathfrak p_x$ such that $f\bdot S\in \mathcal B(G)$, which implies that $A\bdot S\in \mathcal B(G)$ and hence $A=(A\bdot S)/S\in \mathcal B(G)_{\mathfrak p_x}$.

\medskip	
(f) $\Rightarrow$ (a) Suppose $\mathcal B(G)$ is a transfer Krull monoid. Then there is a transfer homomorphism $\theta\colon \mathcal B(G)\rightarrow \mathcal B(G_0)$, where $G_0\subseteq G_1$ and $G_1$ is an abelian group.
	 Assume to the contrary that $G$ is non-abelian. Then there exist $g,h\in G$ such that $gh\neq hg$.
	
	Suppose $gh^2=h^2g$. Let $T=h\bdot gh^{-1}g^{-1}$ and $W=g\bdot g^{-1}$. Then $T\in \mathsf q(\mathcal B(G))\setminus \mathcal B(G)$,  $T^{[2]}=h^{[2]}\bdot gh^{-1}g^{-1}\bdot gh^{-1}g^{-1}\in \mathcal B(G)$, and  $W$, $W\bdot T$ are atoms of $\mathcal B(G)$. It follows that $\theta(W)$ and $\theta(W\bdot T)$ are atoms of $\mathcal B(G_0)$. Note that $\theta(W\bdot T)^2=\theta(W^{[2]}\bdot T^{[2]})=\theta(W)^2\theta(T^{[2]})$. Then $\theta(W)$ divides $\theta(W\bdot T)$ in $\mathcal B(G_0)$, a contradiction to the fact that  $\theta(W)$ and $\theta(W\bdot T)$ are both atoms.

	Suppose $gh^2\neq h^2g$ and $gh^3=h^3g$. Let $T=h\bdot gh^{-1}g^{-1}$ and $W=g\bdot g^{-1}$. Then $T\in \mathsf q(\mathcal B(G))\setminus \mathcal B(G)$, $T^{[3]}=h^{[3]}\bdot gh^{-1}g^{-1}\bdot gh^{-1}g^{-1}\bdot gh^{-1}g^{-1}\in \mathcal B(G)$, and  $W$, $W\bdot T$ are atoms of $\mathcal B(G)$. It follows that  $\theta(W)$ and $\theta(W\bdot T)$ are atoms of $\mathcal B(G_0)$. Note that $\theta(W\bdot T)^{[3]}=\theta(W^{[3]}\bdot T^{[3]})=\theta(W)^{[3]}\theta(T^{[3]})$. Then $\theta(W)$ divides $\theta(W\bdot T)$ in $\mathcal B(G_0)$, a contradiction to the fact that $\theta(W)$ and $\theta(W\bdot T)$ are both atoms.

	Suppose $gh^2\neq h^2g$ and $gh^3\neq h^3g$.
	Let $T=g\bdot h^2g^{-1}h^{-2}\bdot h\bdot h^{-1}$ and $W=h\bdot h^{-1}$. Then $T\in \mathsf q(\mathcal B(G)) \setminus \mathcal B(G)$,  $T^{[2]}=g^{[2]}\bdot (h^{-1})^{[2]}\bdot h^2g^{-1}h^{-2}\bdot h^2g^{-1}h^{-2}\bdot h^{[2]}\in \mathcal B(G)$, and  $W$, $W\bdot T$ are atoms of $\mathcal B(G)$. It follows that $\theta(W)$ and $\theta(W\bdot T)$ are atoms of $\mathcal B(G_0)$. Note that $\theta(W\bdot T)^{[2]}=\theta(W^{[2]}\bdot T^{[2]})=\theta(W)^{[2]}\theta(T^{[2]})$. Then $\theta(W)$ divides $\theta(W\bdot T)$ in $\mathcal B(G_0)$, a contradiction to  the fact that $\theta(W)$ and $\theta(W\bdot T)$ are both atoms.
	
\medskip

2. The proof for $(\Leftarrow)$ is the same as in \cite[Theorem 3.10]{Oh20a}.  It remains to prove
	$(\Rightarrow)$. Suppose $\mathcal B(G)$ is seminormal. Assume to the contrary that $|G'|\ge 3$. To get a contradiction, we will show that there exists a sequence $T\in \mathsf  q(\mathcal B(G))\setminus \mathcal B(G)$ such that $T^{[2]}, T^{[3]}\in \mathcal B(G)$.
	
 Since $G$ is non-abelian,  there exist $g,h\in G$ such that $gh\neq hg$. If $gh^2=h^2g$ and $gh^3=hg^3$, then $gh=hg$, a contradiction. Therefore either $gh^2\neq h^2g$ or  $gh^3\neq hg^3$.
  We distinguish two  cases.

\medskip

\noindent\textbf{Case 1:}  There exist $g,h\in G$ such that $gh\neq hg$ and $gh^2\neq h^2g$.
\smallskip

	If   $gh^3\neq h^3g$,  then let $T=g\bdot h^2g^{-1}h^{-2}\bdot h\bdot h^{-1}$. Therefore  $T\in \mathsf q(\mathcal B(G)) \setminus \mathcal B(G)$ and \begin{align*}
	T^{[2]}&=g^{[2]}\bdot (h^{-1})^{[2]}\bdot h^2g^{-1}h^{-2}\bdot h^2g^{-1}h^{-2}\bdot h^{[2]}\in \mathcal B(G)\,,\\
 T^{[3]}&=g^{[3]}\bdot (h^{-1})^{[2]}\bdot (h^2g^{-1}h^{-2})^{[3]}\bdot h^{[2]}\bdot h\bdot h^{-1}\in \mathcal B(G)\,.
		\end{align*}

Now we may assume $gh^3 =h^3g$.
	If $g^2h= hg^2$, then let
		$T=hgh^2\bdot h^{-1}g^{-1}\bdot h^{-2}$. Then $T\in \mathsf q(\mathcal B(G)) \setminus \mathcal B(G)$ and \begin{align*}
		T^{[2]}&=hgh^2\bdot hgh^2\bdot h^{-1}g^{-1}\bdot h^{-2}\bdot h^{-1}g^{-1}\bdot h^{-2}\in \mathcal B(G)\,,\\
	 T^{[3]}&=hgh^2\bdot (h^{-2})^{[2]}\bdot (h^{-1}g^{-1})^{[2]}\bdot h^{-2}\bdot h^{-1}g^{-1}\bdot (hgh^2)^{[2]}\in \mathcal B(G)\,.
		\end{align*}	
	If $g^2h\neq hg^2$ and $g^3h\neq hg^3$, then let $T=h\bdot g^2h^{-1}g^{-2}\bdot g\bdot g^{-1}$. Therefore  $T\in \mathsf q(\mathcal B(G)) \setminus \mathcal B(G)$ and \begin{align*}
	T^{[2]}&=h^{[2]}\bdot (g^{-1})^{[2]}\bdot g^2h^{-1}g^{-2}\bdot g^2h^{-1}g^{-2}\bdot g^{[2]}\in \mathcal B(G)\,,\\
	T^{[3]}&=h^{[3]}\bdot (g^{-1})^{[2]}\bdot (g^2h^{-1}g^{-2})^{[3]}\bdot g^{[2]}\bdot g\bdot g^{-1}\in \mathcal B(G)\,.
	\end{align*}
Suppose $g^2h\neq hg^2$ and $g^3h= hg^3$.	
If $\ord(ghg^{-1}h^{-1})=2$, then let $T=ghg^{-1}\bdot h^{-1}$. Then $T\in \mathsf q(\mathcal B(G))\setminus \mathcal B(G)$,
\begin{align*}
T^{[2]}&=(ghg^{-1}\bdot h^{-1})^{[2]}\in \mathcal B(G)\,,\\
\text{ and }\quad T^{[3]}&=(ghg^{-1})^{[3]}\bdot (h^{-1})^{[3]}\in \mathcal B(G)\,.
\end{align*}
If $ghgh=hghg$, then let
	 $T=hg\bdot h^{-1}g^{-1}h\bdot h^{-1}$. Then $T\in \mathsf q(\mathcal B(G))\setminus \mathcal B(G)$, \begin{align*}
	 T^{[2]}&=(hg)^{[2]}\bdot (h^{-1}g^{-1}h\bdot h^{-1})^{[2]}\in \mathcal B(G)\,,\\
	 \text{ and }\quad
			T^{[3]}&=(h^{-1}\bdot hg)^{[3]}\bdot (h^{-1}g^{-1}h)^{[3]}\in \mathcal B(G)\quad (\text{note that } g^3h= hg^3)\,.
			\end{align*}		
		If $ghgh=h^2g^2$, then
			let $T=gh\bdot g^{-1}h^{-1}g\bdot g^{-1}$. Then $T\in \mathsf q(\mathcal B(G))\setminus \mathcal B(G)$, \begin{align*}
			T^{[2]}&=(gh)^{[2]}\bdot g^{-1}\bdot (g^{-1}h^{-1}g)^{[2]}\bdot g^{-1}\in \mathcal B(G)\,,\\
			\text{ and }\quad
			T^{[3]}&=(g^{-1}\bdot gh)^{[3]}\bdot (g^{-1}h^{-1}g)^{[3]}\in \mathcal B(G) \quad (\text{note that } gh^3= h^3g)\,.
			\end{align*}		
			If $\ord(ghg^{-1}h^{-1})\neq 2$,  $ghgh\neq h^2g^2$, and $ghgh\neq hghg$, then let
	 $T=ghg^{-1}\bdot h^{-1}\bdot  h^{-1}gh\bdot g^{-1}$. Then $T\in \mathsf q(\mathcal B(G))\setminus \mathcal B(G)$,
	  \begin{align*}
	T^{[2]}&=g^{-1}\bdot (ghg^{-1})^{[2]}\bdot g^{-1}\bdot (h^{-1})^{[2]}\bdot (h^{-1}gh)^{[2]}\in \mathcal B(G)\,,\\
	\text{ and }\quad
T^{[3]}&=(ghg^{-1})^{[3]}\bdot (h^{-1})^{[3]}\bdot  (h^{-1}gh)^{[3]}\bdot( g^{-1})^{[3]}\in \mathcal B(G)\,.
	  \end{align*}
			
		\medskip

\noindent\textbf{Case 2:} For all $g,h\in G$ with $gh\neq hg$, we have that $gh^2=hg^2$.

\smallskip
	
	Then $G_1=\langle h^2\colon  h\in G\rangle\subseteq \mathsf Z(G)=\{g\colon gh=hg\text{ for every }h\in G\}$, whence $G_1$ is a normal subgroup of $G$. Since every element of $G/G_1$ has order at most $2$, we obtain that $G/G_1$ is abelian and hence $G'\subseteq G_1\subseteq \mathsf Z(G)$. Therefore $G'$ is an abelian group.  Let $g,h\in G$ such that $gh\neq hg$. Since $G'\subseteq \mathsf{Z}(G)$ it follows that
	$$ghg^{-1}h^{-1}ghg^{-1}h^{-1}=gghg^{-1}h^{-1}hg^{-1}h^{-1}=1\,,$$
	whence $\ord(ghg^{-1}h^{-1})=2$ and $G'$ is an elementary $2$-group. Therefore $|G'|\ge 4$ and
we proceed by the following claim.
	
	\medskip
	\noindent{\bf Claim A. }{\it   There exist distinct elements $g,h,f\in G$ such that the three elements $1$, $ghg^{-1}h^{-1}$, $gfg^{-1}f^{-1}$ are distinct and  $hfh^{-1}f^{-1}\not\in \{1, ghg^{-1}h^{-1}, ghg^{-1}h^{-1}gfg^{-1}f^{-1}\}$.
	}
	\smallskip

	\begin{proof}[Proof of Claim A]
		Since $G$ is  non-abelian, there exist $g,h\in G$ such that $ghg^{-1}h^{-1}\neq 1$. Since $|G'|\ge 4$, there exists $x\in G'$ such that $x\not\in \{1, ghg^{-1}h^{-1}\}$. Note that $x\in \langle yzy^{-1}z^{-1}\colon y,z\in G\rangle $. Then there exist $u,v\in G$ such that $uvu^{-1}v^{-1}\not\in \{1, ghg^{-1}h^{-1}\}$, whence it follows that $|\{g,h\}\cap \{u,v\}|\le 1$ by applying Lemma \ref{3.13}. From now on Lemma \ref{3.13} will be used frequently without furhter mention for the rest of the argument. We distinguish two cases.
		
		Suppose $|\{g,h\}\cap \{u,v\}|= 1$.
		 By symmetry we may assume $g=u$.
		If $hvh^{-1}v^{-1}\not\in \{1, ghg^{-1}h^{-1}, ghg^{-1}h^{-1}gvg^{-1}v^{-1}\}$, then $g,h,v$ are the required three elements.  If $hvh^{-1}v^{-1}=ghg^{-1}h^{-1}$, then $g,v,h$ are the required three elements.
		 If $hvh^{-1}v^{-1}=ghg^{-1}h^{-1}gvg^{-1}v^{-1}$, then $g(gv)g^{-1}(gv)^{-1}=gvg^{-1}v^{-1}$ (note that $G'\subseteq \mathsf Z(G)$) and $h(gv)h^{-1}(gv)^{-1}=hvh^{-1}v^{-1}ghg^{-1}h^{-1}=gvg^{-1}v^{-1}$, whence $g,h,gv$ are the required three elements. If $hvh^{-1}v^{-1}=1$, then $g(gv)g^{-1}(gv)^{-1}=gvg^{-1}v^{-1}$ and $h(gv)h^{-1}(gv)^{-1}=hvh^{-1}v^{-1}ghg^{-1}h^{-1}=ghg^{-1}h^{-1}$, whence $gv,g,h$ are the required three elements.

			Suppose $|\{g,h\}\cap \{u,v\}|= 0$. If there exist $x\in \{g,h\}$ and $y\in \{u,v\}$ such that $xyx^{-1}y^{-1}\neq 1$, then by symmetry we may assume $x=g$ and $y=u$. If $gug^{-1}u^{-1}\neq ghg^{-1}h^{-1}$, then consider the elements $g,h,u$ and go back to the previous case. If $gug^{-1}u^{-1}=ghg^{-1}h^{-1}$, then $gug^{-1}u^{-1}\neq uvu^{-1}v^{-1}$, consider the elements $u, g, v$, and  go back to the previous case. Therefore we may assume that the elements from $\{g,h\}$ commute with those from $\{u,v\}$, whence
			\begin{align*}
			(gu)h(gu)^{-1}h^{-1}=&guhu^{-1}g^{-1}h^{-1}=uu^{-1}ghg^{-1}h^{-1}=ghg^{-1}h^{-1}\,,\\
			(gu)v(gu)^{-1}v^{-1}=&guvu^{-1}g^{-1}v^{-1}=gg^{-1}uvu^{-1}v^{-1}=uvu^{-1}v^{-1}\,.
			\end{align*}
			  Now consider the elements $gu, h, v$ and go back to the previous case.
			  
		\qedhere[End of proof of Claim A.]
		\end{proof}
	
	By Claim A,  we can choose distinct $g,h,f\in G$ such that the three elements $1, ghg^{-1}h^{-1}, gfg^{-1}f^{-1}$ are distinct and $hfh^{-1}f^{-1}\not\in \{1, ghg^{-1}h^{-1}, ghg^{-1}h^{-1}gfg^{-1}f^{-1}\}$.
	Consider the sequence $T=f^{-1}g^{-1}h^{-1}g^{-1}f^{-1}\bdot g\bdot gf\bdot fh$.  To simplify the notation, we set $e_1=ghg^{-1}h^{-1}$, $e_2=gfg^{-1}f^{-1}$,  $e_3=hfh^{-1}f^{-1}$, and $q=f^{-1}g^{-1}h^{-1}g^{-1}f^{-1}$.
	Note that $G'\subseteq \langle h^2\colon h\in G\rangle\subseteq \mathsf Z(G)$.
	We consider the following products.
	\begin{align*}
	\textcircled{1}&\quad && qg(gf)(fh)=qg^2f^2h=f^{-1}g^{-1}h^{-1}gfh=f^{-1}g^{-1}h^{-1}ghh^{-1}fh=[g,h][f,h]=e_1e_3\,,\\
	\textcircled{2}&&& qg(fh)(gf)=qggffh(fh)^{-1}(gf)^{-1}fhgf=[g,h][f,h][fh,gf]=[f,g]=e_2\,,\\
	\textcircled{3}&&& q(fh)g(gf)=qgfh(fh)^{-1}g^{-1}fhggf=[f,g][fh,g]=[h,g]=e_1\,,\\
	\textcircled{4}&&& q(fh)(gf)g=q(fh)g(gf)(gf)^{-1}g^{-1}gfg=[h,g][gf,g]=[h,g][f,g]=e_1e_2\,,\\
	\textcircled{5}&&& q(gf)(fh)g=q(fh)(gf)(gf)^{-1}(fh)^{-1}gffhg=[h,g][f,g][gf,fh]=[f,h]=e_3\,,\\
	\textcircled{6}&&& q(gf)g(fh)=[gf,h][g,f]=[f,h][f,g][g,h]=e_1e_2e_3\,.
	\end{align*}
	It follows that $\pi(T)\supseteq\{e_1, e_2, e_3, e_1e_2, e_1e_3, e_1e_2e_3\}$, whence $1\in \pi(T^{[2]})$ and $1\in \pi(T^{[3]})$. Since  $T\in \mathsf q(\mathcal B(G))$, it suffices to show $T\not\in \mathcal B(G)$. Assume to contrary that $1\in \pi(T)$. Then $1\in \{e_1, e_2, e_3, e_1e_2, e_1e_3, e_1e_2e_3\}$, a contradiction to the choice of the three elements $g,h,f$.
	
	\medskip
3. Theorem \ref{3.11} implies that $\mathcal B(G)$ is a G-monoid if and only if $G$ is finite. If $G$ is finite, then $\mathcal B(G)$ is a C-monoid by \cite[Theorem 3.2.1]{Cz-Do-Ge16a}, and the converse follows by Proposition \ref{3.9}.2.
\end{proof}

\section{Arithmetic properties of $\mathcal B(G)$ and $\mathcal B(G_0)$} \label{4}

In this section, we study arithmetic properties of monoids of product-one sequences. Theorem \ref{3.11} provides conditions ensuring that monoids of product-one sequences are finitely generated resp. C-monoids. The arithmetic of such monoids is well understood (see \cite[Theorems 3.1.5, 3.3.4, 4.4.11, and 4.6.6]{Ge-HK06a}), and  our goal is to obtain results beyond these classes of monoids (see Proposition \ref{4.2} and Theorem \ref{main2}). In Section \ref{5}, we get more precise results for infinite dihedral groups.

We recall some concepts from factorization theory (for details see \cite[Chapter 1]{Ge-HK06a}), and for simplicity we do this in the setting of reduced atomic monoids. Let $H$ be a reduced atomic monoid and let  $\mathcal A(H)$ be its set of atoms.
Consider the free abelian monoid $\mathsf Z(H)=\mathcal F(\mathcal A(H))$ with the epimorphism $\pi:\mathsf Z(H)\to H$ via $\pi(u)=u$ for all $u\in \mathcal A(H)$.
For $a\in H$,
\begin{itemize}
\item $\mathsf Z(a)=\pi^{-1}(\{a\})$ is the \textit{set of factorizations} of $a$, and

\item  $\mathsf L(a)=\{|z|\colon z\in\mathsf Z(a)\}$ is  the \textit{set of lengths}  of $a$.
\end{itemize}
Then $H$ is said to be a \BF- (resp. FF-) monoid, if  $\mathsf L(a)$ (resp. $\mathsf Z(a)$) is finite for all $a\in H$. We consider the
system
$\mathcal L(H)=\{\mathsf L(a)\colon a\in H\}$ of all sets of lengths of $H$. We need further invariants, describing its structure. The \textit{set of distances} of $H$ is $\Delta(H)=\bigcup_{L\in\mathcal L(H)}\Delta(L)$, where $\Delta(L)=\{d\in\N\colon \text{there is } l\in L \text{ such that } L\cap [l,l+d]=\{l,l+d\}\}$. For  $k\in \N$, we define $\mathcal U_k(H)=\bigcup_{k\in L\in\mathcal L(H)} L$ and set $\rho_k(H)=\sup\mathcal U_k(H)$ and $\lambda_k(H)=\min\mathcal U_k(H)$. The \textit{elasticity} of $H$ is $\rho(H)=\sup\{\frac{\sup(L)}{\min(L)}\colon L\in\mathcal L(H)\}$.

For $a\in H$, we set $\omega(H,a)=\omega(a)$ to be the smallest $N\in\N\cup\{\infty\}$ with the property that for all $n\in\N$ and all $a_1,\hdots ,a_n\in H$ we have that whenever $a\mid \prod_{i=1}^{n}a_i$, there is $\Omega\subseteq [1,n]$ with $|\Omega|\leq N$ such that $a\mid \prod_{i\in\Omega}a_i$.  We denote by $\omega(H)=\sup\{\omega(H,u)\colon u\in\mathcal A(H)\}$.

Now let $u\in \mathcal A(H)$. We define $\tau(H,u)=\tau(u)$ to be the smallest $N\in\N\cup\{\infty\}$ with the property that for all $n\in\N$ and all $a_1,\hdots ,a_n\in \mathcal A(H)$ with $u\t a_1\cdot \ldots\cdot a_n$ and $u\nmid \prod_{i\in [1,n]\setminus\{j\}}a_i$ for every $j\in [1,n]$, 
we have that $\min\mathsf L(a_1\cdot\ldots\cdot a_n\cdot u^{-1})\le N$. We denote by $\tau(H)=\sup\{\tau(H,u)\colon u\in\mathcal A(H)\}$. 
 We denote by $\mathsf t(H,u)=\mathsf t(u)$ the smallest $N\in\N_0\cup\{\infty\}$ with the property that if $a\in H$ and there is a factorization of $a$, where $u$ occurs, then for every factorization $z$ of $a$ there is a factorization $z'$ of $a$ such that $u$ occurs in $z'$ and $\mathsf d(z,z')\leq N$, where $\mathsf d(z,z')=\max\{|z|, |z'|\}-|\gcd(z,z')|$. As before, we set $\mathsf t(H)=\sup\{\mathsf t(H,u)\colon u\in\mathcal A(H)\}$. $H$ is said to be \textit{locally tame} if $\mathsf t(H,u)$ is finite for all $u\in\mathcal A(H)$, and \textit{tame} if $\mathsf t(H)$ is finite.

Let $a\in H$ and $N\in\N_0$. A finite sequence $z_0,\ldots,z_k\in\mathsf Z(a)$ is called an \textit{$N$-chain of factorizations} if $\mathsf d(z_{i-1},z_i)\leq N$ for all $i\in[1,k]$. We denote by $\mathsf c_H(a)=\mathsf c(a)$ the smallest $N\in\N_0\cup\{\infty\}$ such that for any two factorizations $z,z'\in\mathsf Z(a)$ there is an $N$-chain of factorizations from $z$ to $z'$. The \textit{catenary degree} of $H$ is defined to be $\mathsf c(H)=\sup\{\mathsf c_H(a)\colon a\in H\}$.
Let $G$ be a group and $G_0\subseteq G$ a subset. As is convenient, instead of $\ast(\mathcal B(G_0))$ we write $\ast(G_0)$ for the above explained invariants $\ast$.

\smallskip
By \cite[Theorem 1.6.3]{Ge-HK06a} and by \cite[Propositions 3.5 and 3.6]{Ge-Ka10a}, we have
\begin{equation} \label{basic}
\rho (G_0) \le \omega (G_0) \quad \text{and} \quad \sup \Delta (G_0) \le \mathsf c (G_0) \le \omega (G_0) \le \mathsf t (G_0) \le \omega (G_0)^2 \,,
\end{equation}
and  by \cite[Lemma 3.5]{Ge-Ha08}, for all $A\in \mathcal A(G_0)$, we have that
\begin{equation} \label{basic1}
\mathsf t(G_0,A)=\max\{\omega(G_0,A), 1+\tau(G_0,A)\}\,, \text{ whence }\mathsf t(G_0)=\max\{\omega(G_0), 1+\tau(G_0)\}\,.
\end{equation}

In particular, if $\omega (G_0) < \infty$, then the set of distances $\Delta (G_0)$ is finite.

\smallskip
\begin{proposition} \label{4.1}
Let $G$ be a group and $G_0 \subseteq G$ a condensed subset. Then $\mathcal B (G_0)$ is an \FF-monoid and we have
\begin{enumerate}
\item For all $k \in \N$, $\rho_k (G_0) \le k \mathsf D (G_0)/2$ and $\rho (G_0) \le \mathsf D (G_0)/2$.

\item If $G_0$ is finite and $\mathsf D (G_0) < \infty$, then the elasticity $\rho (G_0)$ is accepted and $\omega (G_0) < \infty$.

\item If $\langle G_0 \rangle$ is abelian and $\mathsf D (G_0) < \infty$, then $\omega (G_0) < \infty$.
\end{enumerate}
\end{proposition}

\begin{proof}
Since $\mathcal B (G_0)$ is a submonoid of the free abelian monoid $\mathcal F (G_0)$, it is an \FF-monoid by \cite[Corollary 1.5.7]{Ge-HK06a}.

1. The argument is the same as  when $G$ is abelian. Since it is short,  we give it here. Let $k \in \N$ and $U_1, \ldots, U_k, V_1, \ldots, V_{\ell} \in \mathcal A (G_0)$ such that $U_1 \bdot \ldots \bdot U_k = V_1 \bdot \ldots \bdot V_{\ell}$, where $\ell\in \N$ with $\ell\geq k$. If $|V_i|=1$ for some $i \in [1,\ell]$, then $V_i$ is a prime element of $\mathcal B (G_0)$ and hence there exists $j\in [1,k]$ such that $U_j=V_i$. Set $I=\{i\in [1,\ell]\colon |V_i|=1\}$ and $J=\{j\in [1,k]\colon |U_j|=1\}$.
It follows that
\[
2 (\ell-|I|)+|I| \le \sum_{i =1}^{\ell} |V_i| = \sum_{j=1}^k |U_j| \le  |J|+ \mathsf D (G_0) (k-|J|) \le |I|+\mathsf D(G_0)(k-|I|)\,,
\]
whence $\frac{\ell}{k}\le \frac{\ell-|I|}{k-|I|}\le \mathsf D (G_0)/2$ and
 $\rho_k (G_0) \le k \mathsf D (G_0)/2$.
Since, by \cite[Proposition 1.4.2]{Ge-HK06a},
\[
\rho (G_0) = \lim_{k \to \infty} \frac{\rho_k (G_0)}{k} \,,
\]
we obtain the upper bound for $\rho (G_0)$.

2. If $G_0$ is finite and $\mathsf D (G_0) < \infty$, then $\mathcal B (G_0)$ is finitely generated by Theorem \ref{3.11}. Thus the claim follows from \eqref{basic} and from \cite[Theorem 3.1.4]{Ge-HK06a}.

3. Suppose that $\langle G_0 \rangle$ is abelian. Then it is easy to see that $\omega (G_0, U) \le |U| \le \mathsf D (G_0)$ for all $U \in \mathcal A (G_0)$ (details can be found in \cite[Chapter 3.4]{Ge-HK06a}).
\end{proof}

Let $G_0$ be a  condensed subset of a group. If $G_0$ consists of finitely many torsion elements, then $\mathsf D (G_0) < \infty$ by Theorem \ref{3.11}. Examples  \ref{taualpha} and \ref{infinitefinitary} provide finite subsets with infinite Davenport constant.

\smallskip
To describe the structure of sets of lengths, we need the concept of almost arithmetical (multi)progressions.
Let $M \in \N_0$, $d \in \N$ and $\mathcal D \subseteq [0, d]$ with $\{0, d \} \subseteq \mathcal D$. A  subset $L \subseteq \Z$ is said to be an
\begin{itemize}
\item {\it almost arithmetical multiprogression} (AAMP) with difference $d$, period $\mathcal D$, and bound $M$ if
    \[
    L = y + (L' \cup L^* \cup L'') \subseteq y + \mathcal D + d \Z \quad \text{is finite} \,,
    \]
    where $\min L^* =0$, $L^* = (\mathcal D + d \Z) \cap [0, \max L^*]$, $L' \subseteq [-M, -1]$, and $L'' \subseteq \max L^* + [1, M]$, and $y \in \Z$. $L'$ resp. $L''$ is called the initial resp. the end part of $L$.

\item {\it almost arithmetical progression} (AAP) with difference $d$ and bound $M$ if
    \[
    L = y + (L' \cup L^* \cup L'') \subseteq y  + d \Z  \,,
    \]
    where $L^*$ is a nonempty arithmetical progression with difference $d$ such that $\min L^*=0$, $L' \subseteq [-M, -1]$, and $L'' \subseteq \sup L^* + [1, M]$ (with the convention that $L'' = \emptyset$ if $L^*$ is infinite), and $y \in \Z$.
\end{itemize}
Thus, if $L$ is finite, then it is an AAP with difference $d$ and bound $M$  if and only if if it is an AAMP with difference $d$, bound $M$, and period $\{0, d\}$.

\smallskip
\begin{proposition} \label{4.2}
Let $G$ be a group and let $G_0 \subseteq G$ be a condensed subset with $\omega (G_0) < \infty$.
\begin{enumerate}
\item There is $M \in \N_0$ such that every $L \in \mathcal L (G_0)$ is an \AAMP \ with difference $d \in \Delta (G_0)$ and bound $M$.

\item There is $M \in \N_0$ such that, for every $k \in \N$,  $\mathcal U_k (H)$ is an \AAP \ with difference $\min \Delta (G_0)$ and bound $M$. Moreover, if $G_0$ is finite and $\mathsf D (G_0) < \infty$, then the initial and end parts of the sets $\mathcal U_k (G_0)$ repeat periodically.
\end{enumerate}
\end{proposition}

\begin{proof}
1. This follows from Proposition \ref{4.1} and from \cite[Theorem 5.1]{Ge-Ka10a}.

2. The first statement  follows from  \cite[Theorems 3.5 and 4.2]{Ga-Ge09b}. Suppose that $G_0$ is finite and $\mathsf D (G_0) < \infty$. Then Proposition \ref{4.1} implies that $\omega (G_0) < \infty$ and that the elasticity $\rho (G_0)$ is accepted. Therefore,    the statement follows from \cite[Theorem 1.2]{Tr19a}.
\end{proof}

For every positive integer $d \in \N$, there are a group $G$ and a subset $G_0 \subseteq G$ with $\min \Delta (G_0)=d$. Furthermore, the initial and end parts of the sets $\mathcal U_k (G_0)$ are non-trivial. Our next goal is to show that if $G_0=G$ is the whole group, then $\min \Delta (G)=1$ and the initial and end parts are empty.

\smallskip
\begin{lemma}\label{4.4}
If $G$ is an infinite group, then $\mathcal B(G)$ has atoms of every length. In particular, $\mathsf D(G)=\infty$.
\end{lemma}

\begin{proof}
Suppose $G$ is infinite. Then there exists an infinite sequence $(g_i)_{i=1}^{\infty}$ with terms from $G$
	such that $g_1\neq 1_G$ and for every $j\in \N$, we have that $g_{j+1}^{-1}\not\in \pi(S^{(j)})$, where $S^{(j)}$ ranges over all subsequences of $g_1\bdot\ldots\bdot g_j$. It follows that $g_1\bdot\ldots\bdot g_j$ has no product-one subsequence for every $j\in \N$.
	Note that $1_G$ is an atom of length $1$. Let $m\in \N$, $h_m=g_1^{-1}\ldots g_m^{-1}$, and  $W_m=g_1\bdot\ldots\bdot g_m\bdot h_m$. Then $W_m$ is a product-one sequence of length $m+1$. It suffices to show $W_m$ is an atom.
 Suppose $W=V_1\bdot V_2$ with $h_m\in \supp(V_2)$, where $V_1,V_2\in \mathcal B(G)$. Then $V_1$ is a subsequence of $g_1\bdot\ldots\bdot g_m$ and hence $V_1$ is empty. It follows that $W$ is an atom of length $m+1$ and we are done.
\end{proof}

The proof of the first part of the first statement in the following theorem runs along the same lines as in \cite[Theorem 5.5]{Oh20a}, whose origin is \cite[Theorem 3.1.3]{Ge09a} in the abelian setting. Since the proof is not long, we give it.

\smallskip
\begin{theorem} \label{main2}
Let $G$ be a group.
\begin{enumerate}
\item Let $k \in \N_{\ge 2}$.  We have $\mathcal U_k(G) = [\lambda_k(G),\rho_k(G)]$. Furthermore, if $G$ is infinite, then $\mathcal U_k (G) = \N_{\ge 2}$.

\item If $G$ has an element of infinite order or the orders of the abelian subgroups are unbounded, then $\mathcal L(G)=\{L\subseteq \N_{\geq 2}\colon L \text{ finite and nonempty}\}\cup \{\{0\},\{1\}\}$ and hence $\Delta(G)=\N$.
\end{enumerate}
\end{theorem}

\begin{proof}
1. If $\t G\t\leq 5$, then $G$ is an abelian group and the statement is known (either \cite[Theorem 5.5]{Oh20a} or \cite[Theorem 3.1.3]{Ge09a}).  Suppose that $\t G\t\geq 6$. We need to show $[\lambda_k(G),\rho_k(G)]\subseteq \mathcal U_k(G)$ for all $k\in \N$. We assert  $[k,\rho_k(G)]\subseteq \mathcal U_k(G)$ for all $k\in \N$.  Suppose the assertion holds. Let $k\in \N$ and let $l\in[\lambda_k(G),k]$. Then $l\le k\le  \rho_{\lambda_k(G)}(G)\le \rho_l(G)$. It follows by the assertion that $k\in\mathcal U_l(G)$ and consequently $l\in\mathcal U_k(G)$. Therefore $[\lambda_k(G), \rho_k(G)]\subseteq \mathcal U_k(G)$.

Thus we only need to show the assertion. Let $k\in \N$ and let $\ell\in [k, \rho_k(G)]$ be minimal such that $[\ell,\rho_k(G)]\subseteq \mathcal U_k(G)$.  Assume to the contrary that $k<\ell$. Then $\ell-1\not\in \mathcal U_k(G)$ and $k+2\le \ell$. Define
 $$\Omega=\{A\in\mathcal B(G)\colon \text{ there exists }j\in \N \text{ with } j\geq \ell \text{ such that } \{k,j\}\subseteq \mathsf L(A)\}$$
 and choose $B\in\Omega$ such that $\t B\t$ is minimal. Then $B=U_1\bdot\ldots\bdot U_k=V_1\bdot\ldots\bdot V_t$, where $t\ge \ell\ge k+2$ and $U_1,\ldots,U_k,V_1,\ldots,V_t\in\mathcal A(G)$ and where the order is such that all commonly occuring terms occure at the beginning, i.e. $U_i=V_i$ say for $i\in [1,a]$ but $U_i\neq V_j$ for any $i>a$ and $j>a$. Then, since $l>k$, we cannot have $a=k$. There must exist $i\in [a+1,k]$, say $i=a+1$, such that $|U_{a+1}|\ge 2$. Suppose $U_{a+1}=g_1\bdot\ldots\bdot g_{l}$ with $1=g_1g_2\ldots g_l$, where $l\in \N_{\ge 2}$ and $g_1,\ldots,g_l\in G$. Then there exists $j\in [a+1,t]$, say $j=a+1$, such that $g_1\in \supp(V_{a+1})$. Let $x\in [1,l]$ be maximal such that $g_1\bdot\ldots\bdot g_x$ is a subsequence of $V_{a+1}$. Then $x<l$ and there exists $r\in [a+2,t]$, say $r=a+2$, such that $g_{x+1}\in \supp(V_{a+2})$. Therefore $U_{a+1}'=U_{a+1}\bdot (g_xg_{x+1})\bdot (g_x\bdot g_{x+1})^{[-1]}$ is an atom and $V'=V_{a+1}\bdot V_{a+2}\bdot (g_xg_{x+1})\bdot (g_x\bdot g_{x+1})^{[-1]}\in \mathcal B(G)$. Set $B'=B\bdot (g_xg_{x+1})\bdot (g_x\bdot g_{x+1})^{[-1]}$. Then $|B'|<|B|$ and
 \[
 t-2+\mathsf L(V')\cup \{k\}\subseteq \mathsf L(B').
 \]
 Then the minimality of $|B|$ implies that $B'\not\in \Omega$, whence $V'\in \mathcal A(G)$  and $t=\ell$. It follows that $\{k,\ell-1\}\subseteq \mathsf L(B')$ and hence $\ell-1\in \mathcal U_k(G)$, a contradiction.

To show the "furthermore" statement, let $G$ be infinite. It follows by Lemma \ref{4.4} that for every $k\ge 2$ there is an atom $A_k$ of length $k$.
Suppose $A_k=h_1\bdot\ldots\bdot h_{k}$ and $A_k^{-1}=h_1^{-1}\bdot \ldots\bdot h_k^{-1}$, where $k\in \N$ and $h_1,\ldots,h_k\in G$.
Thus $\{2,k\}\subseteq\mathsf L(A_k\bdot A_k^{-1})$. Therefore $\rho_2(G)=\infty$ and $2\in \mathcal U_k(G)$ for every $k\ge 2$, whence $\rho_k(G)=\infty$ and  $\lambda_k(G)=2$. The assertion follows by the main statement.

2. If $G$ contains an element of infinite order, then $\mathcal L(\Z)\subseteq \mathcal L(G)$ and the assertion follows by Kainrath's Theorem (\cite[Theorem 1]{Kai99a}, \cite[Theorem 7.4.1]{Ge-HK06a}). If $G$ contains abelian subgroups of unbounded order, then $\mathcal L(G)$ contains the systems of sets of lengths of infinitely many pairwise non-isomorphic finite abelian groups and the assertion follows by
\cite[Theorem 3.7]{Ge-Sch-Zho17}.
\end{proof}

\smallskip
\begin{remark}
The proof of the second statement of Theorem \ref{main2} relies heavily on the existence of arbitrary large abelian subgroups, since only then we can use the already known results on sets of lengths. The case that remains open is the one where $G$ is infinite and does not contain arbitrarily large abelian groups, i.e. $G$ is a torsion group and there are only finitely many pairwise non-isomorphic finite abelian subgroups. Such groups exist; indeed there exist infinite groups such that every subgroup is a cyclic group of order a fixed prime; the latter groups are called \textit{Tarski monsters}.
\end{remark}

\section{On infinite dihedral groups}\label{5}

Product-one sequences over finite dihedral groups received considerable attention in the literature (see \cite{Ge-Gr13a, Gr13b, Cz-Do14a, Br-Ri18a, Oh-Zh20a, Oh-Zh20b, G-G-O-Z21a, Zha21a, Zh21b}). In this final section we consider infinite dihedral groups, and in this setting we can extend and refine the algebraic and arithmetic results of the previous two sections. We formulate our two main results (Theorems \ref{5.1} and \ref{5.2}).

\smallskip
\begin{theorem} \label{5.1}
	Let $G$ be an infinite dihedral group, say $G = \langle \alpha, \tau \colon \tau^2=1, \alpha \tau = \tau \alpha^{-1} \rangle$, and let $G_0\subseteq G$ be a finite nonempty subset.

\begin{enumerate}
\item  $\mathcal B(G_0)$ is a $v$-noetherian G-monoid, $(\mathcal B(G_0):\widehat{\mathcal B(G_0)})\neq\emptyset$, and $\widehat{\mathcal B(G_0)}=\mathcal B(G_0)^{\ast}$ is a finitely generated Krull monoid. Moreover, if $\langle G_0\rangle \cong G$, then $\mathcal C(\mathcal B(G_0)^{\ast})$ is isomorphic to a subgroup of $G/G'\cong C_2\times C_2$.

\item  $\mathcal B(G_0)$ is weakly Krull if and only if one of the following holds.
\begin{itemize}
	\item $G_0\subseteq \langle \alpha\rangle$.
	\item $G_0\setminus\{1\}\subseteq \langle \alpha \rangle\tau$ with $|G_0|\le 3$.

	\item  $G_0\setminus\{1\}=\{\alpha^k\tau\}\cup \{\alpha^i\colon i\in I\}\cup \{\alpha^{-j}\colon j\in J\}$, where $k\in \Z$ and $I,J$ are  nonempty sets of positive integers  such that there exist pairwise co-prime positive integers $b_i$, for $i\in I\cup J$, such that $kb_k=\gcd(I\cup J)\prod_{i\in I\cup J}b_i$ for every $k\in I\cup J$. In particular, the above property holds for $G_0=\{\alpha^i,\alpha^{-j},\alpha^k\tau\}$, where $i,j\in \N$ and $k\in \Z$.

\end{itemize} 

\item 
The following statements are equivalent.
\begin{enumerate}
	\item[(a)] $\mathcal B(G_0)$ is tame.
	\item[(b)] $\omega(G_0)<\infty$.
	\item[(c)] $\mathsf D(G_0)<\infty$.
	\item[(d)] $G_0\subseteq \langle \alpha\rangle$ or $G_0\subseteq \langle \alpha\rangle\tau\cup\{1\}$.
	\item[(e)] $\mathcal B(G_0)$ is finitely generated.
\end{enumerate}

\item The following statements are equivalent.
      \begin{enumerate}
      \item[(a)] $\mathcal B(G_0)$ is locally tame.
      \item[(b)] $\rho(G_0)<\infty$.
      \item[(c)] $\rho_k (G_0) < \infty$ for all $k \ge 2$.
      \item[(d)] $\rho(G_0)$ is accepted.
      \item[(e)] $\mathcal B(G_0\cap \langle \alpha \rangle)\subseteq \{1_{\mathcal B(G_0)}, 1_G^{[n]}:n\geq 0\}$ or $G_0\subseteq \langle \alpha\rangle$.
      \end{enumerate}

\item The set of distances $\Delta (G_0)$ is finite, $\mathsf c (G_0) < \infty$, and there is $M \in \N_0$ such that, for all sufficiently large $k \in \N$,  $\mathcal U_k (G_0)$ is an \AAP \ with difference $\min \Delta (G_0)$ and bound $M$.
\end{enumerate}
\end{theorem}

\smallskip
\begin{theorem} \label{5.2}
Let $G$ be an infinite dihedral group.
\begin{enumerate}
\item $\mathcal B (G)$ is neither seminormal nor weakly Krull.

\item $\mathcal L (G) = \{ L \subseteq \N_{\ge 2} \colon L \ \text{is finite and nonempty } \} \cup \big\{ \{0\}, \{1\} \big\}$, whence $\Delta (G) = \N$ and $\mathcal U_k (G) = \N_{\ge 2}$ for all $k \ge 2$.

\item $\mathcal B (G)$ is an \FF-monoid, but not locally tame, and $\mathsf c (G) = \omega (G) = \infty$.
\end{enumerate}
\end{theorem}

\smallskip
We introduce notation which remains valid for the remainder of this section. Let $G$ be an infinite dihedral group, say $G = \langle \alpha, \tau \colon \tau^2=1, \alpha \tau = \tau \alpha^{-1} \rangle$, and let $G_0\subseteq G$ be a finite nonempty subset. We start with the proof of Theorem \ref{5.1}, which will be given in five steps. Then we discuss three examples, which show the diversity of the behavior of $\mathcal B (G_0)$ with respect to properties such as seminormality and root closure. The proof of Theorem \ref{5.2} will be given at the very end of Section \ref{5}.

\medskip
\begin{proof}[Proof of Theorem \ref{5.1}.1] Let $S$ be a sequence over $G_0$.
We set $G_1=G_0\cap \langle \alpha \rangle$, $S_{G_1}=\prod_{g\in G_1}g^{[\mathsf v_g(S)]}$,  $G_2=G_0\cap\langle \alpha \rangle \tau$, and $S_{G_2}=\prod_{g\in G_2}g^{[\mathsf v_g(S)]}$. Then $G_0=G_1\cup G_2$ is a partition of $G_0$ and $S=S_{G_1}\bdot S_{G_2}$. Next we define $\varphi\colon \mathcal F(G)\rightarrow \mathcal F(\langle\alpha\rangle)$ via $\varphi(\alpha^i\tau)=\alpha^i$, $\varphi(\alpha^i)=\alpha^i$,  and $\psi\colon \mathcal F(\langle \alpha\rangle)\rightarrow \mathcal F(\Z)$ via $\psi(\alpha^i)=i$. If $G_2=\emptyset$, then $G_0=G_1$ is a subset of a cyclic group $\langle \alpha\rangle$ and there is nothing to do. Now we assume that $G_2\neq \emptyset$.

\medskip
\noindent\textbf{Claim A. }	
{\it For every sequence $S\in \mathcal F(G_0)$ such that $|S_{G_2}|\ge 2$ is even, we have that $S\in \mathcal B(G_0)$  if and only if $S$ can be written as $S=T_1\bdot T_2\bdot W_1\bdot W_2$ such that $\sigma(\psi(T_1\bdot \varphi(W_1)))=\sigma(\psi(T_2\bdot \varphi(W_2)))$, where $T_1\bdot T_2=S_{G_1}$ and $W_1\bdot W_2=S_{G_2}$ with $|W_1|=|W_2|$.}

\begin{proof}[Proof of Claim A]
Let $S=f_1\bdot\hdots \bdot f_n\in \mathcal B(G_0)$ with $f_1\ldots f_n=1$ and $|S_{G_2}|\ge 2$ even. We set $\beta=\psi\circ\varphi$ and $\chi(f_i)=|\{f_j\colon j<i, f_j\in G_2\}|$. Then $$1=f_1\ldots f_n=\alpha^{(-1)^{\chi(f_1)}\beta(f_1)+(-1)^{\chi(f_2)}\beta(f_2)+\hdots +(-1)^{\chi(f_n)}\beta(f_n)}\,,$$ and hence $\sum\limits_{\chi(f_i) \text{ even}}\beta(f_i)=\sum\limits_{\chi(f_j) \text{ odd}}\beta(f_j)$. Now the statement follows by defining
\begin{equation*}
T_1=\prod_{\substack{f_i\in G_1\\ \chi(f_i) \text{ even}}}f_i, \ \ \ T_2=\prod_{\substack{f_i\in G_1\\ \chi(f_i) \text{ odd}}}f_i, \ \ \
W_1=\prod_{\substack{f_i\in G_2\\ \chi(f_i) \text{ even}}}f_i, \ \ \
W_2=\prod_{\substack{f_i\in G_2\\ \chi(f_i) \text{ odd}}}f_i, \ \ \
\end{equation*}

Conversely, suppose $S$ has such a decomposition with $W_1=w_1^{(1)}\bdot\hdots \bdot w_n^{(1)}$ and $W_2=w_1^{(2)}\bdot\hdots \bdot w_n^{(2)}$, where $n\ge 1$, $w_1^{(1)},\ldots, w_n^{(1)}, w_1^{(2)}, \ldots, w_n^{(2)}\in G_2$. Then
 $$T_1\bdot w_1^{(1)}\bdot T_2\bdot w_1^{(2)}\bdot  w_2^{(1)}\bdot w_2^{(2)}\bdot\hdots \bdot w_n^{(1)}\bdot w_n^{(2)}\in\mathcal B(G_0)\,.$$
\qedhere[Proof of Claim A]
\end{proof}

Let
\begin{equation}\label{equation:D}
D=[\mathcal A(G_0)\cup \{g^{[2]}\colon g\in G_1\}]\subseteq \mathcal F(G_0)
\end{equation} 	
be the submonoid generated by $\mathcal B(G_0)\cup \{g^{[2]}\colon g\in G_1\}$.
 It is easy to see that  $\{x\in D\colon \supp(x)\cap G_2\neq \emptyset\}\subseteq \mathcal B(G_0)$ and that $\mathcal A(D)\subseteq A(G_0)\cup \{g^{[2]}\colon g\in G_1\}$.

\medskip
\noindent \textbf{Claim B:} \textit{$\mathcal A(D)$ is finite.}
\begin{proof}[Proof of Claim B]
 To begin with, we first show $\sup \{\mathsf v_g(A)\colon A\in \mathcal A(D)\}$  is finite for every $g\in G_2$.
  Assume to the contrary that
 there exist $g_1\in G_2$ and a sequence $(A_i)_{i=1}^{\infty}$ of atoms of $D$ with $|\mathsf v_{g_1}(A_i)|\ge 2$ such that $\lim\limits_{i\rightarrow \infty}\mathsf v_{g_1}(A_i)=\infty$. Furthermore, since $G_0$ is finite, we may assume that $(\mathsf v_g(A_i))_{i=1}^{\infty}$ is increasing (maybe not strictly) for every $g\in G_0$.
 Then $A_i\in \mathcal B(G_0)$ for all $i\in \N$ and for each $i$, we fix $T_1^{(i)}, T_2^{(i)}, W_1^{(i)}, W_2^{(i)}$ such that $A_i$ can be written as $A_i=T_1^{(i)}\bdot T_2^{(i)}\bdot W_1^{(i)}\bdot W_2^{(i)}$, $\sigma(\psi(T_1^{(i)}\bdot \varphi(W_1^{(i)})))=\sigma(\psi(T_2^{(i)}\bdot \varphi(W_2^{(i)})))$,  $T_1^{(i)}\bdot T_2^{(i)}=(A_i)_{G_1}$, and $W_1^{(i)}\bdot W_2^{(i)}=(A_i)_{G_2}$ with $|W_1^{(i)}|=|W_2^{(i)}|$ and $g_1\in \supp(W_1^{(i)})$. Since $\lim\limits_{i\rightarrow \infty}\mathsf v_{g_1}(A_i)=\infty$ it follows that $\lim\limits_{i\rightarrow\infty} |W_1^{(i)}|=\infty$ and  $\lim\limits_{i\rightarrow\infty} |W_2^{(i)}|=\infty$. If $\supp(W_1^{(i)})\cap \supp(W_2^{(i)})\neq\emptyset$, say $h\in \supp(W_1^{(i)})\cap \supp(W_2^{(i)})$, then $h^{[2]}$ and $A_i\bdot (h^{[2]})^{[-1]}$ are both product-one sequences by Claim A, a contradiction. Thus $\supp(W_1^{(i)})\cap \supp(W_2^{(i)})=\emptyset$. Therefore there exists $g_2\in G_2$ such that $\lim\limits_{i\rightarrow \infty}\mathsf v_{g_2}(W_2^{(i)})=\infty$.
 After reordering if necessary, we may assume that $g_1=\alpha^r\tau$ and $g_2=\alpha^s\tau$ with $r<s$ such that $g_1\in\supp (W_1^{(i)})$, $g_2\in \supp(W_2^{(i)})$, and 
 \begin{align*}
 r=&\min\{y\colon g=\alpha^y\tau \text{ such that }\lim\limits_{i\rightarrow \infty}\mathsf v_{g}(W_1^{(i)})=\infty\}\,,\\
\text{and}\quad s=&\min\{y\colon g=\alpha^y\tau \text{ such that }\lim\limits_{i\rightarrow \infty}\mathsf v_{g}(W_2^{(i)})=\infty\}\,.
\end{align*}
 	
Suppose $ \max\{y\in \Z\colon g=\alpha^y\tau \text{ such that }\lim\limits_{i\rightarrow \infty}\mathsf v_{g}(W_1^{(i)})=\infty\}>
s$. Then there exists $g_3=\alpha^k\tau\in G_2\setminus\{g_1,g_2\}$ with $k>s$ and $\lim\limits_{i\rightarrow \infty}\mathsf v_{g_3}(W_1^{(i)})=\infty$.
Let $T=g_1^{[k-s]}\cdot g_2^{[k-r]}\cdot g_3^{[s-r]}$. Then $T$ and $A_i\bdot T^{[-1]}$ are both product-one sequences by Claim A, a contradiction, where $i$ is large enough.

Suppose $ \max\{y\in \Z\colon g=\alpha^y\tau \text{ such that }\lim\limits_{i\rightarrow \infty}\mathsf v_{g}(W_1^{(i)})=\infty\}<
s$. It follows by $|W_1^{(i)}|=|W_2^{(i)}|\to\infty$ that 
 $\lim\limits_{i\rightarrow \infty} \sigma(\psi(\varphi(W_2^{(i)})))-\sigma(\psi(\varphi(W_1^{(i)})))=\infty$.
   Then $\lim\limits_{i\rightarrow \infty} \sigma(\psi(T_1^{(i)}))-\sigma(\psi(T_2^{(i)}))=\infty$. Then there must exist $g=\alpha^k\in G_1$  such that either $\lim\limits_{i\rightarrow \infty}\mathsf v_{g}(T_1^{(i)})=\infty$ (if $k>0$) or $\lim\limits_{i\rightarrow \infty}\mathsf v_{g}(T_2^{(i)})=\infty$ (if $k<0$). Let $T=g_1^{[|k|]}\cdot g_2^{[|k|]}\cdot g^{[s-r]}$. Then $T$ and $A_i\bdot T^{[-1]}$ are both product-one sequences by Claim A, a contradiction, where $i$ is large enough.

   Thus, we proved   $\sup \{\mathsf v_g(A)\colon A\in \mathcal A(D)\}$  is finite for every $g\in G_2$. Next,  we show $\sup \{\mathsf v_g(A)\colon A\in \mathcal A(D)\}$  is finite for every $g\in G_1$. Note that $\sup \{\mathsf v_g(A)\colon A\in \mathcal A(D)\cap \mathcal F(G_1) \}\leq\mathsf D(G_1)$ is finite.  It is sufficient to show $\sup \{\mathsf v_g(A)\colon A\in \mathcal A(D) \text{ with } \supp(A)\cap G_2\neq \emptyset\}$  is finite for every $g\in G_1$.
 Assume to the contrary that
there exist $g_1\in G_1$ and a sequence $(A_i)_{i=1}^{\infty}$ of atoms of $D$ with $\supp(A_i)\cap G_2\neq \emptyset$ such that $\lim\limits_{i\rightarrow \infty}\mathsf v_{g_1}(A_i)=\infty$.  Furthermore, since $G_0$ is finite, we may assume that $(\mathsf v_g(A_i))_{i=1}^{\infty}$ is increasing (maybe not strictly) for every $g\in G_0$. Then $A_i\in \mathcal B(G_0)$ for all $i\in \N$  and for each $i$, we fix $T_1^{(i)}, T_2^{(i)}, W_1^{(i)}, W_2^{(i)}$ such that $A_i$ can be written as $A_i=T_1^{(i)}\bdot T_2^{(i)}\bdot W_1^{(i)}\bdot W_2^{(i)}$, $\sigma(\psi(T_1^{(i)}\bdot \varphi(W_1^{(i)})))=\sigma(\psi(T_2^{(i)}\bdot \varphi(W_2^{(i)})))$,  $T_1^{(i)}\bdot T_2^{(i)}=(A_i)_{G_1}$, and $W_1^{(i)}\bdot W_2^{(i)}=(A_i)_{G_2}$ with $|W_1^{(i)}|=|W_2^{(i)}|$ and $g_1\in \supp(T_1^{(i)})$.
If  $\supp(T_1^{(i)})\cap \supp(T_2^{(i)})\neq \emptyset$, say $h\in\supp(T_1^{(i)})\cap \supp(T_2^{(i)})$, then $h^{[2]}\in \mathcal A(D)$ and $A_i\bdot (h^{[2]})^{[-1]}$ is a product-one sequence by Claim A, a contradiction.
Thus $\supp(T_1^{(i)})\cap \supp(T_2^{(i)})=\emptyset$, whence in combination with $\lim\limits_{i\rightarrow \infty}\mathsf v_{g_1}(A_i)=\infty$ it follows that $\lim\limits_{i\rightarrow\infty} |T_1^{(i)}|=\infty$. Note that $T_1^{(i)}$ is a sequence over a cyclic group $\langle \alpha\rangle$. Then $T_1^{(i)}$ is product-one free and hence $\lim\limits_{i\rightarrow\infty} |\sigma(\psi(T_1^{(i)}))|=\infty$. To see this, assume to the contrary that $S_i=\psi(T_1^{(i)})$ are sequences over a finite subset $\psi(G_1)$ of integers with $|\sigma(S_i)|<N$. Then $\mathsf D(G_1\cup [-N,N])=L<\infty$ by Dickson's Lemma. Thus, if $|S_i|\geq L$, then $S_i\bdot -\sigma(S_i)$ is a product-one sequence of length greater than $L$, implying it factors into two non-trivial product-one sequences, one of which is a subsequence of $S_i$. As a result, once $|S_i|\to \infty$ is large enough, we would be guaranteed of $S_i=\sigma(T_1^{(i)})$ having a non-trivial product-one subsequence, contradicting that $T_1^{(i)}$ is product-one free. This shows $\lim\limits_{i\rightarrow\infty} |\sigma(\psi(T_1^{(i)}))|=\infty$.
In combination with $|W_1^{(i)}|=|W_2^{(i)}|\leq \sum_{g\in G_2}\sup \{\mathsf v_g(A):A\in\mathcal A(D)\}<\infty$ (we already proved that) this gives $\lim\limits_{i\rightarrow\infty} |\sigma(\psi(T_2^{(i)}))|=\infty$ and
$\lim\limits_{i\rightarrow\infty} |T_2^{(i)}|=\infty$, whence there exists $g_2\in G_1$ with $\psi(g_1)\psi(g_2)>0$ (i.e. with the same sign) such that $\lim\limits_{i\rightarrow \infty}\mathsf v_{g_2}(T_2^{(i)})=\infty$. Suppose $g_1=\alpha^r$ and $g_2=\alpha^t$. Then $T=g_1^{[2|t|]}\bdot g_2^{[2|r|]}\in D$ and $A_i\bdot T^{[-1]}$ is a product-one sequence by Claim A, a contradiction, where $i$ is large enough. To sum up, we proved that $N=\sum_{g\in G_0}\sup \{\mathsf v_g(A):A\in\mathcal A(D)\}<\infty$. Since $N$ is the greatest length an atom in $D$ can have, the finiteness of $G_0$ gives that $\mathcal A(D)$ is finite.
 \qedhere[Proof of Claim B]
   \end{proof}

Let $X\subseteq \mathcal B(G_0)$ be an infinite subset. For every $S\in X$, we fix a factorization $z_S\in \mathsf Z(D)$. Since $\mathsf Z(D)$  is a finitely generated free abelian monoid, by Dickson's Lemma (\cite[Theorem 1.5.3]{Ge-HK06a}), we obtain the existence of a finite subset $X'\subseteq X$ such that $\{z_S\colon S\in X'\}$ is the set of all minimal elements of $\{z_S\colon S\in X\}$.
 For every $A\in X'$, if $\{A\bdot A'\in X\setminus X'\colon A'\in D\setminus \mathcal B(G_0)\}\neq \emptyset$, then
 we choose such an $A\bdot A'$ and add it to $X'$. Finally, we get a finite set $X_0$ with $|X_0|\le 2|X'|$.
 We claim $(\mathcal B(G_0)\colon X)=(\mathcal B(G_0)\colon X_0)$, which implies $\mathcal B(G_0)$ is $v$-noetherian by \cite[Proposition 2.1.10]{Ge-HK06a}.

 It suffices to show $(\mathcal B(G_0)\colon X_0)\subseteq (\mathcal B(G_0)\colon X)$.  Let $\frac{S_1}{S_2}\in \mathsf q(\mathcal B(G_0))$ with $S_1,S_2\in \mathcal B(G_0)$ such that $\frac{S_1}{S_2}X_0\subseteq \mathcal B(G_0)$. Let $S\in X$. Then there exists $A\in X'\subseteq X$ such that $\frac{S}{A}\in D$.
 If $\frac{S}{A}\in \mathcal B(G_0)$, then $\frac{S_1}{S_2}S=\frac{S_1}{S_2}A\frac{S}{A}\in \mathcal B(G_0)$. Otherwise $\frac{S}{A}=B\bdot B'$ with $B\in \mathcal B(G_0\cap \langle \alpha\rangle)$ and $B'=[g^{[2]}\colon g\in G_0\cap \langle \alpha\rangle]\setminus \{1\}$. Since $A\in X'$ and $A\bdot B\bdot B'\in X$,  by our construction of $X_0$, there is an $A'\in D\setminus \mathcal B(G_0)$ such that $A\bdot A'\in X_0$. We have that $\frac{S_1}{S_2}A\in \mathcal B(G_0)$  and $\frac{S_1}{S_2}A\bdot A'\in \mathcal B(G_0)$. Assume to the contrary, that $\frac{S_1}{S_2}A\in \mathcal B(G_0\cap \langle \alpha\rangle)$. Then $\frac{S_1}{S_2}A$ is a product-one sequence over the abelian group $\langle \alpha\rangle$ and since $A'\in D\setminus\mathcal B(G_0)$, we have that $\supp(A')\subseteq\langle\alpha\rangle$, and $\sigma(\psi(A'))\neq 0$ as $A'\notin\mathcal B(G_0)$. But now $\frac{S_1}{S_2}A\bdot A'$ is a sequence over the abelian group $\langle \alpha\rangle$ with $\sigma(\psi(\frac{S_1}{S_2}A\bdot A'))=\sigma(psi(\frac{S_1}{S_2}A))+\sigma(\psi(A'))=\sigma(\psi(A'))\neq 0$, contradicting that $\frac{S_1}{S_2}A\bdot A'\in\mathcal B(G_0)$. We obtain  that $\frac{S_1}{S_2}A\notin\mathcal B(G_0\cap\langle\alpha\rangle)$.
 Therefore $\frac{S_1}{S_2}S=\frac{S_1}{S_2}A\bdot\frac{S}{A}\in \mathcal B(G_0)$. It follows that  $\frac{S_1}{S_2}X\subseteq\mathcal B(G_0)$ and the $v$-noetherian property is shown.

 By Theorem \ref{3.11}, $\mathcal B(G_0)^*$ is a finitely generated Krull monoid and $\mathcal B(G_0)$ is a G-monoid. Suppose $g\in G_2$ (then $g^{[2]}\in\mathcal B(G_0)$) and $\mathcal A(\mathcal B(G_0)^*)=\{T_1,\ldots, T_r\}$, where $r\in \N$.
 Using Proposition \ref{3.3}.2, for every $i\in [1,r]$ we fix one $T_i'\in \mathcal B(G_0)$ such that $T_i\bdot T_i'\in \mathcal B(G_0)$. Set $W=\prod_{i\in [1,r]}T_i'$.
 Thus for every $T\in \mathcal B(G_0)^*$, we have that $T=\prod_{i\in [1,r]}T_i^{\ell_i}$, where $\ell_i\in \N$ for every $i\in [1,r]$, whence $$T\bdot W\bdot g^{[2]}=\Big(\prod_{i\in [1,r]}T_i^{[2\lfloor \ell_i/2\rfloor]}\bdot g^{[2]}\Big)\bdot \Big(\prod_{i\in [1,r],\ell_i\text{ is odd}}T_i\bdot T_i'\Big)\bdot \prod_{i\in [1,r],\ell_i\text{ is even}}T_i'\in \mathcal B(G_0)\,.$$
 Therefore $(\mathcal B(G_0)\colon \mathcal B(G_0)^*)\neq \emptyset$ and $\mathcal B(G_0)^*\subseteq \widehat{\mathcal B(G_0)}$. It follows by $\widehat{\mathcal B(G_0)}\subseteq \mathcal B(G_0)^*$ that $\widehat{\mathcal B(G_0)}=\mathcal B(G_0)^{*}$ is a finitely generated Krull monoid and $(\mathcal B(G_0):\widehat{\mathcal B(G_0)})\neq\emptyset$.

 Moreover, suppose that $\langle G_0\rangle \cong G$. From Proposition \ref{3.3}, we know that $\mathcal B(G_0)^{\ast}\subseteq \mathcal F(G_0)$ is a saturated and cofinal submonoid. Furthermore, denoting the given inclusion by $\varphi: B(G_0)^{\ast}\to \mathcal F(G_0)$, we have that $\mathcal C(\varphi)=\mathcal F(G_0)/\mathcal B(G_0)^{\ast}\cong G/G'\cong C_2\times C_2$. Now by  \cite[Theorem 2.4.7.2]{Ge-HK06a}, there are submonoids $F_0\subseteq \mathcal F(G_0)$ and $\mathcal C_0=\{[c]_{\varphi}\colon c\in F_0\}\subseteq \mathcal C(\varphi)$ such that there is  an epimorphism $\varphi^*\colon \mathcal C_0\to \mathcal C(\mathcal B(G_0)^{\ast})$. Since $\mathcal B(G_0)^{\ast}\subseteq F_0$ is still cofinal, it follows that $\mathcal C_0$ is a subgroup of $\mathcal C(\varphi)$, whence $\mathcal C(\mathcal B(G_0)^{\ast})$ is a factor group of  a subgroup of $C(\varphi)\cong C_2 \times C_2$.
\end{proof}

\smallskip

To continue with the proof we need the following results.

\begin{lemma}\label{new1}
	Let $i,j,k$ be  distinct positive integers with $\gcd(i,j,k)=1$.
	Then the following are equivalent.
	\begin{enumerate}
		\item There exist $x,y,z\in \N$ with $\gcd(x,y,z)=1$ such that $ix+jy=kz$ and $ix\not\equiv 0\pmod k$.
		
		\item There exist $x,y,z\in \N$ with $\gcd(x,y,z)=1$ such that $ix+jy=kz$, $ix'\not\equiv 0\pmod k$ for every $x'\in [1,x]$, and $jy'\not\equiv 0\pmod k$ for every $y'\in [1,y]$.
		
		\item $k\neq \gcd(i,k)\gcd(j,k)$.
	\end{enumerate}
\end{lemma}

\begin{proof}
	Since $\gcd(i,j,k)=1$, we have that $\gcd(\gcd(i,k),\gcd(j,k))=1$ and hence $\gcd(i,k)\gcd(j,k)$ divides $k$. 
	
	$(3)\Rightarrow (1)$ Suppose $k\neq \gcd(i,k)\gcd(j,k)$. There exist $y',z'\in \N$ such that $\gcd(j,k)=kz'-jy'$ and hence $i\gcd(j,k)+j(y'i)=k(z'i)$. Assume to the contrary that $i\gcd(j,k)\equiv 0\pmod k$. Then ${k/(\gcd(i,k)\gcd(j,k))}$ divides $i/\gcd(i,k) $, a contradiction to the fact that  $$\gcd(i/\gcd(i,k), k/\gcd(i,k))=1\,.$$
	Let $d=\gcd(\gcd(j,k), y'i, z'i)$. Then the assertion follows by choosing $$(x,y,z)=(\gcd(j,k)/d, y'i/d, z'i/d)\,.$$

	$(1)\Rightarrow (2)$. Among all the choices of $(x,y,z)\in \N^3$ with $\gcd(x,y,z)=1$ such that $ix+jy=kz$ and $ix\not\equiv 0\pmod k$, we let $(x_0,y_0,z_0)$ be the choice such that $x_0+y_0$ is minimal. If there exists $x'\in [1,x_0-1]$ such that $ix'\equiv 0\pmod k$, then there exists $z'\in [1,z_0-1]$ such that $ix'=kz'$ and hence $i(x_0-x')+jy_0=k(z_0-z')$. Thus dividing $x_0-x', y_0$ and $z_0-z'$ by their gcd leads to a contradiction to the minimality of $x_0+y_0$. Suppose there exists $y'\in [1,y_0]$ such that $jy'\equiv 0\pmod k$. If $y'=y_0$, then $k$ divides $ix_0$, a contradiction. If $y'<y_0$, then there exists $z'\in [1,z_0-1]$ such that $jy'=kz'$ and hence $ix_0+j(y_0-y')=k(z_0-z')$, as before a contradiction to the minimality of $x_0+y_0$.

	$(2)\Rightarrow(3)$ Suppose there exist $x,y,z\in \N$ such that $ix+jy=kz$ and $ix\not\equiv 0\pmod k$. Assume to the contrary that $k=\gcd(i,k)\gcd(j,k)$. Since $\gcd(i,k)$ divides $ix$ and $\gcd(j,k)$ divides $ix$, we obtain that $k=\gcd(i,k)\gcd(j,k)$ divides $ix$, a contradiction.
	\end{proof}

\begin{lemma}
	\label{new2}
	Let $I\subseteq \N$ be a subset with $|I|\ge 3$. Suppose that for any three elements $i,j,k$ of $I$, we have that $i=\frac{\gcd(i,j,k)jk}{\gcd(j,k)^2}$, $j=\frac{\gcd(i,j,k)ik}{\gcd(i,k)^2}$, and $k=\frac{\gcd(i,j,k)ij}{\gcd(i,j)^2}$. Then there exist pairwise co-prime positive integers $b_i$, for $ i\in I$, such that
	\[
	kb_k=\gcd(I)\prod_{i\in I}b_i
	\]
for every $k\in I$.
\end{lemma}

\begin{proof}
Suppose $|I|=3$, say $I=\{i,j,k\}$. Let $b_i=\gcd(j,k)/\gcd(i,j,k)$, $b_j=\gcd(i,k)/\gcd(i,j,k)$, and $b_k=\gcd(i,j)/\gcd(i,j,k)$.
By symmetry, it suffices to show that $k=\gcd(i,j,k)b_ib_j$. 
Since $i\gcd(i,j,k)=jk/b_i^2$ and $j\gcd(i,j,k)=ik/b_j^2$, we obtain that
$ij\gcd(i,j,k)^2=\frac{ijk^2}{b_i^2b_j^2}$, whence $k=\gcd(i,j,k)b_ib_j$.

Suppose $t=|I|\ge 4$. We proceed by induction on $t$. Suppose the assertion holds for every subset $J\subseteq I$ with $|J|=t-1$. 
Let $I=\{a_1,\ldots,a_t\}$ and $b_t=\gcd(a_1,\ldots,a_{t-1})/\gcd(a_1,\ldots,a_t)$. By induction hypothesis, there exist pairwise co-prime positive integers $b_1,\ldots,b_{t-1}$ such that $a_ib_i=\gcd(a_1,\ldots,a_t)b_1\ldots b_{t}$ for every $i\in [1,t-1]$.
Let $j\in [2,t-1]$. Consider the subset $\{a_1,a_j,a_t\}$ of $I$. In view of $\gcd(a_1, a_j)=\gcd(a_1,\ldots,a_t)\gcd(\frac{b_1\ldots b_t}{b_1}, \frac{b_1\ldots b_t}{b_j})=\gcd(a_1,\ldots,a_t)\frac{b_1\ldots b_t}{b_1b_j}$, we have that
\begin{equation}\label{equation:a_t}
a_t=\frac{\gcd(a_1,a_j,a_t)a_1a_j}{\gcd(a_1,a_j)^2}=\gcd(a_1,a_j,a_t)b_1b_j,
\end{equation}
 whence $b_j$ divides $a_t/\gcd(a_1,\ldots,a_t)$ for every $j\in [1,t-1]$.
 It follows that $b_1\ldots b_{t-1}$ divides $a_t/\gcd(a_1,\ldots,a_t)$. Using the equation $\gcd(a_1, a_j)=\gcd(a_1,\ldots,a_t)\frac{b_1\ldots b_t}{b_1b_j}$ again, it follows by $\gcd(b_t, \frac{a_t}{\gcd(a_1,\ldots ,a_t)})=1$, that  $$\gcd(a_1,a_j,a_t)=\gcd(a_1,\ldots,a_t)\gcd\left(\frac{b_1\ldots b_{t}}{b_1b_j}, \frac{a_t}{\gcd(a_1,\ldots, a_t)}\right)=\gcd(a_1,\ldots,a_t)\frac{b_1\ldots b_{t-1}}{b_1b_j}\,.$$
 Therefore $a_tb_t=\gcd(a_1,a_j,a_t)b_1b_jb_t=\gcd(a_1,\ldots,a_t)b_1\ldots b_t$ by (\ref{equation:a_t}).
 
 Since $1=\gcd(b_t, \frac{a_t}{\gcd(a_1,\ldots ,a_t)})=\gcd(b_t, b_1\ldots b_{t-1})$, we know $b_1,\ldots,b_t$ are pairwise co-prime positive integers.
\end{proof}

\begin{lemma}
	\label{new3}
	Let $G$ be an infinite dihedral group, say $G = \langle \alpha, \tau \colon \tau^2=1, \alpha \tau = \tau \alpha^{-1} \rangle$, and let $G_0=\{\tau\}\cup\{\alpha^i\colon i\in  I\}\cup \{\alpha^{-j}\colon j\in J\}$, where 
	$I,J$ are  nonempty sets of positive integers.
	\begin{enumerate}
		\item $\mathfrak X(\mathcal B(G_0))= \{\mathfrak p_a\colon a\in G_0\}$ and $\bigcap_{\mathfrak p\in \mathfrak X(\mathcal B(G_0))}\mathcal B(G_0)_{\mathfrak p}\subseteq \mathcal F(G_0)$, where $\mathfrak p_{a}=\{S\in\mathcal B(G_0)\colon a\in \supp(S)\}
		$.
		
		\item If $\mathcal B(G_0)$ is weakly Krull, then for any disjoint subsets $K_1,K_2\subseteq I\cup J$ with $1\le |K_1|+|K_2|$ and $|K_1|<|I\cup J|$, we have that $\mathcal B\big(\{\tau\}\cup\{\alpha^i\colon i\in (I\cup J)\setminus K_1\}\cup \{\alpha^{-k}\colon k\in K_1\cup K_2\}\big)$ is weakly Krull.
	\end{enumerate}
	
\end{lemma}

\begin{proof}
	1. Let $\mathfrak p_{a}=\{S\in\mathcal B(G_0)\colon a\in \supp(S)\}
	$, where $a\in G_0$. Then Proposition \ref{3.8} implies that $\mathfrak X(\mathcal B(G_0))\subseteq \{\mathfrak p_a\colon a\in G_0\}$.
	Assume to the contrary that there exist distinct $a,b\in G_0$ such that $\mathfrak p_a\subseteq \mathfrak p_b$. Since $\tau^{[2]}$ is an atom, we obtain that $a\neq \tau$ and hence $a\in G_0\setminus\{\tau\}$. Since $a^{[2]}\bdot \tau^{[2]}$ is an atom, we have that $b$ can only be $\tau$. Note that $I$ and $J$ are both nonempty sets. There exists a product-one sequence $B$ with $a\in \supp(B)\subseteq G_0\setminus \{\tau\}$, whence $B\in \mathfrak p_a\setminus \mathfrak p_b$, a contradiction. Therefore $\mathfrak X(\mathcal B(G_0))= \{\mathfrak p_a\colon a\in G_0\}$ and it follows by Proposition \ref{3.8}.2 that $\bigcap_{\mathfrak p\in \mathfrak X(\mathcal B(G_0))}\mathcal B(G_0)_{\mathfrak p}\subseteq \mathcal F(G_0)$.

	\medskip
	2.	Let $i\in I$ and let $j\in J$. It suffices to show the following Claim.
	
	\smallskip	
	\noindent
	{\bf Claim:} \begin{enumerate}
		\item[(a)] $\mathcal B(G_0\cup\{\alpha^{-i}\}\setminus \{\alpha^i\})$ is weakly Krull, provided that $|I|\ge 2$.
		
		\item[(b)]  $\mathcal B(G_0\cup\{\alpha^{j}\}\setminus \{\alpha^{-j}\})$ is weakly Krull, provided that $|J|\ge 2$.
		
		\item[(c)] $\mathcal B(G_0\cup\{\alpha^{-i}\})$   is weakly Krull.
		
			\item[(d)]  $\mathcal B(G_0\cup\{\alpha^{j}\})$ is weakly Krull.
	\end{enumerate} 
	\smallskip
	
	By symmetry, we only prove items (a) and (c). Let $G_1=G_0\cup\{\alpha^{-i}\}$ and let $G_2=G_1\setminus \{\alpha^i\}$.
	We define a homomorphism
	 $\phi\colon \mathcal F(G_1)\rightarrow \mathcal F(G_0)$  by $\phi(\alpha^{-i})=\alpha^{i}$ and $\phi(x)=x$ for all $x\in G_0\setminus\{\alpha^{-i}\}$. It easy to see that, for a sequence $S\in \mathcal F(G_1)$ with $\mathsf v_{\tau}(S)>0$, we have that $S\in \mathcal B(G_1)$ if and only if $\phi(S)\in \mathcal B(G_0)$.	
It follows by 1. that $\bigcap_{\mathfrak p\in \mathfrak X(\mathcal B(G_1))}\mathcal B(G_1)_{\mathfrak p}\subseteq \mathcal F(G_1)$ and $\bigcap_{\mathfrak p\in \mathfrak X(\mathcal B(G_2))}\mathcal B(G_2)_{\mathfrak p}\subseteq \mathcal F(G_2)$.

Assume to the contrary that $\mathcal B(G_1)$ is not weakly Krull.
	Then there exist $T\in \mathcal F(G_1)\setminus\mathcal B(G_1)$ and $S_a\in \mathcal B(G_1)$, for every $a\in G_1$, such that $a\not\in \supp(S_a)$ and $T\bdot S_a\in \mathcal B(G_1)$. If $\tau\not\in \supp(T)$, then $T\bdot S_{\tau}$ and $S_{\tau}$ are both product-one sequences over $\langle \alpha \rangle$, whence $T$ is a product-one sequence, a contradiction. Thus $\tau\in \supp(T)$ and hence $\phi(T)\in \mathcal F(G_0)\setminus\mathcal B(G_0)$,  $\phi(S_a\bdot \tau^{[2]})\in \mathcal B(G_0)$, and $\phi(T\bdot S_a\bdot\tau^{[2]})\in \mathcal B(G_0)$  for every $a\in G_1\setminus \{\tau\}$. Suppose $t=\mathsf v_{\alpha^{-i}}(S_{\tau})$. Then $\alpha^{2it}\in \pi(\phi(S_{\tau}))$ and $T_{\tau}:=\phi(S_{\tau}^{[j]})\bdot (\alpha^{-j})^{[2it]}$ is a product-one sequence with $\tau\not\in \supp(T_{\tau})$.
	Since $T\bdot S_{\tau}^{[j]}=T\bdot S_{\tau}\bdot S_{\tau}^{[j-1]}$ is a product-one sequence over $G_1$, we obtain that $\phi(T\bdot S_{\tau}^{[j]})$ is a product-one sequence over $G_0$ and hence $\phi(T)\bdot T_{\tau}$ is a product-one sequence. 
	Putting this all together, we obtain that $\phi(T)\in (\bigcap_{\mathfrak p\in \mathfrak X(\mathcal B(G_0))}\mathcal B(G_0)_{\mathfrak p})\setminus \mathcal B(G_0)$, a contradiction.

	Assume to the contrary that $\mathcal B(G_2)$ is not weakly Krull.
	Then there exist $T\in \mathcal F(G_2)\setminus\mathcal B(G_2)$ and $S_a\in \mathcal B(G_2)$, for every $a\in G_2$, such that $a\not\in \supp(S_a)$ and $T\bdot S_a\in \mathcal B(G_2)$. If $\tau\not\in \supp(T)$, then $T\bdot S_{\tau}$ and $S_{\tau}$ are both product-one sequences over $\langle \alpha \rangle$, whence $T$ is a product-one sequence, a contradiction. Thus $\tau\in \supp(T)$ and hence $\phi(T)\in \mathcal F(G_0)\setminus\mathcal B(G_0)$,  $\phi(S_a\bdot \tau^{[2]})\in \mathcal B(G_0)$, and $\phi(T\bdot S_a\bdot\tau^{[2]})\in \mathcal B(G_0)$  for every $a\in G_2\setminus \{\tau\}$. Suppose $t=\mathsf v_{\alpha^{-i}}(S_{\tau})$. Then $\alpha^{2it}\in \pi(\phi(S_{\tau}))$ and $T_{\tau}:=\phi(S_{\tau}^{[j]})\bdot (\alpha^{-j})^{[2it]}$ is a product-one sequence with $\tau\not\in \supp(T_{\tau})$.
	Since $T\bdot S_{\tau}^{[j]}=T\bdot S_{\tau}\bdot S_{\tau}^{[j-1]}$ is a product-one sequence over $G_2$, we obtain that $\phi(T\bdot S_{\tau}^{[j]})$ is a product-one sequence over $G_0$ and hence $\phi(T)\bdot T_{\tau}$ is a product-one sequence. 
	Putting this all together, we obtain that $\phi(T)\in (\bigcap_{\mathfrak p\in \mathfrak X(\mathcal B(G_0))}\mathcal B(G_0)_{\mathfrak p})\setminus \mathcal B(G_0)$, a contradiction. 
\end{proof}

\begin{lemma}\label{new4}
	Let $G$ be an infinite dihedral group, say $G = \langle \alpha, \tau \colon \tau^2=1, \alpha \tau = \tau \alpha^{-1} \rangle$.
	Let  $G_0=\{\alpha^i, \alpha^j, \alpha^{-k}, \tau\}$ be a subset, where $i,j,k$ are distinct positive integers with $\gcd(i,j,k)=1$. If $\mathcal B(G_0)$ is weakly Krull, then there exist pairwise co-prime  integers $r,s,t\in \N$  such that $i=st$, $j=rt$, and $k=rs$.
\end{lemma}

\begin{proof}
	Let $i=2^{i_0}i_1$, $j=2^{j_0}j_1$, and $k=2^{k_0}k_1$, where $i_0,j_0,k_0\in \N_0$ and $i_1,j_1,k_1$ are odd.
	Let $G_1=\{\alpha^i, \alpha^{-j}, \alpha^{k}, \tau\}$ and $G_2=\{\alpha^{-i}, \alpha^j, \alpha^{k}, \tau\}$.
	Then Lemma \ref{new3}.2 implies that $\mathcal B(G_0)$ is weakly Krull if and only if $\mathcal B(G_1)$ is weakly Krull if and only if $\mathcal B(G_2)$ is weakly Krull. Therefore the elements $i,j,k$ are symmetric and hence we may assume that $i_0\ge j_0\ge k_0$. Since $\gcd(i,j,k)=1$, we know $k_0=0$, i.e., $k$ is odd.

	Let $r=\gcd(j,k)$, $s=\gcd(i,k)$, and $t=\gcd(i,j)$. It follows from $\gcd(i,j,k)=1$ that $r,s,t$ are pairwise co-prime.
	We distinguish four cases.

	\medskip
	\noindent{\bf Case 1: } $i$ is even and $j$ is odd.
	
	Since the assertion can not hold in this case, we need to show that $\mathcal B(G_0)$ is not weakly Krull.
	Let $S=(\alpha^{-i})^{[j]}\bdot \tau^{[2]}$. Then $S$ is not a product-one sequence. It follows by Lemma \ref{new3}.1 that
	\[
	S=\frac{(\alpha^{-i})^{[k+j]}\bdot (\alpha^{k})^{[i]}\bdot \tau^{[2]}}{(\alpha^{-i})^{[k]}\bdot (\alpha^{k})^{[i]}}=\frac{(\alpha^{j})^{[i]}\bdot (\alpha^{-i})^{[j]} \bdot \tau^{[4]}}{(\alpha^{j})^{[i]}\bdot \tau^{[2]}}\in \bigcap_{\mathfrak p\in \mathfrak X(\mathcal B(G_2))}\mathcal B(G_2)_{\mathfrak p}\,.
	\]
	Therefore $\mathcal B(G_2)$ is not weakly Krull and hence $\mathcal B(G_0)$ is not weakly Krull.

	\medskip
	\noindent{\bf Case 2: } $i,j$ are both even and 
	$i_0>j_0$.
	
	Since the assertion can not hold in this case, we need to show that $\mathcal B(G_0)$ is not weakly Krull.
	Let $S=(\alpha^{-i})^{[k]}\bdot \tau^{[2]}$. Then $S$ is not a product-one sequence. It follows by Lemma \ref{new3}.1 that 
	\begin{align*}
	S&=\frac{(\alpha^{-i})^{[k]}\bdot (\alpha^{k})^{[i]}\bdot \tau^{[4]}}{ (\alpha^{k})^{[i]}\bdot \tau^{[2]}}\\
	&=\frac{(\alpha^{j})^{[2^{i_0-j_0}i_1]}\bdot (\alpha^{-i})^{[j_1+k]}\bdot \tau^{[2]}}{(\alpha^{j})^{[2^{i_0-j_0}i_1]}\bdot (\alpha^{-i})^{[j_1]}}\in \bigcap_{\mathfrak p\in \mathfrak X(\mathcal B(G_2))}\mathcal B(G_2)_{\mathfrak p}\,.
	\end{align*}
	Therefore $\mathcal B(G_2)$ is not weakly Krull and hence $\mathcal B(G_0)$ is not weakly Krull.

	\medskip
	\noindent{\bf Case 3: }
	$i, j, k$ are odd.

	Suppose $k\neq rs$. Then Lemma \ref{new1} implies that there exist
	$x,y,z\in \N$ with $\gcd(x,y,z)=1$ such that $ix+jy=kz$, $ix'\not\equiv0\pmod k $ for every $x'\in [1,x]$, and $jy'\not\equiv0\pmod k $ for every $y'\in [1,y]$. 
	If $z$ is odd, then $x$ or $y$ must be odd. By symmetry, we may assume that $x$ is odd. Then $y$ is even. Let $S=(\alpha^i)^{[x]}\bdot (\alpha^{-k})^{[z]}\bdot \tau^{[2]}$. Assume to the contrary that $S$ is a product-one sequence. Then there are subsequences $T_1,T_2$ over $\{\alpha^i,\alpha^{-k}\}$ such that $S=T_1\bdot \tau\bdot T_2\bdot \tau$ and $\pi(T_1)=\pi(T_2)$. Suppose $T_1=(\alpha^i)^{[x_0]}\bdot (\alpha^{-k})^{[z_0]}$ and $T_2=(\alpha^i)^{[x-x_0]}\bdot (\alpha^{-k})^{[z-z_0]}$, where $x_0\in [0, x]$ and $z_0\in [0,z]$. Then $ix_0-kz_0=i(x-x_0)-k(z-z_0)$ and hence $i|x-2x_0|\equiv 0\pmod k$, a contradiction to the fact that $ix'\not\equiv0\pmod k $ for every $x'\in [1,x]$.
Therefore $S$ is not a product-one sequence and  it follows by Lemma \ref{new3}.1 that
	\[
	S=\frac{S\bdot (\alpha^i)^{[k]}\bdot (\alpha^{-k})^{[i]}}{(\alpha^i)^{[k]}\bdot (\alpha^{-k})^{[i]}}=\frac{S\bdot (\alpha^j)^{[y]}\bdot \tau^{[2]}}{(\alpha^j)^{[y]}\bdot \tau^{[2]}}\in \bigcap_{\mathfrak p\in \mathfrak X(\mathcal B(G_0))}\mathcal B(G_0)_{\mathfrak p}\,,
	\]
	whence $\mathcal B(G_0)$ is not weakly Krull, a contradiction.
	If $z$ is even, then  $x$ and $y$ must be odd.   
	Let $S=(\alpha^i)^{[x]}\bdot (\alpha^{-k})^{[z+j]}\bdot \tau^{[2]}$. Then $S$ is not a product-one sequence and it follows by Lemma \ref{new3}.1 that
	\[
	S=\frac{S\bdot (\alpha^i)^{[k]}\bdot (\alpha^{-k})^{[i]}}{(\alpha^i)^{[k]}\bdot (\alpha^{-k})^{[i]}}=\frac{S\bdot (\alpha^j)^{[y+k]}\bdot \tau^{[2]}}{(\alpha^j)^{[y+k]}\bdot \tau^{[2]}}\in \bigcap_{\mathfrak p\in \mathfrak X(\mathcal B(G_0))}\mathcal B(G_0)_{\mathfrak p}\,,
	\]
	whence $\mathcal B(G_0)$ is not weakly Krull, a contradiction.
	
	Suppose $i\neq st$. Similarly we can prove that $\mathcal B(G_2)$ is not weakly Krull and hence $\mathcal B(G_0)$ is not weakly Krull, a contradiction.
	Suppose $j\neq rt$. Similarly we can prove that $\mathcal B(G_1)$ is not weakly Krull and hence $\mathcal B(G_0)$ is not weakly Krull, a contradiction.

	\medskip
	\noindent{\bf Case 4: }
	$i=2^{u}i_1$ and $j=2^{u}j_1$ such that $u\in \N$  and $i_1, j_1, k$ odd.
	
	Suppose $k\neq rs$. Then Lemma \ref{new1} implies that there exist
	$x,y,z\in \N$ with $\gcd(x,y,z)=1$ such that $ix+jy=kz$, $ix'\not\equiv0\pmod k $ for every $x'\in [1,x]$, and $jy'\not\equiv0\pmod k $ for every $y'\in [1,y]$. 
	Note that $z$ is even and either  $x$ or $y$ must be odd. By symmetry, we may suppose $x$ is odd. 
	If $y$ is even, we 
	let $S=(\alpha^i)^{[x]}\bdot (\alpha^{-k})^{[z]}\bdot \tau^{[2]}$. Then $S$ is not a product-one sequence. Note that $x+k$ and $z+i$ are both even, whence $(\alpha^{i})^{[x+k]}\bdot (\alpha^{-k})^{[z+i]}\bdot \tau^{[2]}$ is a product-one sequence. It follows by Lemma \ref{new3}.1 that
	\begin{align*}
	S=\frac{(\alpha^{i})^{[x+k]}\bdot (\alpha^{-k})^{[z+i]}\bdot \tau^{[2]}}{ (\alpha^{i})^{[k]}\bdot (\alpha^{-k})^{[i]}}
	=\frac{(\alpha^{i})^{[x]}\bdot (\alpha^j)^{[y]}\bdot (\alpha^{-k})^{[z]}\bdot \tau^{[4]}}{(\alpha^{j})^{[y]}\bdot \tau^{[2]}}\in \bigcap_{\mathfrak p\in \mathfrak X(\mathcal B(G_0))}\mathcal B(G_0)_{\mathfrak p}\,,
	\end{align*}
	whence $\mathcal B(G_0)$ is not weakly Krull, a contradiction.
		If $y$ is odd, we 
	let $S=(\alpha^i)^{[x]}\bdot (\alpha^{-k})^{[z+j]}\bdot \tau^{{2}}$. Then $S$ is not a product-one sequence. It follows by Lemma \ref{new3}.1 that
	\begin{align*}
	S=\frac{(\alpha^{i})^{[x+k]}\bdot (\alpha^{-k})^{[z+i+j]}\bdot \tau^{[2]}}{ (\alpha^{i})^{[k]}\bdot (\alpha^{-k})^{[i]}}
	=\frac{(\alpha^{i})^{[x]}\bdot (\alpha^j)^{[y+k]}\bdot (\alpha^{-k})^{[z+j]}\bdot \tau^{[4]}}{(\alpha^{j})^{[y+k]}\bdot \tau^{[2]}}\in \bigcap_{\mathfrak p\in \mathfrak X(\mathcal B(G_0))}\mathcal B(G_0)_{\mathfrak p}\,,
	\end{align*}
	whence $\mathcal B(G_0)$ is  not weakly Krull, a contradiction.

	Suppose $i\neq st$.  Then Lemma \ref{new1} implies that there exist
	$x,y,z\in \N$ with $\gcd(x,y,z)=1$ such that $ix=jy+kz$, $kz'\not\equiv0\pmod i $ for every $z'\in [1,z]$, and $jy'\not\equiv0\pmod i $ for every $y'\in [1,y]$. 
	Note that $z$ is even and either  $x$ or $y$ must be odd. 
	If $x$ is odd and $y$ is even, we 
	let $S=(\alpha^{-i})^{[x]}\bdot (\alpha^{k})^{[z]}\bdot \tau^{[2]}$. Then $S$ is not a product-one sequence. It follows by Lemma \ref{new3}.1 that
	\begin{align*}
	S=\frac{(\alpha^{-i})^{[x+k]}\bdot (\alpha^{k})^{[z+i]}\bdot \tau^{[2]}}{ (\alpha^{-i})^{[k]}\bdot (\alpha^{k})^{[i]}}
	=\frac{(\alpha^{-i})^{[x]}\bdot (\alpha^j)^{[y]}\bdot (\alpha^{k})^{[z]}\bdot \tau^{[4]}}{(\alpha^{j})^{[y]}\bdot \tau^{[2]}}\in \bigcap_{\mathfrak p\in \mathfrak X(\mathcal B(G_2))}\mathcal B(G_2)_{\mathfrak p}\,,
	\end{align*}
	whence $\mathcal B(G_2)$ is  not weakly Krull. Then  $\mathcal B(G_0)$ is  not weakly Krull, a contradiction.
	If $x$ and $y$ are both odd, we 
	let $S=(\alpha^{-i})^{[x]}\bdot (\alpha^{j})^{[y]}\bdot \tau^{[2]}$. Then $S$ is not a product-one sequence. It follows by Lemma \ref{new3}.1 that 
	\begin{align*}
	S=\frac{(\alpha^{-i})^{[x+j_1]}\bdot (\alpha^{j})^{[y+i_1]}\bdot \tau^{[2]}}{ (\alpha^{-i})^{[j_1]}\bdot (\alpha^{j})^{[i_1]}}
	=\frac{(\alpha^{-i})^{[x]}\bdot (\alpha^j)^{[y]}\bdot (\alpha^{k})^{[z]}\bdot \tau^{[4]}}{(\alpha^{k})^{[z]}\bdot \tau^{[2]}}\in \bigcap_{\mathfrak p\in \mathfrak X(\mathcal B(G_2))}\mathcal B(G_2)_{\mathfrak p}\,,
	\end{align*}
	whence $\mathcal B(G_2)$ is  not weakly Krull. Then  $\mathcal B(G_0)$ is  not weakly Krull, a contradiction.
	If $x$ is even and $y$ is odd,  we 
	let $S=(\alpha^{-i})^{[x+k]}\bdot (\alpha^{j})^{[y]}\bdot \tau^{[2]}$. Then $S$ is not a product-one sequence. It follows by Lemma \ref{new3}.1 that
	\begin{align*}
	S=\frac{(\alpha^{-i})^{[x+k+j_1]}\bdot (\alpha^{j})^{[y+i_1]}\bdot \tau^{[2]}}{ (\alpha^{-i})^{[j_1]}\bdot (\alpha^{j})^{[i_1]}}
	=\frac{(\alpha^{-i})^{[x+k]}\bdot (\alpha^j)^{[y]}\bdot (\alpha^{k})^{[z+i]}\bdot \tau^{[4]}}{(\alpha^{k})^{[z+i]}\bdot \tau^{[2]}}\in \bigcap_{\mathfrak p\in \mathfrak X(\mathcal B(G_2))}\mathcal B(G_2)_{\mathfrak p}\,.
	\end{align*}
	Then $\mathcal B(G_2)$ is not weakly Krull and hence $\mathcal B(G_0)$ is  not weakly Krull, a contradiction.

	Suppose $j\neq rt$. Similarly, we can prove that $\mathcal B(G_1)$ is not weakly Krull and hence $\mathcal B(G_0)$ is  not weakly Krull, a contradiction.	
\end{proof}

\begin{proposition}\label{new5}Let $G$ be an infinite dihedral group, say $G = \langle \alpha, \tau \colon \tau^2=1, \alpha \tau = \tau \alpha^{-1} \rangle$.
	Let $I$ be a set of positive integers with $|I|\ge 2$  and $\gcd(I)=1$, let $J\subseteq I$ be a subset with $1\le |J|<|I|$, and let $G_0=\{\tau\}\cup \{\alpha^i\colon i\in I\setminus J\}\cup \{\alpha^{-j}\colon j\in J\}$. Then $\mathcal B(G_0)$ is weakly Krull if and only if  there exist pairwise co-prime positive integers $b_i$, for $i\in I$, such that $kb_k=\prod_{i\in I}b_i$ for every $k\in I$.
\end{proposition}
\begin{proof}
	Suppose  $\mathcal B(G_0)$ is weakly Krull. If $|I|=2$, say $I=\{i,j\}$, then the assertion follows by letting $b_i=j$ and $b_j=i$. Now suppose $|I|\ge 3$.
	 Let $\{i,j,k\}$ be a subset of $I$. By Lemma \ref{new3}.2, we may assume that $i,j\in I\setminus J$ and $k\in J$.
	If $\mathcal B(\{\tau, \alpha^i,\alpha^j,\alpha^{-k}\})$ is not weakly Krull, then by Lemma \ref{new3}.1 there exist $S\in \mathcal F(\{\tau, \alpha^i,\alpha^j,\alpha^{-k}\})\setminus \mathcal B(\{\tau, \alpha^i,\alpha^j,\alpha^{-k}\})$  and $S_a\in \mathcal B(\{\tau, \alpha^i,\alpha^j,\alpha^{-k}\})$ with $a\not\in\supp(S_a)$ for all $a\in \{\tau, \alpha^i,\alpha^j,\alpha^{-k}\}$ such that $S\bdot S_a\in \mathcal B(\{\tau, \alpha^i,\alpha^j,\alpha^{-k}\})$ for every $a\in \{\tau, \alpha^i,\alpha^j,\alpha^{-k}\}$.
	 Since $S\in \mathcal F(G_0)\setminus \mathcal B(G_0)$ and $\supp(S_{\tau})\cap (G_0\setminus \{\tau, \alpha^i,\alpha^j,\alpha^{-k}\})=\emptyset$, we obtain that
	 $S\in \bigcap_{\mathfrak p\in \mathfrak X(\mathcal B(G_0))}\mathcal B(G_0)_{\mathfrak p}\setminus \mathcal B(G_0)$, whence $\mathcal B(G_0)$ is not weakly Krull, a contradiction.
	 Thus $\mathcal B(\{\tau, \alpha^i,\alpha^j,\alpha^{-k}\})$ is  weakly Krull. Let $d=\gcd(i,j,k)$. Then $\mathcal B(\{\tau, \alpha^{i/d},\alpha^{j/d},\alpha^{-k/d}\})$ is also weakly Krull.
	 It follows by Lemma \ref{new4} that there exist pairwise co-prime positive integers $r,s,t$ such that $i/d=st$, $j/d=rt$, and $k/d=rs$, whence $\gcd(j,k)=dr$, $\gcd(i,j)=dt$, and $\gcd(i,k)=ds$. Therefore $i=dst=\frac{d^3str^2}{d^2r^2}=\frac{\gcd(i,j,k)jk}{\gcd(j,k)^2}$, $j=drt=\frac{d^3rts^2}{d^2s^2}=\frac{\gcd(i,j,k)ik}{\gcd(i,k)^2}$,
	  and $k=drs=\frac{d^3rst^2}{d^2t^2}=\frac{\gcd(i,j,k)ij}{\gcd(i,j)^2}$.
	  The assertion follows by Lemma  \ref{new2}.

Suppose there exist pairwise co-prime positive integers $b_i$, for $i\in I$, such that $kb_k=\prod_{i\in I}b_i$ for every $k\in I$.
Then $\gcd(I\setminus\{i\})=b_i$ for all $i\in I$.
For every sequence $T\in \mathcal F(G_0)$, we define $\mathsf v_i(T)=\mathsf v_{\alpha^i}(T)+\mathsf v_{\alpha^{-i}}(T)$ (note that $\mathsf v_{\alpha^i}(T)=0$ or $\mathsf v_{\alpha^{-i}}(T)=0$) for all $i\in I$. 

\medskip
\noindent
{\bf Claim C: }{\it Let $S\in \mathcal F(G_0)$ such that $\mathsf v_{\tau}(S)$ is a positive even integer.  Then  $S\in \mathcal B(G_0)$ if and only if there exist $x_i,y_i\in \N_0$, for $i\in I$, such that $\mathsf v_i(S)=x_ib_i+2y_i$ for every $i\in I$ and $\sum_{i\in I}x_i$ is even.} 

\medskip
\noindent{\it Proof of Claim C. } $(\Rightarrow)$
Suppose $S\in \mathcal B(G_0)$. There exist subsequences $T_1,T_2$ over $G_0\setminus\{\tau\}$ such that $S=T_1\bdot \tau\bdot T_2\bdot \tau^{[\mathsf v_{\tau}(S)-1]}$ and $\pi(T_1)=\pi(T_2)$. 
Suppose $T_1=\prod_{i\in I\setminus J}(\alpha^i)^{[k_i]} \bdot \prod_{j\in  J}(\alpha^{-j})^{[k_j]}$ and $T_2=\prod_{i\in I\setminus J}(\alpha^i)^{[r_i]} \bdot \prod_{j\in  J}(\alpha^{-j})^{[r_j]}$. Then $\sum_{i\in I\setminus J}ik_i-\sum_{j\in J}jk_j=\sum_{i\in I\setminus J}ir_i-\sum_{j\in J}jr_j$, whence $b_{\ell}=\gcd(I\setminus \{\ell\})$ divides $\ell(k_{\ell}-r_{\ell})$ for every $\ell\in I$. Since $\gcd(I)=1$, we obtain that $\gcd(b_{\ell}, \ell)=1$ and hence $b_{\ell}\t (k_{\ell}-r_{\ell})$  for every $\ell\in I$. Let $x_{\ell}'=(k_{\ell}-r_{\ell})/b_{\ell}$ and $x_{\ell}=|x_{\ell}'|$ for every $\ell\in I$. Then 
$\mathsf v_{\ell}(S)=\mathsf v_{\ell}(T_1\bdot T_2)=k_{\ell}+r_{\ell}=x_{\ell}b_{\ell}+2\min\{k_{\ell}, r_{\ell}\}$ for every $\ell\in I$. Since
 $$\sum_{i\in I\setminus J}ib_ix_i'=\sum_{i\in I\setminus J}i(k_i-r_i)=\sum_{j\in  J}j(k_j-r_j)=\sum_{j\in J}jb_jx_j'\,,$$
and $ib_i=jb_j$ for all $i,j\in I$, we obtain that $\sum_{i\in I\setminus J}x_i'-\sum_{j\in J}x_j'=0$, whence $$\sum_{\ell\in I}x_{\ell}\equiv \sum_{\ell\in I}x_{\ell}'\equiv \sum_{i\in I\setminus J}x_i'-\sum_{j\in J}x_j'\equiv 0\pmod 2\,.$$

$(\Leftarrow)$ Suppose there exist $x_i,y_i\in \N_0$, for $i\in I$, such that $\mathsf v_i(S)=x_ib_i+2y_i$ for every $i\in I$ and $\sum_{i\in I}x_i$ is even. Then there exist $c_i,d_i\in \N_0$, for $i\in I$, such that $\sum_{i\in I}c_i=\sum_{i\in I}d_i$ and $c_i+d_i=x_i$ for every $i\in I$.
Let $$T_1={\small \prod_{i\in I\setminus J}}(\alpha^i)^{[c_ib_i+y_i]}\bdot \prod_{j\in J}(\alpha^{-j})^{[d_jb_j+y_j]} \quad \text{ and } \quad T_2=\prod_{i\in I\setminus J}(\alpha^i)^{[d_ib_i+y_i]}\bdot \prod_{j\in J}(\alpha^{-j})^{[c_jb_j+y_j]}\,.$$ Then $\pi(T_1)=\pi(T_2)$ and hence $S=T_1\bdot \tau\bdot T_2\bdot\tau^{[\mathsf v_{\tau}(S)-1]}$ is a product-one sequence.

\qed[Claim C.]
	

\medskip
Let $S\in \cap_{\mathfrak p\in \mathfrak X(\mathcal B(G_0))}\mathcal B(G_0)_{\mathfrak p}$. Then Lemma \ref{new3}.1 implies that $S\in  \mathcal F(G_0)$. It suffices to show that $S$ is a product-one sequence.

Since there exists $S_{\tau}\in \mathcal B(G_0)$ with $\tau\not\in \supp(S_{\tau})$  such that $S\bdot S_{\tau}$ is a product-one sequence, we obtain that $\mathsf v_\tau(S)$ is even. If $\mathsf v_\tau(S)=0$, then $S$ is a product-one sequence. Now we suppose $\mathsf v_{\tau}(S)$ is a positive even integer. For every $i\in I$, there exists $S_i\in \mathcal B(G_0)$ with $\{\alpha^i, \alpha^{-i}\}\cap \supp(S_i)=\emptyset$ such that $S\bdot S_i$ is a product-one sequence.
Then {\bf Claim C} implies that
 $\mathsf v_i(S)=\mathsf v_i(S\bdot S_i)=r_ib_i+2s_i$ for some $r_i,s_i\in \N_0$ and $\mathsf v_i(S_{\tau})=\mathsf v_i(S_{\tau}\bdot \tau^{[2]})=r_i'b_i+2s_i'$ for some $r_i',s_i'\in \N_0$ with $\sum_{i\in I}r_i'$ even.

\medskip
\noindent{\bf Case 1:}  The $b_i$ are odd for all $i\in I$.

Since $S\bdot S_{\tau}$ is a product-one sequence, it follows by {\bf Claim C} that there exist $x_i,y_i\in \N_0$, for $i\in I$, such that $\mathsf v_i(S\bdot S_{\tau})=x_ib_i+2y_i$ for every $i\in I$ and $\sum_{i\in I}x_i$ is even. Therefore $$\sum_{i\in I}{r_i}\equiv \sum_{i\in I}\mathsf v_i(S)\equiv \sum_{i\in I}\mathsf v_i(S)+\sum_{i\in I}r_i'\equiv \sum_{i\in I}\mathsf v_i(S\bdot S_{\tau})\equiv \sum_{i\in I}x_i\equiv 0\pmod 2\,,$$
whence $S$ is a product-one sequence by {\bf Claim C}.

\medskip
\noindent{\bf Case 2:} There exists $i_0\in I$ such that $b_{i_0}$ is even.

Then $b_j$ is odd for all $j\in I\setminus \{i_0\}$ and $\mathsf v_{i_0}(S)=r_{i_0}b_{i_0}+2s_{i_0}$ is even. We may suppose $\mathsf v_{i_0}(S)=r_{i_0}b_{i_0}+2s_{i_0}$ with $2s_{i_0}<b_{i_0}$, where $r_{i_0},s_{i_0}\in \N_0$. If $\sum_{i\in I\setminus\{i_0\}}r_i$ is even, then $S\bdot (\alpha^{i_0})^{[-\mathsf v_{i_0}(S)]}$ is a product-one sequence by {\bf Claim C}. It follows by the fact that $\mathsf v_{i_0}(S)$ is even that $S$ is a product-one sequence. If $\sum_{i\in I\setminus\{i_0\}}r_i$ is odd and $r_{i_0}\ge 1$, then $S\bdot (\alpha^{i_0})^{[-\mathsf v_{i_0}(S)+b_{i_0}]}$ is a product-one sequence by {\bf Claim C}. It follows by the fact that $\mathsf v_{i_0}(S)-b_{i_0}$ is even that $S$ is a product-one sequence.

 Suppose $\sum_{i\in I\setminus\{i_0\}}r_i$ is odd and $r_{i_0}=0$.
 Let $S_{0}$ be a product-one sequence with $\mathsf v_{i_0}(S_{0})=0$ such that $S\bdot S_{0}$ is a product-one sequence. Then {\bf Claim C} implies that $\mathsf v_i(S_{0})=\mathsf v_i(S_0\bdot \tau^{[2]})=r_i''b_i+2s_i''$ for some $r_i'', s_i''\in \N_0$ with $\sum_{i\in I}r_i''$  even and that there exist $x_i,y_i\in \N_0$, for $i\in I$, such that $\mathsf v_i(S\bdot S_{0})=x_ib_i+2y_i$ for every $i\in I$ and $\sum_{i\in I}x_i$ is even. Since $\mathsf v_{i_0}(S\bdot S_{0})=\mathsf v_{i_0}(S)=2s_{i_0}<b_{i_0}$, we obtain that $x_{i_0}=0$.
  Therefore
$$\sum_{i\in I\setminus\{i_0\}}r_i\equiv \sum_{i\in I}\mathsf v_i(S)\equiv \sum_{i\in I}\mathsf v_i(S)+\sum_{i\in I}r_i''\equiv \sum_{i\in I}\mathsf v_i(S\bdot S_{0})\equiv \sum_{i\in I\setminus\{i_0\}}x_i\equiv \sum_{i\in I}x_i\equiv0\pmod 2\,,$$
 a contradiction.
	\end{proof}

For the notions that are used in the next proposition but are not introduced in this work we refer to \cite{HK98}.

\begin{proposition}\label{product:weaklyKrull}
Let $H_1$ and $H_2$ be monoids. Then $H_1\times H_2$ is weakly Krull if and only if $H_1$ and $H_2$ are both weakly Krull.
\end{proposition}
\begin{proof}

Let $H=H_1\times H_2$.
We need the following claims.

\smallskip
\noindent \textbf{Claim A. } 
{\it Let $I$ be a nonempty $t$-ideal of $H$. Then $I=I_1\times I_2$, where
\begin{align*}
I_1&=\{x\in H_1\colon \text{ there exists } y\in H_2 \text{ such that } (x,y)\in I\} \\
\text{ and }
I_2&=\{y\in H_2\colon \text{ there exists } x\in H_1 \text{ such that }(x,y)\in I\}\,.
\end{align*}
Furthermore, if $I$ is a prime $s$-ideal, then $I_1=H_1$ or $I_2=H_2$.}

\begin{proof}[Proof of Claim A]
Clearly $I\subseteq I_1\times I_2$. It remains to prove the converse inclusion. Let $(a,b)\in I_1\times I_2$. Then there exist $c\in H_2$ and $d\in H_1$ such that $(a,c),(d,b)\in I$. Let $(x,y)\in (H\colon \{(a,c),(d,b)\})$. Then $(ax,cy), (dx,by)\in H$, whence $ax\in H_1$ and $by\in H_2$. It follows that $(a,b)(x,y)\in H$ and hence $(a,b)\in \{(a,c),(d,b)\}_v$.
 Since $I$ is a $t$-ideal, we have $(a,b)\in \{(a,c),(d,b)\}_v\subset I$. 
 
 For the "furthermore" part, suppose $I$ is a prime $s$-ideal.  Let $(a,b)\in I$. Then $(a,1)(1,b)\in I$ implies that $(a,1)\in I$ or $(1,b)\in I$. If $(a,1)\in I$, then $1\in I_2$ and hence $I_2=H_2$. If $(1,b)\in I$, then $1\in I_1$ and hence $I_1=H_1$.
\qedhere[Proof of Claim A.]
\end{proof}

\smallskip
\noindent \textbf{Claim B. }{\it  $\mathfrak X(H)=\{P_1\times H_2\colon P_1\in\mathfrak X(H_1)\}\cup \{H_1\times P_2\colon P_2\in\mathfrak X(H_2)\}$. }
 
 \begin{proof}[Proof of Claim B]
 We first show that $\mathfrak X(H)\supset\{P_1\times H_2\colon P_1\in\mathfrak X(H_1)\}$. 
 Let $P_1\in \mathfrak X(H_1)$. It is clear that $P_1\times H_2$ is a prime $s$-ideal of $H$. By \cite[Proposition 11.6(ii)]{HK98} it follows that $P_1\times H_2$ contains a nonemtpy prime $t$-ideal $I$. We will show that $P_1\times H_2$ equals every nonempty prime $t$-ideal it contains. Then $P_1\times H_2$ is a minimal prime $t$-ideal. By using \cite[Proposition 11.6(ii)]{HK98} again, we obtain that $P_1\times H_2\in\mathfrak X(H)$.
 Let $I\subset P_1\times H_2$ be a nonempty prime $t$-ideal. Then Claim A implies $I=I_1\times H_2$, where $I_1\subset P_1$ is a prime  $s$-ideal.
 Hence the minimality of $P_1$ implies that $I_1=P_1$. 
 
 By symmetry we can show that $\mathfrak X(H)\supset\{H_1\times P_2\colon P_2\in\mathfrak X(H_2)\}$. Now we prove that $\mathfrak X(H)\subset\{P_1\times H_2\colon P_1\in\mathfrak X(H_1)\}\cup \{H_1\times P_2\colon P_2\in\mathfrak X(H_2)\}$. Let $I\in \mathfrak X(H)$.  Then \cite[Proposition 11.6(iii)]{HK98} shows that $I$ is a $t$-ideal and Claim A implies that $I=I_1\times H_2$ or $I=H_1\times I_2$, where $I_1\subset H_1$ and $I_2\subset H_2$ are prime $s$-ideals. The minimality of $I$ and $\mathfrak X(H)\supset\{P_1\times H_2\colon P_1\in\mathfrak X(H_1)\}\cup \{H_1\times P_2\colon P_2\in\mathfrak X(H_2)\}$  imply that $I_1$ and $I_2$ must be minimal, and we complete the proof.
 \qedhere[Proof of Claim B.]	 
 	\end{proof}

Suppose $H_1$ and $H_2$ are weakly Krull.  
Let $(a,b)\in H\setminus H^\times$. Then $a\notin H_1^\times$ or $b\notin H_2^\times$, say $a\notin H_1^\times$. Then by \cite[Theorem 22.7]{HK98} there exist primary ideals $Q_1,\ldots ,Q_n$ of $H_1$ such that $aH_1=Q_1\cap\ldots\cap Q_n$ and $\sqrt{Q_i}=P_i\in \mathfrak X(H_1)$ for all $i\in [1,n]$. If $b\in H_2^\times$, then $(a,b)H=(aH_1)\times H_2=\bigcap_{i=1}^n Q_i\times H_2$. By Claim B, we obtain that $\sqrt{Q_i\times H_2}=P_i\times H_2\in \mathfrak X(H)$ for all $i\in [1,n]$. If $b\notin H_2^\times$, then there exist primary ideals $L_1,\ldots ,L_m\subset H_2$ such that $bH_2=L_1\cap\ldots\cap L_m$ and $\sqrt{L_i}=N_i\in \mathfrak X(H_2)$ for all $j\in [1,m]$. Then $(a,b)H=(\bigcap_{i=1}^n Q_i\times H_2)\cap (\bigcap_{j=1}^m H_1\times L_j)$ and Claim B implies that all radicals of those ideals are in $\mathfrak X(H)$.  It follows by \cite[Theorem 22.7]{HK98} that both cases imply that $H$ is weakly Krull.

Conversely, suppose $H$ is weakly Krull. By symmetry, we only need to show that $H_1$ is weakly Krull. Let $a\in H_1\setminus H_1^\times$. Then $(a,1)\in H\setminus H^\times$ and by \cite[Theorem 22.7]{HK98} there exist $n\in \N$ and primary $t$-ideals $I_1,\ldots ,I_n\subset H$ such that $(a,1)H=(aH_1)\times H_2=I_1\cap\ldots\cap I_n$ and $\sqrt{I_i}\in \mathfrak X(H)$ for all $i\in [1,n]$. Since $(aH_1)\times H_2\subset I_i$, it follows by  Claim A that $I_i=Q_i\times H_2$ for every $i\in [1,n]$.
It is easy to see that each $Q_i$ must be primary. Since $\sqrt{I_i}=\sqrt{Q_i}\times H_2\in \mathfrak X(H)$, Claim B implies that $\sqrt{Q_i}\in \mathfrak X(H_1)$ for all $i\in [1,n]$.
Note that $aH_1=Q_1\cap\ldots\cap Q_n$.  It follows by \cite[Theorem 22.7]{HK98}  that $H_1$ is weakly Krull.
\end{proof}

\begin{corollary}\label{corollary:1doesntmatter}
Let $G$ be a group and let $G_0\subset G$ be a subset. Then $\mathcal B(G_0)$ is weakly Krull if and only if $\mathcal B(G_0\setminus\{1\})$ is weakly Krull.
\end{corollary}
\begin{proof}
If $1\notin G_0$ the statement is trivial, so let $1\in G_0$. Since $1$ (considered as a sequence) has the property that $S\in \mathcal B(G_0)$ if and only if $S\bdot 1\in \mathcal B(G_0)$ for all $S\in \mathcal F(G_0)$, it is clear that $\mathcal B(G_0\setminus\{1\})\times \mathcal F(\{1\})\cong \mathcal B(G_0)$ (the isomorphism being $(S,1^{[n]})\mapsto S\bdot 1^{[n]})$. Now the assertion follows by Proposition \ref{product:weaklyKrull}.
\end{proof}

\smallskip
\begin{proof}[Proof of Theorem \ref{5.1}.2]
By Corollary \ref{corollary:1doesntmatter} we can assume that $G_0$ does not contain the element $1$.
	If $G_0\subset\langle \alpha\rangle$ or $|G_0|=1$, then $\mathcal B(G_0)$ is Krull and hence weakly Krull. For every $g\in G_0$, we let $\mathfrak p_g=\{S\in \mathcal B(G_0)\colon g\in \supp(S)\}$.
	 We distinguish three cases depending on $G_0$ and in each case we prove the asserted equivalence. 
	
	\smallskip
	\noindent{\bf Case 1:} $G_0\subset \langle \alpha\rangle\tau$ with $|G_0|\ge 2$.
	\smallskip
	
	 If $|G_0|= 2$, say $G_0=\{\alpha^i\tau,\alpha^j\tau\}$, where $i,j\in \Z$ are distinct, then $\mathcal A(G_0)=\{(\alpha^i\tau)^{[2]},(\alpha^j\tau)^{[2]}\}$, which implies $\mathcal B(G_0)$ is factorial and hence weakly Krull.
	
	Suppose $|G_0|=3$. We will show $\mathcal B(G_0)$ is weakly Krull.
	Let $G_0=\{\alpha^i\tau,\alpha^j\tau,\alpha^k\tau\}$, where $i,j,k\in \Z$ are distinct. 
	
	\medskip
	\noindent{\bf Claim D: }{\it Let $d=\gcd(|k-i|,|k-j|,|j-i|)$. Then $$\mathcal A(\mathcal B(G_0))=\{(\alpha^i\tau)^{[2]}, (\alpha^j\tau)^{[2]},  (\alpha^k\tau)^{[2]}, A:=(\alpha^i\tau)^{[|k-j|/d]}\bdot (\alpha^j\tau)^{[|k-i|/d]}\bdot (\alpha^k\tau)^{[|j-i|/d]}\}\,.$$
	}
	\smallskip
	
	\noindent{\it Proof of  Claim D. } 
	Let $S=(\alpha^i\tau)^{[x]}\bdot (\alpha^j\tau)^{[y]}\bdot (\alpha^k\tau)^{[z]}\in \mathcal F(G_0)$, where $x,y,z\in \N_0$. We define the map $\phi\colon \mathcal F(G_0)\rightarrow \Z$ by $\phi(S)=ix+jy+kz$.  
	To show the assertion, we will use the fact that 
	$S\in \mathcal B(G_0)$ if and only if there exist subsequences  $T_1$ and $T_2$ such that $S=T_1\bdot T_2$, $|T_1|=|T_2|$, and $\phi(T_1)=\phi(T_2)$ (see  {\bf Claim A} of Proof of Theorem \ref{5.1}.1). In particular, if $S$ is an atom, then $\supp(T_1)\cap \supp(T_2)=\emptyset$, since otherwise $h^{[2]}\bdot (S\bdot h^{[-2]})$ would be a decomposition of $S$ for every $h\in\supp(T_1)\cap\supp(T_2)$ by \textbf{Claim A}; contradicting the fact that $S$ is an atom.

	Let $A_0\in \mathcal A(\mathcal B(G_0))$. If $|\supp(A_0)|=1$, then $A_0\in \{(\alpha^i\tau)^{[2]}, (\alpha^j\tau)^{[2]},  (\alpha^k\tau)^{[2]}\}$. 
	
	Suppose $|\supp(A_0)|=2$, say $\supp(A_0)=\{\alpha^i\tau, \alpha^j\tau\}$. Let $T_1$ and $T_2$ be the subsequences such that $A_0=T_1\bdot T_2$, $|T_1|=|T_2|$, $\phi(T_1)=\phi(T_2)$, $\alpha^i\tau\in \supp(T_1)$.
	Then $\supp(T_1)\cap \supp(T_2)=\emptyset$ and hence   $T_1=(\alpha^i\tau)^{[\mathsf v_{\alpha^i\tau}(A_0)]}$ and $T_2=(\alpha^j\tau)^{[\mathsf v_{\alpha^j\tau}(A_0)]}$. Therefore 
	 $\mathsf v_{\alpha^i\tau}(S)=\mathsf v_{\alpha^j\tau}(S)$ and  $\mathsf v_{\alpha^i\tau}(S)i=\mathsf v_{\alpha^j\tau}(S)j$, a contradiction.
	
	Suppose $|\supp(A_0)|=3$. After renumbering if necessary, we may suppose $i<j<k$ and $A_0=(\alpha^i\tau)^{[x]}\bdot (\alpha^j\tau)^{[y]}\bdot (\alpha^k\tau)^{[z]}$, where $x,y,z\in \N$. Let $T_1$ and $T_2$ be the subsequences such that $A_0=T_1\bdot T_2$, $|T_1|=|T_2|$,  $\phi(T_1)=\phi(T_2)$, and $\alpha^i\tau\in \supp(T_1)$. Then $\supp(T_1)\cap \supp(T_2)=\emptyset$. If $\supp(T_1)=\{\alpha^i\tau\}$, then $T_1=(\alpha^i\tau)^{[x]}$ and $T_2=(\alpha^j\tau)^{[y]}\bdot (\alpha^k\tau)^{[z]}$, whence
	$x=y+z$ and $ix=jy+kz>iy+iz$, a contradiction.
	 If $\supp(T_1)=\{\alpha^i\tau, \alpha^j\tau\}$, then $T_1=(\alpha^i\tau)^{[x]}\bdot (\alpha^j\tau)^{[y]}$ and $T_2=(\alpha^k\tau)^{[z]}$, whence
	 $x+y=z$ and $ix+jy=kz=kx+ky>ix+jy$, a contradiction.
	Therefore $\supp(T_1)=\{\alpha^i\tau, \alpha^k\tau\}$ and hence $T_1=(\alpha^i\tau)^{[x]}\bdot (\alpha^k\tau)^{[z]}$ and $T_2=(\alpha^j\tau)^{[y]}$.
	It follows that $ix+kz=jy$ and $x+z=y$.
	Since $\big((k-j)/d, (k-i)/d, (j-i)/d\big)$ is a solution of the above linear equations, we obtain that $(x,y,z)=r\big((k-j)/d, (k-i)/d, (j-i)/d\big)$, where $r$ is a positive rational. Since $\gcd\big((k-j)/d, (k-i)/d, (j-i)/d\big)=1$, we have that $r\in \N$ and hence $r=1$, since $A_0$ is an atom.
	\qed[Claim D.]
	\medskip
	
	Since $G_0$ consists of torsion elements, it follows  by Proposition \ref{3.8} that $\mathfrak X(\mathcal B(G_0))=\{\mathfrak p_g\colon g\in G_0\}$ is finite.
	Now let $S\in\bigcap_{g\in G_0}\mathcal B(G_0)_{\mathfrak p_g}$. Then $S\in \mathcal F(G_0)$ and $|S|$ is even. By definition of weakly Krull monoids, we need to show $S\in\mathcal B(G_0)$. If $|\supp(S)|=1$, then $S\in\mathcal B(G_0)$. If $|\supp(S)|=2$, say $\supp(S)=\{\alpha^i\tau, \alpha^j\tau\}$,  then there exist $n,m\in\N$ such that $S=(\alpha^i\tau)^{[n]}\bdot (\alpha^j\tau)^{[m]}$. It follows by the fact that $|S|$ is even that   $n$ and $m$ are either both even or both odd. If both are even,  then $S\in\mathcal B(G_0)$. Suppose  both are odd.  Since $S\in \mathcal B(G_0)_{\mathfrak p_{\alpha^k\tau}}$, we obtain there exists $S_1\in\mathcal B(G_0)\setminus \mathfrak p_{\alpha^k\tau}$ such that $S\bdot S_1\in\mathcal B(\{\alpha^i\tau, \alpha^j\tau\})$. It follows that $\mathsf v_{\alpha^i\tau}(S\bdot S_1)$ and $\mathsf v_{\alpha^j\tau}(S\bdot S_1)$ are both even, whence $\mathsf v_{\alpha^i\tau}(S_1)$ and $\mathsf v_{\alpha^j\tau}(S_1)$ are both odd, a contradiction to $S_1\in \mathcal B(G_0)\setminus \mathfrak p_{\alpha^k\tau}=\mathcal B(\{\alpha^i\tau, \alpha^j\tau\})$.
	
	So it remains to treat the case when $|\supp(S)|=3$. Assume $S=(\alpha^i\tau)^{[k_1]}\bdot (\alpha^j\tau)^{[k_2]}\bdot (\alpha^k\tau)^{[k_3]}$, where $k_1,k_2,k_3\in\N$ with $k_1+k_2+k_3$ even. If all three of them are even, then $S\in \mathcal B(G_0)$. Otherwise, after renumbering if necessary, we may assume $k_1, k_2$ are odd and $k_3$ is even. Since $S\in \mathcal B(G_0)_{\mathfrak p_{\alpha^k\tau}}$, we obtain some $S_1\in\mathcal B(G_0)\setminus \mathfrak p_{\alpha^k\tau}$ such that $S\bdot S_1\in\mathcal B(G_0)$. 
	It follows by $\supp(S_1)\subset\{\alpha^i\tau, \alpha^j\tau\}$ that $\mathsf v_{\alpha^i\tau}(S_1)$ is even and hence $\mathsf v_{\alpha^i\tau}(S\bdot S_1)$ is odd. By {\bf Claim D}, we may 
	suppose $$S\bdot S_1=A^{[w]}\bdot ((\alpha^i\tau)^{[2]})^{[x]}\bdot ((\alpha^j\tau)^{[2]})^{[y]}\bdot ((\alpha^k\tau)^{[2]})^{[z]}\ ,\quad  \text{ where }w,x,y,z\in \N_0\,.$$
	Then $w\cdot \mathsf v_{\alpha^i\tau}(A)$ and hence  $\mathsf v_{\alpha^i\tau}(A)$ are odd. A similar argument shows $\mathsf v_{\alpha^j\tau}(A)$ is odd and $|A|$ even implies that $\mathsf v_{\alpha^k\tau}(A)$ is even. If $A\mid_{\mathcal F(G_0)}S$, then $\mathsf v_{\alpha^l\tau}(S\bdot A^{[-1]})$ is even for all $l\in \{i,j,k\}$, which implies $S\bdot A^{[-1]}$ and hence $S$ are product-one sequences.
	If $A\nmid_{\mathcal F(G_0)}S$, then there is $l\in\{i,j,k\}$ such that $\mathsf v_{\alpha^l\tau}(S)<\mathsf v_{\alpha^l\tau}(A)$. Since $S\in \mathcal B(G_0)_{\mathfrak p_{\alpha^l\tau}}$,  it follows that there is $S_1\in\mathcal B(G_0\setminus\{\alpha^l\tau\})$ such that $S\bdot S_1\in \mathcal B(G_0)$. Since  $\mathsf v_{\alpha^l\tau}(S\bdot S_1)=\mathsf v_{\alpha^l\tau}(S)<\mathsf v_{\alpha^l\tau}(A)$,  there exist $x,y,z\in\N_0$ such that $S\bdot S_1=(\alpha^i\tau)^{[2x]}\bdot (\alpha^j\tau)^{[2y]}\bdot (\alpha^k\tau)^{[2z]}$, whence $k_1+\mathsf v_{\alpha^i\tau}(S_1)=2x$ is even. Since $k_1$ is odd, we have that $\mathsf v_{\alpha^i\tau}(S_1)$ is odd, a contradiction to the fact that $S_1\in\mathcal B(G_0\setminus\{\alpha^l\tau\})$.
	
Suppose  $|G_0|\geq 4$. Then  there exist distinct $i,j,k,r\in\Z\setminus\{0\}$ such that $\{\alpha^i\tau,\alpha^j\tau,\alpha^k\tau,\alpha^r\tau\}\subset G_0$. In this case, we will show $\mathcal B(G_0)$ is not weakly Krull.
Set $d_i=\gcd(j-k,j-r,k-r)=2^{\beta_i}d_i'$, $d_j=\gcd(i-k,i-r,k-r)=2^{\beta_j}d_j'$, $d_k=\gcd(j-i,j-r,i-r)=2^{\beta_k}d_k'$, and $d_r=\gcd(j-k,j-i,k-i)=2^{\beta_r}d_r'$ such that 
$d_i',d_j',d_k',d_r'$ are all odd.
By symmetry, we may assume that  $\beta_i=\max\{\beta_i,\beta_j,\beta_k,\beta_r\}$. 

Note that $A_i=(\alpha^j\tau)^{[|k-r|/{d_i}]}\bdot (\alpha^k\tau)^{[|j-r|/{d_i}]}\bdot(\alpha^r\tau)^{[|j-k|/{d_i}]}\in \mathcal A(G_0)$ and $|A_i|$ is even. It follows by 
$d_i=\gcd(j-k,j-r,k-r)$ that two of $(j-k)/d_i,(j-r)/d_i,(k-r)/d_i$ are odd.
By symmetry, we may suppose $(k-r)/d_i$ is even and  $(j-r)/d_i, (j-k)/d_i$ are  odd.
Consider $A_j=(\alpha^i\tau)^{[|k-r|/{d_j}]}\bdot (\alpha^k\tau)^{[|i-r|/{d_j}]}\bdot(\alpha^r\tau)^{[|i-k|/{d_j}]}\in \mathcal A(G_0)$. Since $\beta_i\ge \beta_j$, we obtain $(k-r)/d_j$ is even and hence  $(i-r)/d_j, (i-k)/d_j$ are both odd.

Let $S=(\alpha^k\tau)^{[|j-r|/{d_i}]}\bdot(\alpha^r\tau)^{[|j-k|/{d_i}]}$. Then $S\not\in \mathcal B(G_0)$ and
$$S=\frac{A_i}{(\alpha^j\tau)^{[|k-r|/{d_i}]}}=\frac{A_j\bdot (\alpha^k\tau)^{[|j-r|/{d_i}-|i-r|/d_j]}\bdot(\alpha^r\tau)^{[|j-k|/{d_i}-|i-k|/d_j]}  }{(\alpha^i\tau)^{[|k-r|/{d_j}]}}\in \bigcap_{\mathfrak p\in \mathfrak X(\mathcal B(G_0))}\mathcal B(G_0)_{\mathfrak p}\,.$$
By definition of weakly Krull monoids, we obtain $\mathcal B(G_0)$ is not weakly Krull.

\smallskip
\noindent{\bf Case 2:} $|G_0\cap \langle \alpha\rangle|\ge 1$ and $|G_0\cap \langle \alpha\rangle\tau|\ge 2$.
\smallskip

 Then  there are $i\in\Z\setminus\{0\}$ and distinct $j,k\in \Z$ such that $\{\alpha^i,\alpha^j\tau, \alpha^k\tau\}\subset G_0$.
	Therefore $\mathfrak X(\mathcal B(G_0))\subset\{\mathfrak p_{g}\colon g\in G_0\}$ and 
	$$(\alpha^{i})^{[2]}=\frac{(\alpha^{i})^{[2]}\bdot (\alpha^{j}\tau)^{[2]}}{(\alpha^{j}\tau)^{[2]}}=\frac{(\alpha^{i})^{[2]}\bdot (\alpha^{k}\tau)^{[2]}}{(\alpha^{k}\tau)^{[2]}}\in\bigcap_{\mathfrak p\in\mathfrak X(\mathcal B(G_0))}\mathcal B(G_0)_{\mathfrak p}\,.$$
	By definition of weakly Krull monoids, it follows that $\mathcal B(G_0)$ is not weakly Krull.

\smallskip
\noindent{\bf Case 3:} $|G_0\cap \langle \alpha\rangle|\ge 1$ and $|G_0\cap \langle \alpha\rangle\tau|= 1$.
\smallskip

	Suppose $G_0\cap\langle\alpha\rangle\tau=\{\alpha^k\tau\}$, where $k\in \Z$. If  $G_0\setminus\{\alpha^k\tau\}\subset \{\alpha^i\colon i\in \N\}$ or $G_0\setminus\{\alpha^k\tau\}\subset \{\alpha^{-i}\colon i\in \N\}$, then 
	$\mathfrak X(\mathcal B(G_0))=\{\mathfrak p_{g}\colon g\in G_0\setminus\{\mathfrak p_{\alpha^k\tau}\}\}$ and hence for $\alpha^x\in G_0$ we have that $$(\alpha^x)^{[2]}\in\bigcap_{\mathfrak p\in\mathfrak X(\mathcal B(G_0))}\mathcal B(G_0)_{\mathfrak p}\setminus \mathcal B(G_0)\,.$$
	By definition of weakly Krull monoids, it follows  that $\mathcal B(G_0)$ is not weakly Krull.
	
	Now we assume that $G_0\setminus\{\alpha^k\tau\}=\{\alpha^i\colon i\in I\} \cup \{\alpha^{-j}\colon j\in J\}$, where $I,J$ are nonempty sets of positive integers. Let $d=\gcd(I\cup J)$ and let  $G_1=\{\tau\}\cup \{\alpha^{i/d}\colon i\in I\} \cup \{\alpha^{-j/d}\colon j\in J\}$.
	By changing bases, we obtain that $\mathcal B(G_0)$ is weakly Krull if and only if $\mathcal B(G_1)$ is weakly Krull.
	
	 Suppose $|I\cup J|\ge 2$. Let $i_0\in I\cup J$ and $G_2=\{\tau, \alpha^{i_0/d}\}\cup  \{\alpha^{-j/d}\colon j\in (I\cup J)\setminus \{i_0\}\}$. By Lemma \ref{new3}.2, we obtain that $\mathcal B(G_1)$ is weakly Krull if and only if $\mathcal B(G_2)$ is weakly Krull. Now the assertion follows from Proposition \ref{new5}.
	 
	  Suppose $|I\cup J|=1$, say $I\cup J=\{i\}$. Then $G_1=\{\tau, \alpha, \alpha^{-1}\}$. To finish the proof, we only need to show that $\mathcal B(G_1)$ is weakly Krull. 
Let $S\in \bigcap_{\mathfrak p\in \mathfrak X(\mathcal B(G_1))}\mathcal B(G_1)_{\mathfrak p}$. Then Lemma \ref{new3}.1 implies that $S\in  \mathcal F(G_1)$. It suffices to show that $S$ is a product-one sequence.
Since there exists $S_{\tau}\in \mathcal B(G_1)$ with $\tau\not\in \supp(S_{\tau})$  such that $S\bdot S_{\tau}$ is a product-one sequence, we obtain that $\mathsf v_\tau(S)$ is even. If $\mathsf v_\tau(S)=0$, then $S$ is a product-one sequence. Now we suppose $\mathsf v_{\tau}(S)$ is a positive even integer. Since there exists $S_{\alpha}\in \mathcal B(G_1)$ with $\alpha\not\in \supp(S_{\alpha})$ such that $S\bdot S_{\alpha}$ is a product-one sequence, we obtain that $\mathsf v_{\alpha}(S\bdot S_{\alpha})+\mathsf v_{\alpha^{-1}}(S\bdot S_{\alpha})$ and $\mathsf v_{\alpha^{-1}}(S_{\alpha})$ are both even, whence $\mathsf v_{\alpha}(S)+\mathsf v_{\alpha^{-1}}(S)$ is even. It follows that $S$ is a product-one sequence.
\end{proof}

\smallskip
\begin{proof}[Proof of Theorem \ref{5.1}.3] $(c)\Leftrightarrow (e)$ follows by Theorem \ref{3.11}.2, $(a)\Leftrightarrow (b)$ follows by Equation \eqref{basic}, and $(c)\Rightarrow (b)$ follows by Proposition \ref{4.1}.2. It remains to show $(b)\Rightarrow (d)$ and $(d)\Rightarrow (e)$.

$(b)\Rightarrow (d)$ If $G_0\nsubseteq \langle \alpha\rangle$ and $G_0\nsubseteq \langle\alpha\rangle\tau\cup\{1\}$, then there are $i\in\Z\setminus\{0\}$ and $j\in \Z$ such that $\alpha^i, \alpha^j\tau\in G_0$. Then for all $n\in\N$ we have that $A_n=(\alpha^i)^{[2n]}\bdot (\alpha^j\tau)^{[2]}\in\mathcal A(G_0)$. Since $A_n\mid_{\mathcal B(G_0)} ((\alpha^i)^{[2]}\bdot (\alpha^j\tau)^{[2]})^n$ and $A_n\nmid_{\mathcal B(G_0)} ((\alpha^i)^{[2]}\bdot (\alpha^j\tau)^{[2]})^m$ for any $m\in [1,n-1]$, it follows that $\omega (\mathcal B(G_0), A_n)\geq n$ and hence $\omega(G_0)=\infty$.

$(d)\Rightarrow (e)$ If $G_0\subset \langle \alpha\rangle$, then  $\langle G_0\rangle$ is abelian and hence the finiteness of $G_0$ implies $\mathcal B(G_0)$ is  finitely generated.  If $G_0\subset\langle\alpha\rangle\tau\cup\{1\}$, then $G_0$ consists of torsion elements. It follows by Theorem \ref{3.11}.2 that $\mathcal B(G_0)$ is  finitely generated. 
\end{proof}

\smallskip
\begin{proof}[Proof of Theorem \ref{5.1}.4] $(d)\Rightarrow (b)$ is clear.
For $(b)\Rightarrow (c)$, recall that $\rho=\sup\{\rho_k/k:k\geq 2\}$.	
$(a)\Rightarrow (e)$ and $(c)\Rightarrow (e)$ If $G_0\cap \langle\alpha\rangle\tau\neq\emptyset$ and $\mathcal B(G_0\cap\langle\alpha\rangle)\nsubseteq\{1_{\mathcal B(G_0)}, 1_G^{[n]}:n\geq 0\}$, then there exist $i\in\N, j\in\Z_{<0}$, and $k\in\Z$ such that $\{\alpha^i,\alpha^j,\alpha^k\tau\}\subset G_0$. We will show that $\mathsf t(G_0,(\alpha^k\tau)^{[2]})=\infty$ and $\rho_2(G_0)=\infty$. Let $m=\lcm(i,-j)$ and $n_i=m/i$, $n_j=m/(-j)$. Then $(\alpha^i)^{[n_i]}\bdot (\alpha^j)^{[n_j]}$ is an atom.  For every $n\in\N$, let $S_n=(\alpha^k\tau)^{[4]}\bdot (\alpha^i)^{[2n_in]}\bdot (\alpha^j)^{[2n_jn]}$. Then
\begin{equation*}
\mathsf Z(S_n)\cap (\alpha^k\tau)^{[2]}\bdot \mathsf Z(\mathcal B(G_0))=\{z_n':=(\alpha^k\tau)^{[2]}\bdot (\alpha^k\tau)^{[2]}\bdot ((\alpha^i)^{[n_i]}\bdot (\alpha^j)^{[n_j]})^{[2n]}\}.
\end{equation*}
Set $z_n=((\alpha^k\tau)^{[2]}\bdot (\alpha^i)^{[2n_in]})\bdot ((\alpha^k\tau)^{[2]}\bdot (\alpha^j)^{[2n_jn]})\in \mathsf Z(S_n)$. Then $\mathsf t(G_0,(\alpha^k\tau)^{[2]})\geq \mathsf d(z_n,z_n')=2n+2$ and $\rho_2(G_0)\ge 2n+2$, whence $\mathsf t(G_0,(\alpha^k\tau)^{[2]})=\infty$ and $\rho_2(G_0)=\infty$.

%

\medskip

$(e)\Rightarrow (a)$ and $(e)\Rightarrow (d)$ If  $G_0\subset\langle\alpha\rangle$, then it follows by Theorem \ref{5.1}.3 that  $\mathsf D(G_0)<\infty$ and $\mathcal B(G_0)$ is tame. Now Proposition \ref{4.1}.2 implies $\rho(G_0)$ is accepted.

Suppose $\mathcal B(G_0\cap \langle \alpha\rangle)\subset \{1_{\mathcal B(G_0)}, 1_G^{[n]}:n\geq 0\}$ and $G_0\cap\langle\alpha\rangle\tau\neq\emptyset$. First, we prove that for all $A\in\mathcal A(G_0)\setminus \{1_G\}$ we have that
\begin{equation*}
2\le \sum\limits_{g\in G_0\cap\langle\alpha\rangle\tau}\mathsf v_g(A)\leq \max\Big\{\sum\limits_{g\in G_0\cap\langle\alpha\rangle\tau}\mathsf v_g(S)\colon S\in \mathcal A(D)\Big\}=:M,
\end{equation*}
where $D$ is the finitely generated monoid defined in Equation \eqref{equation:D}. The first inequality $2\le \sum\limits_{g\in G_0\cap\langle\alpha\rangle\tau}\mathsf v_g(A)$ follows from
 the fact that $\mathcal B(G_0\cap \langle\alpha\rangle)\subset\{1_{\mathcal B(G_0)}, 1_G^{[n]}:n\geq 0\}$. Let $W\in \mathcal A(G_0)\setminus \mathcal A(D)$. Then $1_G\not\in \supp(W)$. 
  We can write $W=U_1\bdot\ldots\bdot U_r\bdot g_1^{[2]}\bdot\ldots\bdot g_t^{[2]}$, where $r,t\in \N_0$,  $g_1,\ldots, g_t \in G_0\cap \langle\alpha\rangle$, and $U_1,\ldots, U_r\in \mathcal A(D)\cap \mathcal A(G_0)$.
  Thus $\supp(U_i)\cap \langle\alpha\rangle\tau\neq\emptyset$ for every $i\in [1,r]$.
  If $r=0$, then $\supp(B)\in G_0\cap \langle\alpha\rangle$, a contradiction. Suppose $r\ge 2$. Since $U_1'\colon =U_1\bdot g_1^{[2]}\bdot\ldots\bdot g_t^{[2]}\in \mathcal A(G_0)$, we obtain that $W=U_1'\bdot U_2\bdot\ldots\bdot U_r$, a contradiction. Thus $r=1$
  and hence $\sum\limits_{g\in G_0\cap\langle\alpha\rangle\tau}\mathsf v_g(W)= \sum\limits_{g\in G_0\cap\langle\alpha\rangle\tau}\mathsf v_g(U_1)\le M$. The second inequality follows.

  We will show  $\mathsf t(G_0,A)<\infty$ for every $A\in\mathcal A(G_0)$ and $\rho(G_0)=M/2$ is accepted.

  If $A=1_G$ this is clear by \cite[Lemma 1.6.5.2]{Ge-HK06a}, so let $A\in\mathcal A(G_0)\setminus\{1_G\}$. By Equation \eqref{basic1}, 
   it suffices to prove $\omega(G_0, A)$ and $\tau(G_0, A)$ are both finite. Let $n\in\N$ and $U_1,\hdots ,U_n\in \mathcal B(G_0)\setminus\{1_G\}$ such that $A\mid_{\mathcal B(G_0)}U_1\bdot\hdots \bdot U_n$. Let $A=A_1\bdot \hdots \bdot A_m$ be a factorization of $A$ in $D$. Note that $\omega(D)$ is finite by \cite[Theorem 3.1.4]{Ge-HK06a}. Since $\omega(D, A)\le m\omega(D)$, after renumbering if necessary, we may assume $A\mid_D U_1\bdot\ldots\bdot U_{m\omega(D)}$. Set $A'= U_1\bdot\ldots\bdot U_{m\omega(D)}\bdot A^{{-1}}$. Since $\sum\limits_{g\in G_0\cap\langle\alpha\rangle\tau}\mathsf v_g(U_{m\omega(D)+1})\ge 2$, we obtain $A'\bdot  U_{m\omega(D)+1}\in \mathcal B(G_0)$ and hence $A\mid_{\mathcal B(G_0)} U_1\bdot\ldots\bdot U_{m\omega(D)+1} $. Therefore $\omega(G_0,A)\le m\omega(D)+1$ is finite.
Note that for every $U\in\mathcal A(G_0)\setminus\{1_G\}$ we have that $\sum\limits_{g\in G_0\cap\langle \alpha\rangle\tau}\mathsf v_g(U)\geq 2$. Suppose $A\mid_{\mathcal B(G_0)}U_1\bdot \hdots \bdot U_n$, but  $A$ does not divide any proper subproduct.  Then $n\leq \omega(G_0,A)$ and $U_1\bdot \ldots\bdot U_{n}\bdot A^{[-1]}=W_1\bdot \ldots \bdot W_m$, where $m\in \N$ and $W_1,\ldots, W_m\in \mathcal A(G_0)$. Since $2\le \sum_{g\in G_0\cap \langle \alpha\rangle\tau}\mathsf v_g(W_i)$ for every $i\in [1,m]$ and $M\ge \sum_{g\in G_0\cap \langle \alpha\rangle\tau}\mathsf v_g(U_j)$ for every $j\in [1,n]$, we obtain that $m\le \frac{nM}{2}\le \frac{M\cdot \omega(G_0,A)}{2}$. The definition of 
$\tau(G_0,A)$ implies that $\tau(G_0, A)\le \frac{M\cdot \omega(G_0,A)}{2}$ is finite.

Note that
\begin{align*}
\rho(G_0)&=\sup\Big\{\frac{\max(\mathsf L(S))}{\min(\mathsf L(S))}\colon S\in\mathcal B(G_0\setminus\{1_G\})\setminus \{1_{\mathcal B(G_0)}\}\Big\}\\
&\leq \sup\Big\{\frac{M/2\cdot\min(\mathsf L(S))}{\min(\mathsf L(S))}\colon S\in\mathcal B(G_0\setminus\{1_G\})\setminus \{1_{\mathcal B(G_0)}\}\Big\}\\
&=\frac{M}{2}.
\end{align*}
Let $A\in\mathcal A(G_0)$ such that $\sum\limits_{g\in G_0\cap\langle\alpha\rangle\tau}\mathsf v_g(A)= M$. Then $A=h_1\bdot\ldots\bdot h_r\bdot g_1\bdot\ldots\bdot g_M$, where $h_i\in G_0\cap \langle \alpha\rangle$ and $g_i\in G_0\cap \langle\alpha\rangle\tau$. It follows that $A^{[2]}=(h_1^{[2]}\bdot\ldots\bdot h_r^{[2]}\bdot g_1^{[2]})\bdot g_2^{[2]}\bdot\ldots\bdot g_M^{[2]}$, whence $\rho(\mathsf L(A^{[2]}))\geq\frac{M}{2}$, thus, using the above inequality, $\rho(\mathsf L(A^{[2]}))=\frac{M}{2}$ and $\rho(G_0)$ is accepted.
%
%
%
\end{proof}

\smallskip
\begin{proof}[Proof of Theorem \ref{5.1}.5]
 Suppose  $\mathsf c (G_0) < \infty$. Then the set of distances is finite by Equation \eqref{basic}. If $\rho_k (G_0) = \infty$ for some $k \in \N$, then the sets $\mathcal U_k (G_0)$ have the asserted form by \cite[Theorem 4.2]{Ga-Ge09b}. If $\rho_k (G_0) < \infty$ for all $k \ge 2$, then $\rho (G_0)$ is accepted by Theorem \ref{5.1}.4, whence the claim for $\mathcal U_k (G_0)$ holds by \cite[Lemma 3.3 and Theorem 4.2]{Ga-Ge09b} (see also \cite[Theorem 1.2]{Tr19a}).

It remains to show the finiteness of the catenary degree. Let $D$ be the finitely generated monoid defined in Equation \eqref{equation:D}. If $G_0\subset \langle \alpha\rangle$, then by Theorem \ref{5.1}.3 and Equation \eqref{basic} we obtain $\mathsf c(G_0)$ is finite. Suppose $G_0\cap \langle\alpha\rangle\tau\neq\emptyset$ and set $$N=\max\{|A|\colon A\in \mathcal A(D)\}.$$ Let $G_1=\langle \alpha\rangle\cap G_0$. 
	We first show the following two claims.
	
	\medskip
	\noindent{\bf Claim A. }{\it  Let $S_1,\ldots, S_{|G_1|+1}\in \mathcal B(G_1)$.
		Then there exists a nonempty subset $I\subset [1,|G_1|+1]$ such that $\mathsf v_g(\prod_{i\in I}S_i)$ is even for all $g\in G_1$.}
	
	\begin{proof}[Proof of Claim A]
		Suppose $(e_1,\ldots,e_{|G_1|})$ is a basis of the elementary $2$-group $C_2^{|G_1|}$ and let $\eta\colon \mathcal F(G_1)\rightarrow C_{2}^{|G_1|}$ be a monoid homomorphism defined by $\eta(g_i)=e_i$, where $\{g_1,\ldots, g_{|G_1|}\}=G_1$.  Therefore for every sequence $S$ over $G_1$, we know $\mathsf v_g(S)$ is even for all $g\in G_1$ if and only if $\sigma(\eta(S))=0$.  
		
		 Since $\sigma(\eta(S_1))\bdot\ldots\bdot \sigma(\eta(S_{|G_1|+1}))$ is a sequence over $C_2^{|G_1|}$ of length $|G_1|+1=\mathsf D(C_2^{|G_1|})$, there exists a nonempty subset $I\subset [1, |G_1|+1]$ such that the sequence $\prod_{i\in I}\sigma(\eta(S_i))$  is a product-one sequence, whence $\mathsf v_g(\prod_{i\in I}S_i)$ is even for all $g\in G_1$.	
		\qedhere[End of proof of Claim A.]
	\end{proof}
	
	\medskip
	\noindent{\bf Claim B. }{\it  Let $A\in \mathcal A(G_0)$ with $\supp(A)\cap \langle\alpha\rangle\tau\neq\emptyset$ and let $T$ be a sequence over $G_1$.
		\begin{itemize}
			\item[(i)] For every $T_0\in \mathcal F(G_0\cap \langle\alpha\rangle)$, we have that $T_0^{[2]}\bdot A\in \mathcal B(G_0)$ and $\max\mathsf L(T_0^{[2]}\bdot A)\le 2|T_0|+N+|G_1|$.

			\item[(ii)] Suppose $A\in \mathcal A(D)$. Let $T_0\in \mathcal F(G_1)$   be a subsequence of $T$ such that
			$(T\bdot T_0^{[-1]})^{[2]}\in \mathcal B(G_0)$ with minimal length. Then $\max\mathsf L(T_0^{[2]}\bdot A)\le |G_1|+N$.
		\end{itemize}
	}
	
	\begin{proof}[Proof of Claim B]
		(i)	Let $T_0\in \mathcal F(G_0\cap \langle\alpha\rangle)$.  Since $A\in \mathcal A(G_0)\subset\mathcal B(G_0)$, there is an even number of terms from $G_0\cap\langle\alpha\rangle\tau$ contained in $A$, so distributing $T_0^{[2]}$ appropriately gives $T_0^{[2]}\bdot A\in\mathcal B(G_0)$. The definition of $D$ implies that there exist a sequence $T_1\in \mathcal F(G_0\cap \langle\alpha\rangle)$  and an atom $A_1\in \mathcal A(D)$ such that $A=T_1^{[2]}\bdot A_1$.
		Suppose $$T_0^{[2]}\bdot A=T_0^{[2]}\bdot T_1^{[2]}\bdot  A_1=W_1\bdot \ldots\bdot W_{r}\,,$$
		where $W_i\in \mathcal A(G_0)$ for all $i\in [1,r]$. Let $I\subset [1,r]$ be the maximal subset such that $\prod_{i\in I}W_i$ is a subsequence of $T_1^{[2]}$. Then for every $j\in [1,r]\setminus I$, we have that $W_j\bdot \prod_{i\in I}W_i\nmid_{\mathcal F(G_0)} T_1^{[2]}$, whence $r-|I|\le 2|T_0|+|A_1|\le 2|T_0|+N$.
		Assume to the contrary that $|I|\ge |G_1|+1$. Then {\bf Claim  A} implies that there exists  a nonempty subset $J\subset I$ such that $\prod_{i\in J}W_i=T_3^{[2]}$ for some $T_3\in\mathcal F(G_1)$. Note that $(T_1\bdot T_3^{[-1]})^{[2]}\bdot A_1$ is still  a product-one sequence, using the same argument as above for $T_0^{[2]}\bdot A\in\mathcal B(G_0)$, since we just cancel out terms of $T_1$. This is a contradiction to the fact that $A\in \mathcal A(G_0)$. Therefore $r\le 2|T_0|+N+|I|\le 2|T_0|+N+|G_1|$.

		(ii) Suppose $A\in \mathcal A(D)$. Let $T_0\in \mathcal F(G_1)$  be a subsequence of $T$ such that
		$(T\bdot T_0^{[-1]})^{[2]}\in \mathcal B(G_0)$ with minimal length.
		Suppose $$T_0^{[2]}\bdot A=W_1\bdot \ldots\bdot W_{r}\,,$$
		where $W_i\in \mathcal A(G_0)$ for all $i\in [1,r]$. After renumbering if necessary, we may assume there exists $t\in \N_0$ such  that \begin{equation*}
		W_1\bdot\ldots\bdot W_t\mid_{\mathcal F(G_0)} T_0^{[2]} \text{ and } W_j\nmid_{\mathcal F(G_0)}T_0^{[2]}\bdot (W_1\bdot\ldots\bdot W_t)^{[-1]} \text{ for all }j\in [t+1,r]\,,
		\end{equation*}
		whence $r-t\le |A|\le N$. If $t\ge |G_1|+1$, then {\bf Claim A} implies that there is a subset $I\subset [1,t]$ such that $\prod_{i\in I}W_i=T_1^{[2]}$ for some sequence $T_1\in \mathcal F(G_1)$. It follows that $T_0\cdot T_1^{[-1]}$ is a subsequence of $T$ such that $(T\bdot (T_0\bdot T_1^{[-1]})^{[-1]})^{[2]}=T_1^{[2]}\bdot (T\bdot T_0^{[-1]})^{[2]}=\prod_{i\in I}W_i\bdot (T\bdot T_0^{[-1]})^{[2]}\in \mathcal B(G_0)$, a contradiction to the minimality of $|T_0|$. Thus $t\le |G_1|$ and $r\le |G_1|+N$.
		\qedhere[End of proof of Claim B.]
	\end{proof}

	To prove the finiteness of $\mathsf c(G_0)$, we show that $$\mathsf c(G_0)\le \gamma= 4\mathsf t(D)+2+2|G_1|+2N\,.$$
	
	Assume to the contrary that there exists $B\in \mathcal B(G_0)$ with $\mathsf c_{\mathcal B(G_0)}(B)>\gamma$. Let $B\in \mathcal B(G_0)$ be the counter example such that $|B|$ is minimal and let $z_1=U_1\bdot\ldots\bdot U_k$, $z_2=V_1\bdot\ldots\bdot V_{\ell}$ be two factorizations of $B$ between which there is no $\gamma$-chain. Since $\mathcal A(G_1)\subset \mathcal A(D)$, we obtain that $\mathsf c(G_1)\leq\mathsf D(G_1)\leq N$, whence $\supp(B)\cap \langle\alpha\rangle\tau\neq\emptyset$. By symmetry, we may suppose $\supp(V_2\bdot\ldots\bdot V_{\ell})\cap \langle\alpha\rangle\tau\neq\emptyset$. Suppose $V_1=V_1'\bdot h_1^{[2]}\bdot\ldots\bdot h_r^{[2]}$, where $r\in \N_0$, $V_1'\in \mathcal A(G_0)\cap \mathcal A(D)$, and $h_1,\ldots, h_r\in G_1$.  Therefore for every subset $I\subset [1,r]$, we have that $V_1'\bdot \prod_{i\in I}h_i^{[2]}\in \mathcal A(G_0)$ and $V_1'\bdot \prod_{i\in I}h_i^{[2]}\Big|_{\mathcal B(G_0)} B$. To see this, note that $\supp(V_1')\subset \langle\alpha\rangle$ implies that $r=0$ and the statement is trivial. In the case $\supp(V_1')\cap \langle\alpha\rangle\tau\neq \emptyset$, we use the same argument as above to see that $V_1'\bdot \prod_{i\in I}h_i^{[2]}\in\mathcal B(G_0)$, namely distribute the $h_i$ appropriately. If $V_1'\bdot \prod_{i\in I}h_i^{[2]}\notin\mathcal A(G_0)$ we could distribute the remaining $h_j$ and would get a contradiction to $V_1\in\mathcal A(G_0)$. Also the second statement can be seen using the argument above to show that $(V_2\bdot\ldots\bdot V_{\ell})\bdot\prod_{j\in[1,r]\setminus I}h_j^{[2]}\in \mathcal B(G_0)$. We proceed by the following two claims.
	
	\medskip
	\noindent{\bf Claim C. }{\it There is a factorization $z_3=V_1'\bdot W_1\bdot \ldots \bdot W_t$ such that there is a $\gamma$-chain between $z_1$ and $z_3$.}
	
	\noindent{\bf Claim D. }{\it Let $I\subsetneq [1,r]$, $V_1^*=V_1'\bdot \prod_{i\in I}h_i^{[2]}$, and $j\in [1,r]\setminus I$. If $z'=V_1^*\bdot Y_1\bdot \ldots\bdot Y_x$ is a factorization of $B$, then there exists a factorization $z''=(V_1^*\bdot h_j^{[2]})\bdot Y_1'\bdot \ldots\bdot Y_{x'}'$ of $B$ such that there is a $\gamma$-chain between $z'$ and $z''$.}
	
	\medskip
	If {\bf Claim C} and {\bf Claim D} hold, then there exists a factorization $z_4=V_1\bdot X_1\bdot\ldots\bdot X_{t}$ of $B$ such that there is a $\gamma$-chain between $z_1$ and $z_4$. Since $|V_2\bdot\ldots \bdot V_{\ell}|=|X_1\bdot \ldots\bdot X_t|<|B|$, by the minimality of $|B|$ there is a $\gamma$-chain between $z_4$ and $z_2$, whence we are done.

	\bigskip
	
	\begin{proof}[Proof of Claims C and D]
		Let   $G_2=G_0\cap \langle \alpha\rangle\tau$.
		Let $U^*$ be an atom of $D$. If $U^*$ is also an atom of $\mathcal B(G_0)$, then let $U_0=1_{\mathcal B(G_0)}$ and otherwise let $U_0$ be an atom of $\mathcal B(G_0)$ such that $U^*\bdot U_0$ is also an atom of $\mathcal B(G_0)$.
		
		Suppose $B$ has two factorizations $z_7=U_0\bdot U_1\bdot\ldots \bdot U_{l}$ and $z_8=(U_0\bdot U^*)\bdot V_1\bdot\ldots\bdot V_{m}$ and suppose $\supp(U_1\bdot\ldots U_{l}\bdot (U^*)^{[-1]})\cap G_2\neq \emptyset$. It suffices to  show that there is a $\gamma$-chain between them. Then both statements follow by choosing the parameters for Claim C as $U^*=V_1'$, $U_0$ trivial, $U_1\bdot\ldots\bdot U_l=z_1$ and for Claim D as $U^*=h_j^{[2]}$ and $U_0=V_1'$.

		Note that $U^*$ divides $U_1\bdot\ldots\bdot U_{l}$ in $D$.  There is a subset $I\subset [1,l]$, say $I=[1, w]$, such that $|I|=w\le \omega(D)+1$, $U^*\t_D \prod_{i=1}^wU_i$, and $\supp((U^*)^{[-1]}\bdot\prod_{i=1}^wU_i)\cap G_2\neq\emptyset$. Therefore $(U^*)^{[-1]}\bdot\prod_{i=1}^wU_i\in \mathcal B(G_0)$ and $U^*\t_{\mathcal B(G_0)} \prod_{i=1}^wU_i$. For each $i\in [1,w]$, suppose $U_i=U_i'\bdot g_{i,1}^{[2]}\bdot \ldots\bdot g_{i,t_i}^{[2]}$, where $U_i'\in \mathcal A(D)$, $t_i\in \N_0$, and $g_{i,1},\ldots, g_{i,t_i}\in G_1$. Let $J\subset [1,w]$ and $I_i\subset [1,t_i]$ for each $i\in [1,w]$ be subsets such that
		$$U^*\t_D X=\prod_{j\in J}U_j'\bdot \prod_{i=1}^w\big(\prod_{j_i\in I_i}g_{i,j_i}^{[2]}\big)\,,$$
		but $U^*$ does not divide any proper subproduct. Thus $(U^*)^{[-1]}\bdot X$ has a factorization in $D$ of length $\le \tau(D)$. Suppose $$X\bdot \prod_{j\in [1,w]\setminus J}U_j'=U^*\bdot X_1\bdot \ldots\bdot X_{\tau}\,,$$ where $X_i\in \mathcal A(D)$ and $\tau\le \tau(D)+w-|J|\le 2\mathsf t(D)+1$.
		
		Let $T= \prod_{i=1}^w(\prod_{j_i\in [1,t_i]\setminus I_i }g_{i,j_i})$. Then  $$B=U_0\bdot T^{[2]}\bdot X\bdot \prod_{j\in [1,w]\setminus J}U_j'\bdot \prod_{j=w+1}^lU_j\,.$$ Let  $T_0=\prod_{i=1}^w(\prod_{j_i\in I_i'}g_{i,j_i})$ be a subsequence of $T$ such that
		$(T\bdot T_0^{[-1]})^{[2]}\in \mathcal B(G_0)$ with minimal length, where $I_i'\subset [1,t_i]\setminus I_i$. Note that $T_0$ can be trivial or equal to $T$. Suppose $$(T\bdot T_0^{[-1]})^{[2]}=W_1\bdot \ldots \bdot W_y\,,$$ where $W_i\in \mathcal A(G_1)$. Let $U_i''=U_i'\bdot \prod_{j\in I_i\cup I_i'}g_{i,j}^{[2]}$. Then $U_i''\in \mathcal A(G_0)$ (using the same argument as above Claim C) and  $B$ has a factorization
		$$z_{10}=U_0\bdot U_1''\bdot \ldots\bdot U_w''\bdot W_1\bdot\ldots\bdot W_{y}\bdot U_{w+1}\bdot\ldots\bdot U_l\ (\text{ note that }U_0\in \mathcal A(G_0)\cup\{1_{\mathcal B(G_0)}\}) \,.$$
		
		After renumbering if necessary, we may assume there exists $\tau_0\in \N$ such that $X_i\in \mathcal A(G_0)$ for all $i\in [1,\tau_0]$, $X_i\not\in \mathcal A(G_0)$ for all $i\in [\tau_0+1, \tau]$, and $\supp(X_1)\cap G_2\neq\emptyset$ (note that $\supp((U^*)^{[-1]}\bdot\prod_{i=1}^wU_i)\cap G_2\neq\emptyset$).
		Then $|X_i|=2$ for every $i\in [\tau_0+1, \tau]$, $X_1\bdot T_0^{[2]}\in \mathcal B(G_0)$, and by {\bf Claim B}, $X_1\bdot T_0^{[2]}$ has a factorization $X_1\bdot T_0^{[2]}=Y_1\bdot\ldots\bdot Y_s$  of length $s\le |G_1|+N$. We may suppose $\supp(Y_1)\cap G_2\neq \emptyset$. Then by {\bf Claim B} again, $X'=Y_1\bdot X_{\tau_0+1}\bdot\ldots\bdot X_{\tau}\in \mathcal B(G_0)$ has a factorization of length $\le |X_{\tau_0+1}\bdot\ldots\bdot X_{\tau}|+N+|G_1|= 2(\tau-\tau_0)+N+|G_1|$. It follows by $$B\bdot (T\bdot T_0^{[-1]})^{[-2]}\bdot (U_{w+1}\bdot\ldots\bdot U_l)^{[-1]}\bdot (U^*\bdot U_0)^{[-1]}= T_0^2\bdot X_1\bdot\ldots\bdot X_{\tau}= X_2\bdot\ldots\bdot X_{\tau_0}\bdot Y_2\bdot\ldots\bdot Y_s\bdot X'$$ that $B\bdot (T\bdot T_0^{[-1]})^{[-2]}\bdot (U_{w+1}\bdot\ldots\bdot U_l)^{[-1]}\bdot (U^*\bdot U_0)^{[-1]}$ has a factorization of length $\le (\tau_0-1)+|G_1|+N+2(\tau-\tau_0)+N+|G_1|\le \gamma$.
		Therefore $B$ has a factorization $$z_3=(U^*\bdot U_0)\bdot W_1'\bdot\ldots\bdot W_x'\bdot W_1\bdot\ldots\bdot W_y\bdot U_{w+1}\bdot\ldots\bdot U_l$$ with $x\le \gamma-1$, whence $\mathsf d(z_{3},z_{10})\le \max\{\gamma,\omega(D)+1\}$, where $W_1',\ldots, W_x'\in \mathcal A(G_0)$. Note that by choice and Equation \ref{4.1} we get that $w\leq \omega(D)+1\leq \mathsf t(D)+1$.

%
		By {\bf Claim A}, there exist $a\in \N$ and sequences $T_1',\ldots, T_{a}'$ over $G_1$  such that $[1,y]=K_1\uplus \ldots \uplus K_a$, where $K_1,\ldots,K_a$ are disjoint subsets of $[1,y]$ with $|K_i|\le |G_1|+1$ for every $i\in [1,a]$, and 
		 $(T_{i}')^{[2]}=\prod_{j\in K_i}W_j$.
		  Furthermore, there exist $J_i^{(j)}\subset [1,t_i]\setminus (I_i\cup I_i')$  for all $i\in [1,w]$ and all $j\in [1,a]$ such that $T_k'=\prod_{i=1}^w(\prod_{j\in J_i^{(k)}}g_{i,j})$ for all $k\in [1,a]$. Set $$U_i^{(k)}=U_i''\bdot\prod_{j\in J_i^{(1)}\cup\ldots \cup J_i^{(k)}}g_{i,j}^{[2]}$$ and hence $U_i^{(k)}\in \mathcal A(G_0)$ for all $k\in [1,a]$ and $i\in [1,w]$. Then $$z_{10+k}=U_0\bdot U_1^{(k)}\bdot\ldots U_w^{(k)}\bdot \prod_{j\in [1,y]\setminus (K_1\cup\ldots \cup K_k)}W_j\bdot \prod_{i=w+1}^lU_i$$ is a factorization of $B$, where $k\in [1,a]$. Note that $z_{10+a}=z_7$ and 
		  $\mathsf d(z_{10+i},z_{10+i+1})\le w+|G_1|+1$ for every $i\in [0, a-1]$. Then there is a $(w+ |G_1|+1)$-chain between $z_7$ and $z_{10}$.

		By the minimality of the choice of $|B|$, we obtain there is a $\gamma$-chain between $z_3$ and $z_8$, whence there is a $\gamma$-chain between $z_7$ and $z_8$.
				
		\qedhere[End of proof of Claims C and D.]
	\end{proof}

Thus the proof of Theorem \ref{5.1}.5 is complete.
\end{proof}

We continue with three examples.

\smallskip
\begin{example}\label{infdih}
If $G_0=\{\alpha, \alpha^{-1}, \tau\}$, then $\mathcal B(G_0)$ is a root closed $v$-noetherian G-monoid, but neither a C-monoid nor locally tame, and $\mathsf D(G_0)=\omega(G_0)=\infty$.
\end{example}
\begin{proof}
Since $G_0$ is finite,	
it follows by Theorem \ref{5.1} that  $\mathcal B(G_0)$ is a $v$-noetherian G-monoid, $\mathsf D(G_0)=\omega(G_0)=\infty$, and not locally tame, whence not a C-monoid by \cite[Theorem 3.3.4.3]{Ge-HK06a}.

 To show $\mathcal B(G_0)$ is root closed, it suffices to prove $\widetilde{\mathcal B(G_0)}\subset \mathcal B(G_0)$.  Let $\frac{S_1}{S_2}\in \widetilde{\mathcal B(G_0)}\subset \mathcal F(G_0)$ with $S_1=\alpha^{[n_1]}\bdot (\alpha^{-1})^{[m_1]}\bdot\tau^{[l_1]}, S_2=\alpha^{[n_2]}\bdot (\alpha^{-1})^{[m_2]}\bdot\tau^{[l_2]}\in\mathcal B(G_0)$. Then $n_1\ge n_2 $, $m_1\ge m_2$, $l_1\ge l_2$, and $l_1,l_2, |n_1-m_1|, |n_2-m_2|$ are even. Therefore $(n_1-n_2)-(m_1-m_2), l_1-l_2$ are even and
$\frac{S_1}{S_2}=\alpha^{[n_1-n_2]}\bdot (\alpha^{-1})^{[m_1-m_2]}\bdot\tau^{[l_1-l_2]}$. If $l_1-l_2\neq 0$, then $\frac{S_1}{S_2}\in \mathcal B(G_0)$. Suppose $l_1=l_2$. Since $\frac{S_1}{S_2}\in \widetilde{\mathcal B(G_0)}$, there exists $N\in \N$ such that $(\frac{S_1}{S_2})^{[N]}=\alpha^{[N(n_1-n_2)]}\bdot (\alpha^{-1})^{[N(m_1-m_2)]}\in \mathcal B(G_0)$, which implies that  $(\frac{S_1}{S_2})=\alpha^{[n_1-n_2]}\bdot (\alpha^{-1})^{[m_1-m_2]}\in \mathcal B(G_0)$ and we are done.
\end{proof}

\smallskip
\begin{example} \label{taualpha}
If $G_0=\{\alpha, \tau\}$, then  $\mathcal B(G_0)$ is a root closed $v$-noetherian G-monoid, a half-factorial non-finitely generated C-monoid whence locally tame, but not completely integrally closed, and $\mathsf D(G_0)= \omega (G_0) = \infty$.
\end{example}

\begin{proof}
It is clear that $\mathcal A(G_0)=\{\alpha^{[2n]}\bdot\tau^{[2]} \colon n\in\N_0\}$. Thus 	
	$\mathsf D(G_0)=\omega(G_0)=\infty$, $\mathcal B(G_0)$ is not finitely generated, and  every sequence $S\in\mathcal B(G_0)\setminus\{1_{\mathcal B(G_0)}\}$ can be written as $S=\alpha^{[2n]}\bdot \tau^{[2m]}$ for some $n\in\N_0, m\in\N$, whence $\mathsf L(S)=\{\frac{\mathsf v_{\tau}(S)}{2}\}$.
Therefore $\mathcal B(G_0)$ is half-factorial and  $$\mathcal C(\mathcal B(G_0),\mathcal F(G_0))=\{[1_{\mathcal F(G_0)}],[\alpha^{[2]}],[\alpha],[\tau ^{[2]}],[\tau],[\alpha\bdot \tau^{[2]}],[\alpha\bdot \tau]\}\,,$$ whence $\mathcal B(G_0)$ is a C-monoid and therefore locally tame by \cite[Theorem 3.3.4.3]{Ge-HK06a}. Example \ref{infdih} implies that  $\mathcal B(G_0)$ is root closed.
Note that $\alpha ^{[2]}=\frac{\alpha ^{[2]}\bdot \tau ^{[2]}}{\tau ^{[2]}}\in\mathsf q(\mathcal B(G_0))$ and $\tau ^{[2]}\bdot (\alpha ^{[2]})^{[n]}\in\mathcal B(G_0)$ for all $n\in \N$, but $\alpha ^{[2]}\notin\mathcal B(G_0)$. We obtain $\mathcal B(G_0)$ is not completely integrally closed.
\end{proof}

In Proposition \ref{finitary} we  proved that monoids of product-one sequences over finite subsets of groups are finitary. Next we show that there are finitary monoids of product-one sequences over infinite subsets of groups.

\smallskip
\begin{example} \label{infinitefinitary}
If $G_0 =\{\tau, \alpha^{m} \colon m\in\N\}$, then $\mathcal B(G_0)$ is finitary, but neither a G-monoid nor seminormal, and $\mathsf D(G_0)= \omega (G_0) = \infty$.
\end{example}

\begin{proof}

If $G_1 = \{\alpha, \tau\}$, then $\mathcal B (G_1) \subset \mathcal B (G_0)$ is a divisor closed submonoid, whence Example \ref{taualpha} implies that $\infty = \mathsf D (G_1) \le \mathsf D (G_0)$ and $\infty = \omega (G_1)  \le \omega (G_0)$.

In the notation of Proposition \ref{finitary}, let $n=1$ and $A_1=\tau^{[2]}$. Then $\{A_1\}$ satisfies the conditions in Proposition \ref{finitary} and hence $\mathcal B(G_0)$ is finitary. By Theorem \ref{3.11}, $\mathcal B(G_0)$ is not a G-monoid. To see that $\mathcal B(G_0)$ is not seminormal, we set $S_1=(\alpha^2)^{[3]}\bdot \alpha^6\bdot \tau^{[4]}$ and $S_2=(\alpha^2)^{[2]}\bdot\tau^{[2]}$. Then $S_1,S_2\in\mathcal B(G_0)$ and $S=\alpha^2\bdot\alpha^6\bdot \tau^{[2]}=\frac{S_1}{S_2}\in\mathsf q(\mathcal B(G_0))\setminus \mathcal B(G_0)$, but $S^{[2]},S^{[3]}\in\mathcal B(G_0)$.
\end{proof}

Finally, we give the proof of Theorem \ref{5.2}.

\smallskip
\begin{proof}[Proof of Theorem \ref{5.2}]
Let $G$ be an infinite dihedral group.

1. This follows from Theorem \ref{3.14}

2.  Theorem \ref{main2} implies that $\mathcal L (G)$ has the given form, whence $\Delta (G)$ and $\mathcal U_k (G)$ are as asserted.

3. The monoid $\mathcal B (G)$ is an \FF-monoid by Proposition \ref{4.1}. Since  $\Delta (G) = \N$, we obtain $\mathsf c (G) = \omega (G) = \infty$ by Equation \eqref{basic}. Example \ref{infdih} shows that $\mathcal B (G)$ has divisor-closed submonoids, that are not locally tame, whence $\mathcal B (G)$ is not locally tame.
\end{proof}

\section*{Acknowledgments}

The authors have discussed this work with A. Geroldinger in innumerable conversations. We want to thank him for his suggestions and comments. We are also grateful to the anonymous referee for his careful reading; he found a mistake in a proof and helped a lot to get the manuscript into a readable shape.

\bigskip


\providecommand{\bysame}{\leavevmode\hbox to3em{\hrulefill}\thinspace}
\providecommand{\MR}{\relax\ifhmode\unskip\space\fi MR }
\providecommand{\MRhref}[2]{%
  \href{http://www.ams.org/mathscinet-getitem?mr=#1}{#2}
}
\providecommand{\href}[2]{#2}

\end{document}